\documentclass[a4paper,reqno,11pt,oneside]{amsart}
\usepackage{amsmath,geometry}
\usepackage{graphicx}
\usepackage{mathrsfs}
\usepackage{amssymb,amsmath, amsfonts, amsthm,color}
\usepackage{latexsym}
\usepackage{color} 
\usepackage[dvips]{epsfig}
\usepackage[colorlinks]{hyperref}
\usepackage{comment}
\hypersetup{
    colorlinks,
    citecolor=blue,
    linkcolor=blue,
    urlcolor=black
}
\usepackage{tikz}
\usepackage{pgfplots}
\usetikzlibrary{patterns}
\usepackage{bm}
\newtheorem{theorem}{Theorem}[section]
\newtheorem{lemma}[theorem]{Lemma}

\newtheorem{proposition}[theorem]{Proposition}

\newtheorem{remark}[theorem]{Remark}

\numberwithin{equation}{section}

\newcommand{\e}{\varepsilon}

\newcommand{\R}{\mathbb{R}}

\newcommand{\de}{\partial}
\newcommand{\weakto}{\rightharpoonup}

\newcommand{\hg}{g_{\mathbb{H}}}

\DeclareMathOperator{\jac}{Jac}

\renewcommand{\d }{\delta }

\newcommand{\Sph}{{\mathbb{S}}}

\newcommand{\D}{{\mathfrak{D}}}

\newcommand{\dsh}{{2^\sharp}}
\newcommand{\dst}{{2^*}}
\newcommand{\hsp}{\hspace{0.2cm}}
\newcommand{\Ud}{{\mathcal{U}_{\delta, p}}}
\newcommand{\Vd}{\mathcal{V}_{\delta, p}}

\newcommand{\E}{\mathcal{E}}

\newcommand{\W}{\mathcal{W}}
\newcommand{\iM}{i^*_{M}}
\newcommand{\idM}{i^*_{\partial M}}
\newcommand{\sgg}{\mathfrak g}
\newcommand{\sff}{\mathfrak f}
\newcommand{\Z}{\mathcal Z}
\newcommand*{\scal}[1]{\left\langle #1 \right\rangle}

\newcommand*{\abs}[1]{\left\vert #1\right\vert}
\newcommand*{\norm}[1]{\left\Vert #1\right\Vert}
\newcommand*{\bigo}[1]{\mathcal O\left( #1 \right)}

\newcommand{\beq }{\begin{equation}}
	\newcommand{\eeq }{\end{equation}}

\title[\textit{Clustering} phenomena in low dimensions for a boundary Yamabe problem]{\textit{Clustering} phenomena  in low dimensions for a boundary Yamabe problem}

\author{Sergio Cruz-Blázquez}
\address{Sergio Cruz Blázquez, Dipartimento di Matematica, Università degli Studi di Bari, Via E. Orabona, 4, 70125 Bari (Italy)}
\email{sergio.cruz@uniba.it}

\author{Angela Pistoia}
\address{Angela Pistoia, Dipartimento SBAI,  Sapienza Università di Roma, Building RM 002, Via Antonio Scarpa 16
	00161 Roma (Italy)}
\email{angela.pistoia@uniroma1.it}

\author{Giusi Vaira}
\address{Giusi Vaira, Dipartimento di Matematica, Università degli Studi di Bari, Via E. Orabona, 4, 70125 Bari (Italy)}
\email{giusi.vaira@uniba.it}

\begin{document}
\begin{abstract}
We consider the classical  geometric problem of prescribing the scalar and boundary mean curvatures via conformal deformation of the metric on a $n-$dimensional compact Riemannian manifold. We deal with the case of negative scalar curvature and positive boundary mean curvature. It is known that if $n=3$ all  the blow-up points are isolated and simple. In this work we prove that, for a linear perturbation, this is not true anymore in low dimensions $4\leq n\leq 7$. 
In particular, 
we construct a solution with a clustering blow-up boundary  point (i.e.  non-isolated),  which is non-umbilic  and is a local minimizer of the norm of the trace-free  second fundamental form of the boundary.
\end{abstract}
\date\today
 \thanks{ A. Pistoia has been partially supported by  GNAMPA, Italy as part of INdAM. 
 G. Vaira and S. Cruz have been partially supported by INdAM – GNAMPA Project 2022 “Fenomeni di blow-up per equazioni non lineari”, E55F22000270001S and by PRIN 2017JPCAPN.}

\maketitle
	\section{Introduction}	
 Given a compact  Riemannian manifold $(M,  g)$  of dimension $n\geq3$ with boundary $\partial M,$  a widely studied geometric problem is the following one: 
{\em  given two smooth functions $K$ and $H$ find a metric conformal to $g$ whose scalar curvature is $K$ and boundary mean curvature is $H$.}
\\ 
As it is well known, the geometric problem can be rephrased into the following one:
{\em  given two smooth functions $K$ and $H$ find a positive solution to the PDE }
 \begin{equation}\label{pb0}
\left\{\begin{aligned}&-\frac{4(n-1)}{n-2}\Delta_g u +\mathcal S_g u =K u^{\frac{n+2}{n-2}}\quad&\mbox{in}\,\, M\\
&\frac{2}{n-2}\frac{\partial u}{\partial \nu}+h_g u = H u^{\frac{n}{n-2}}\quad &\mbox{on}\,\, \partial M.\end{aligned}\right.\end{equation}
Here  $\Delta_g$ is the Laplace-Beltrami operator, $\mathcal S_g$ is  the scalar curvature and  $h_g$ the boundary mean curvature associated to the metric $g$
 and $\nu$ is the outward unit normal vector to $\partial M$.
The metric $\tilde g=u^{4\over n-2}g$ is conformal to $g$ and its 
 scalar and boundary mean  curvatures are nothing but   $K$ and $H$, respectively.\\
 
The study began with the work of Cherrier \cite{C} who gave a first criterion for the existence and regularity of solution of \eqref{pb0}. Successively, 
Escobar in a series of papers  \cite{E1,E2,E3} found a solution to \eqref{pb0} when either $K=0$ (i.e. scalar flat metric)  and $H$ constant or
$H=0$ (i.e. minimal boundary) and $K$ is constant. The proof strongly relies on the dimensions of the manifold, on the properties of the boundary (e.g. being or not umbilic)   and on vanishing properties of the Weyl tensor (e.g. being identically zero or not on the boundary or on the whole manifold). Important contributions in this framework are due to the work of Marques in \cite{M1,M2}, Almaraz \cite{A}, Brendle \& Chen \cite{BC} and Mayer \& Ndiaye \cite{MN}.  The case when $K>0$ and $H$  is an arbitrary constant, has been successfully treated by    Han \& Li in \cite{HY1,HY2} and  Chen, Ruan \& Sun \cite{crs}.
\\

There are a few results concerning the  general case in which $K$ and $H$ are functions (not necessarily constants) and all of them have been obtained for special manifolds (e.g. typically the unit ball or the half sphere). In particular,  we refer to the works of 
Ben Ayed,  El Mehdi \& Ould Ahmedou \cite{BEO1, BEO2} and Li \cite{L} when $H=0$ and   Abdelhedi,  Chtioui \&  Ould Ahmedou \cite{ACA},
 Chang,  Xu  \& Yang
\cite{CXY}, Djadli,  Malchiodi \& Ould Ahmedou \cite{DMA2} and Xu \& Zhang \cite{XZ} when $K=0$.
The case when both $K$ and $H$ do not vanish, has been studied by
Ambrosetti,  Li  \& Malchiodi  \cite{AYM} in a perturbative setting  on the $n-$dimensional unit ball   and by   Djadli,  Malchiodi \& Ould Ahmedou \cite{DMA1} on the three-dimensional half sphere.
Finally, we quote the result of  Chen,  Ho \& Sun \cite{CHS} where they found a solution to \eqref{pb0} when $H$ and $K$ are negative functions
provide the manifold has  a  
boundary of negative Yamabe invariant.
\\

Recently,  Cruz-Bl\'azquez, Malchiodi and Ruiz \cite{CMR} considered  a manifold whose scalar curvature $\mathcal S_g\leq 0$ and the  case $K$ negative and $H$ of arbitrary sign. They introduce
  the {\textit{scaling invariant}} quantity 
\begin{equation}\label{dn}\mathfrak D_n(p)=\sqrt{n(n-1)}\frac{H(p)}{\sqrt{|K(p)|}},\ p\in\partial M\end{equation}
and established the existence of a solution to \eqref{pb0} whenever  $\mathfrak D_n<1$ along the whole boundary. On the other hand, if  $\mathfrak D_n>1
$ at some boundary points they got a solution only in a three dimensional manifold, for a generic choice of $K$ and $H$.
Let us describe more carefully their result.
First of all, via the conformal change of metric due to Escobar \cite{E2}, one can assume that the mean curvature $h_g = 0$ {and $\mathcal S_g$ has constant sign} ({this will be also assumed understood in the rest of our paper}), {so} problem \eqref{pb0} reads as
\begin{equation}\label{pb0p}
\left\{\begin{aligned}&-\frac{4(n-1)}{n-2}\Delta_g u +\mathcal S_g u =K u^{\frac{n+2}{n-2}}\quad&\mbox{in}\,\, M\\
&\frac{2}{n-2}\frac{\partial u}{\partial \nu}= H u^{\frac{n}{n-2}}\quad &\mbox{on}\,\, \partial M.\end{aligned}\right.\end{equation}
Problem \eqref{pb0p} is variational in nature, i.e. the solutions of \eqref{pb0p} are critical points of the energy functional defined on $H^1(M)$
$$
J(u)=\frac{2(n-1)}{n-2}\int_M |\nabla_g u|^2+\frac 12 \int_M \mathcal S_g u^2-\frac{1}{\dst}\int_M K (u^+)^{\dst}-(n-2)\int_{\partial M} H (u^+)^{\dsh}$$
where $\dst= \frac{2n}{n-2}$ and $\dsh=\frac{2(n-1)}{n-2}$ are the critical Sobolev exponent for $M$ and the critical trace embedding exponent for $\partial M$, respectively.
 In \cite{CMR} the authors show that if $\mathcal S_g\leq0 $ and  $\mathfrak D_n <1$ along the whole boundary, the functional becomes coercive and they found a global minimizer. On the other hand, if there exists $p\in\partial M$ such that $\mathfrak D_n(p)>1$, they construct a sequence of functions $u_i$ such that the energy $J(u_i)\to -\infty$ and the minimum point does not exist anymore. However, on a $3-$dimensional manifold they recover  the existence of a positive solution by using a mountain pass type argument. 
Their proof relies on a careful blow-up analysis: first they show that the blow-up phenomena occurs  at boundary points $p$ with $\mathfrak D_n(p)\geq 1$, {with different behaviours depending on whether $\D_n(p) = 1$ or $\D_n(p) > 1$. To deal with the loss of compactness at points with $\D_n(p) > 1$, where \textit{bubbling} of solutions occurs, it is shown that in dimension three all the blow-up points are isolated and simple (see also \cite{DMA1}). As a consequence, the number of blow-up points is finite  and the blow-up is excluded via integral estimates that hold true when $\mathcal S_g\leq 0$.}
{In that regard, $n=3$ is the maximal dimension for which one can prove that the blow-up points with $\mathfrak D_n>1$ are isolated and simple for generic choices of $K$ and $H$. In the closed case such a property is assured up to dimension four (see \cite{Li-CPAM96}) but, as observed in \cite{DMA1}, the presence of the boundary produces a stronger interaction of the \textit{bubbling} solutions with the function $K$.}\\

 Therefore, a natural question arises.
 \begin{itemize}
 \item[(Q)] {\em In higher dimensions $n\geq 4$, are the blow-up points still isolated and simple?}
 \end{itemize}
 In the present paper we  give a {partial} negative answer. \\
Let us consider the {linearly perturbed} problem
 \begin{equation}\label{pb}
\left\{\begin{aligned}&-\frac{4(n-1)}{n-2}\Delta_g u +\mathcal S_g u =K u^{\frac{n+2}{n-2}}\quad&\mbox{in}\,\, M\\
&\frac{2}{n-2}\frac{\partial u}{\partial \nu}+\varepsilon u = H u^{\frac{n}{n-2}}\quad &\mbox{on}\,\, \partial M\end{aligned}\right.\end{equation} 
where $\varepsilon$ is a small and positive parameter. Let $\pi$ be  the second fundamental form of $\partial M.$
{Our main result reads as follows}
\begin{theorem}\label{main}
Assume
\begin{itemize}
\item[(i)] $4\leq n\leq 7$ {and $\mathcal S_g>0$},
\item[(ii)] $H>0$ and $K<0$ are constant functions {such that $\mathfrak D_n>$1}, 
\item[(iii)] $p\in\partial M$  is non-umbilic (i.e. $ \pi(p) \not=0$) and non-degenerate  minimum point  of $\|\pi(\cdot)\|^2 $.   
\end{itemize}  
Then $p$ is a  ``clustering'' blow-up point, i.e.  for any $k\in\mathbb N$, there exist $p_\varepsilon^j\in \partial M$ for $j=1, \ldots, k$ and $\varepsilon_k>0$ such that for all $\varepsilon\in (0, \varepsilon_k)$ the problem \eqref{pb} has a solution $u_\varepsilon $ with $k$ positive peaks at $p_\varepsilon^j$ and $p_\varepsilon^j \to p$ as $\varepsilon\to 0$.
\end{theorem}

\begin{remark} We remind that  a point 
$p\in\partial M$ is non-umbilic if   the trace-free part of the second fundamental form of $\partial M$ does not   vanish at $p.$.  Since $h_g=0$, the
  tensor $T_{ij}=h_{ij}-h_g g_{ij}$ reduces to the second fundamental form $\pi$ whose components are $h_{ij}$
and so  $p$ is non-umbilic if  $\|\pi(p)\|>0.$\\
We believe that the non-degeneracy assumption is satisfied
   for  generic Riemannian metrics  (this could be proved using transversality tools as  in \cite{gm,MP1,MP2}).
 \end{remark}

\bigskip 
 
 The main ingredients of our construction are the so-called {\em bubbles}, i.e. the solutions of the problem 
\begin{equation}\label{limpb}\left\{
\begin{aligned}&-c_n\Delta u=K u^{\frac{n+2}{n-2}}\quad &\mbox{in}\,\, &\mathbb R^n_+\\ &\frac{2}{n-2}\frac{\partial u}{\partial\nu}=Hu^{\frac{n}{n-2}}\quad&\mbox{on}\,\, &\partial\mathbb R^n_+\end{aligned}\right.\end{equation}
   where  $c_n:=\frac{4(n-1)}{n-2}$ and $\mathfrak D_n :=\sqrt{n(n-1)}\frac{H }{\sqrt{|K|}}>1$ ($\nu$ is the exterior normal vector to $\partial\mathbb R^n_+$). Solutions to \eqref{limpb} are completely classified in \cite{cfs} (see also \cite{YZ}) . These are given by
\begin{equation}\label{bubble}
U_{\delta, y}(x):= \frac1{\delta^{n-2\over2}}U\left(x-y\over\delta\right),\ U(x):=
\frac{\alpha_n}{|K|^{\frac{n-2}{4}}}\frac{1}{\left(|\tilde x|^2+(x_n+\mathfrak D_n)^2-1\right)^{\frac{n-2}{2}}}
\end{equation} 
where $\alpha_n:=\left(4n(n-1)\right)^{\frac{n-2}{4}}$, $x=(\tilde x,x_n)$, $y=(\tilde y,0)$ and $\delta >0$.
  The solutions we are looking for are the sum of $k$ positive bubbles which concentrate at the same boundary point $p$ with the same speeds, i.e. in   local coordinates  (see \eqref{bb} and \eqref{ue}) around $p$
  \begin{equation}\label{ans1}
u_\varepsilon(x)\sim  \sum_{j=1}^k  {\frac{1}{\delta_j^{\frac{n-2}{2}}}U\left(\frac{x-\eta_j}{\delta_j}\right)} \end{equation}
where the all  concentration parameters $\delta_j$ have the same speed with respect to $\varepsilon$ and all the concentration points $\eta_j$ collapse to $0$  as $\varepsilon\to0$ (see \eqref{conf}, \eqref{deltaj} and \eqref{deltaj4}).

Unfortunately this first approximation is not as good as one can expect. We need to refine it adding  some extra terms which solve the  linear problem
\eqref{L-f}. 
To find thix extra terms, it is crucial the study of the linear theory developed in Section \ref{seclineare}. The novelty is  Theorem \ref{nondegeneracy} which states the non-degeneracy of the bubble \eqref{bubble}, i.e.   all the 
solution of the linearized problem  
\begin{equation}\label{linearized}
\left\{	\begin{aligned}
		-c_n\Delta v - \frac{n+2}{n-2}KU^\frac{4}{n-2}v=0 & \text{ in } \mathbb R^n_+, \\ 
	\frac{2}{n-2}\frac{\partial v}{\partial \nu} - \frac{n}{n-2}H U^\frac{2}{n-2}v=0 & \text{ on } \partial \mathbb R^n_+,
	\end{aligned}\right.
\end{equation}
 are  a linear combination of the functions
$$\mathfrak z_i(x):={\partial U\over\partial x_i}(x),\ i=1, \ldots, n-1,\quad \hbox{and}
\ \mathfrak z_n(x) :=\left(\frac{2-n}{2}U(x)-\nabla U(x)\cdot (x+\mathfrak D_n\mathfrak e_n)+\mathfrak D_n\frac{\partial U}{\partial x_n}\right)
.$$
The proof relies on some new ideas which allow a comparison among solutions to the linear problem \eqref{linearized} and the eigenfunctions of the Neumann problem on the ball equipped with the hyperbolic metric (see Lemma \ref{lemma-eigen}).
It is worthwhile to point out that  Han \& Lin in \cite{HY2}  and Almaraz in \cite{almaraz2011}  related similar linear problems in the case $K\geq0$ 
with some eigenvalue problems on spherical caps with standard metric.
\\

Once the refinement of the ansatz is made, we argue using a Ljapunov-Schmidt procedure. As it is usual, the last step consists in finding a critical point of the so-called {\em reduced energy} and to achieve this goal it is necessary to know the energy of   each bubble together with its correction.
The contribution of the correction to the energy is relevant and to capture it it is necessary to know the exact expression of the correction itself. This part is new and requires a lot of work. 
This is done in Section \ref{buiblo}.
 Finally, we can write  the main terms of the reduced energy which
come from  the
contribution   of each peak  $\eta_j$,  the interaction between different peaks $\eta_j$ and $\eta_i$ and the linear perturbation $\epsilon-$term. 
For example, in dimension $n\geq5$ (up to some constants) it looks like 
\begin{equation}\label{leading}\sum_{j=1}^k\left[\delta_j^2 \left( \|\pi(p)\|^2+\mathfrak Q(p)(\eta_i,\eta_j)\right)+\sum_{i\not=j}{\left(\delta_i\delta_j\right)^{n-2\over2}\over |\eta_i-\eta_j|^{n-2}}-\epsilon\delta_j\right]+h.o.t.\end{equation}
 here  $\mathfrak Q(p)$ is the quadratic form associated with the second derivative of $\|\pi(\cdot)\|^2$ at the point $p$ which is supposed to be positively definite
 (remind that  $p$ is a minimum point of $\pi$).
Now, if 
we  choose 
$$\delta_j\sim\epsilon\ \hbox{and}\ |\eta_j|\sim \eta\ \hbox{with}\ \epsilon^2\eta^2\sim {\epsilon^{n-2}\over\eta^{n-2}}$$
we can minimize the leading term in \eqref{leading} as soon as the term {\em ``h.o.t.''} is really an higher order term and this is true only in low dimensions $4\leq n\leq7.$ We believe that this is not merely a technical issue. It would be extremely interesting to understand if  in higher dimensions the clustering phenomena appears 
if the blow-up point is {\em umbilic}, i.e. $\pi(p)=0.$ It is clear that in this case building the clustering configuration is even more difficult than in the {\em non-umbilic} case, because the ansatz must be refined at an higher order.\\

Even if our result holds true in low dimensions we decide to write all the steps of the Ljapunov-Schmidt procedure in any dimensions because it would be useful in studying some related problems. In particular, our argument allows to prove that if $n\geq4$ the problem \eqref{pb} has always a solution with one blow-up boundary point $p$ which is non-umbilic and minimizes $\|\pi(\cdot)\|.$ In fact, if $k=1$ the expansion of the reduced energy in    \eqref{leading} holds true in any dimensions.  The existence 	of solutions with a single blow-up point was studied by Ghimenti, Micheletti \& Pistoia in \cite{GMP,GMP2019} when $K=0$ and $H=1$ in presence of a linear non-autonomus perturbation $\epsilon \gamma u$.
being $\gamma\in C^2(\de M)$.

 	{We remark that very recently  Ben Ayed \& Ould Ahmedou \cite{ABA} found solutions with \textit{clustering} blow-up points  on half spheres of dimension greater than five for a subcritical approximation of the geometric problem \eqref{pb0}, with a nonconstant function $K>0$ and $H=0$.
 	As far as we know, our result is a pioneering work in the construction of solutions with \textit{clustering} blow-up points for the problem \eqref{pb} with $K$ and $H$ not identically zero. In particular, it is the first time that this argument is carried out with $K<0$ and $H>0$, which has been proved to be especially challenging due to the existing competition between the critical terms of the energy functional.}\\

Finally, we point out that
the \textit{clustering} and \textit{towering} phenomena for {Yamabe-type equations} have been largely studied {in the literature, although most of the results available concern the problem on closed compact manifolds}. Consider {for instance} the linear perturbation of the classical Yamabe equation,
	 	\begin{equation}\label{yam}
		 		-\Delta_g u+\mathcal S_g u+\varepsilon u=u^{n+2\over n-2} \ \hbox{in}\ M.
		 	\end{equation}
 	{It is known that} in $3-$dimensional manifolds all the solutions to  \eqref{yam} have isolated and simple blow-up points (see Li and Zhu \cite{lz}). {However, this property is lost in higher dimensions.}\\
 	If $n\geq 7$, Pistoia \& Vaira \cite{pv} build a solution to  \eqref{yam}  with a clustering (i.e. non-isolated) blow-up point at a  non-degenerate and
 	non-vanishing minimum point of the Weyl’s tensor.
 	In any dimensions $n\geq 4$ the clustering phenomena appears if the linear perturbation term $\epsilon u$ is replaced with  a function $h_\varepsilon$
 	converging to a suitable  function $h_0$ as showed by Druet \& Hebey \cite{dh} and Robert \& Vétois  \cite{rv} if $n\geq6$ and  by  Thizy \& Vétois \cite{tv} if $n=4,5.$ \\ 	The existence of solutions to \eqref{yam} with a towering (i.e. isolated but non-simple) blow-up point has been proved in dimensions $n\geq7$, by Morabito, Pistoia \& Vaira \cite{mpv}  on symmetric nonlocally conformally flat manifolds and by Premoselli \cite{pre} in the locally flat case.\\
	In the spirit of \cite{dh,rv} it would be interesting to replace  the linear perturbation term in \eqref{pb} with some functions $h_\varepsilon$	in order to build a solution with a clustering blow-up point in any dimensions $n\geq4.$\\
	Moreover, inspired by the above results we strongly believe that it would be possible to build solutions to problem \eqref{pb}   with a towering blow-up point in any dimensions $n\geq4$.  This will be the topic (at least in a symmetric setting) of a forthcoming paper.\\

The paper is organized as follows. In Section \ref{seclineare} we study the linear problem \eqref{linearized}. In Section \ref{buiblo} we find out the correction term. In Section \ref{dimo}  we sketch the main steps of the proof, which 
  relies on standard arguments typical of the Ljapunov-Schmidt procedure. However, since it  
involves a lot of new delicate and quite technical estimates, in order  to streamline the reading of the work, we have decided to postpone them
  in the appendices. \\

In what follows we agree that $f\lesssim g$  means $|f|\leq c |g| $ for some positive constant $c$ which is independent on $f$ and $g$ and $f\sim g$  means $f= g(1+o(1))$.\\

\section{The key linear problem}\label{seclineare}

First of all, it is necessary to study the set of the solutions for the linearized problem:
\begin{equation}\label{L}
	\left\lbrace \begin{array}{ll}
		-\frac{4(n-1)}{n-2}\Delta v + \frac{n+2}{n-2}\abs{K}U^\frac{4}{n-2}v=0 & \text{ in } \R^n_+, \\[0.15cm]
	\frac{2}{n-2}\frac{\de v}{\de \nu} - \frac{n}{n-2}H U^\frac{2}{n-2}v=0 & \text{ on } \de \R^n_+,
	\end{array}\right.
\end{equation}
where $\nu = -e_n$ is the exterior normal vector to $\de\R^n_+$ and 
 \begin{equation}\label{U}U(x)=U_{1, x_0(1)}(\tilde x, x_n)=\frac{\alpha_n}{|K|^{\frac{n-2}{4}}}\frac{1}{\left(|\tilde x|^2+(x_n+\mathfrak D_n)^2-1\right)^{\frac{n-2}{2}}}\end{equation} where $\tilde x=(x_1, \ldots, x_{n-1})\in \mathbb R^{n-1}$ and $x_n\geq 0,$ stands for the simplest solution to the boundary Yamabe problem defined in \eqref{limpb} when $\D_n(p)>1$.
 
 \medskip

\begin{theorem}\label{nondegeneracy} Let $v \in H^1(\R_n^+)$ be a solution of \eqref{L}. Then $v$ is a linear combination of the functions
\begin{equation}\label{Ji}
\mathfrak z_i(x):=\frac{\partial U}{\partial x_i}(x)=\frac{\alpha_n}{|K|^{\frac{n-2}{4}}}\frac{(2-n) x_i}{\left(|\tilde x|^2+(x_n+\mathfrak D_n)^2-1\right)^{\frac{n}{2}}},\quad i=1, \ldots, n-1
\end{equation}
and
\begin{equation}\label{Jn}\begin{aligned}
\mathfrak z_n(x)&:=\left(\frac{2-n}{2}U(x)-\nabla U(x)\cdot (x+\mathfrak D_n\mathfrak e_n)+\mathfrak D_n\frac{\partial U}{\partial x_n}\right)\\
&=\frac{\alpha_n}{|K|^{\frac{n-2}{4}}}\frac{n-2}{2}\frac{|x|^2+1-\mathfrak D_n^2}{\left(|\tilde x|^2+(x_n+\mathfrak D_n)^2-1\right)^{\frac{n}{2}}}\end{aligned}\end{equation}
\end{theorem}
The proof of Theorem \ref{nondegeneracy} will require some preliminary results.
 In particular it is useful to recall the properties of the conformal Laplacian and boundary operator.  For a given metric $g$, they are defined as
 	\begin{equation*} 
 		L_gv=-\frac{4(n-1)}{n-2}\Delta_g v + S_g v \hsp \text{and,} \hsp B_gv=\frac{2}{n-2}\frac{\de v}{\de \nu} + h_g v,
 	\end{equation*}
 	being $S_g$ and $h_g$ the scalar and boundary mean curvatures. If we choose a conformal metric of the form $\rho^\frac{4}{n-2} g$, then $L_g$ and $B_g$ are \textit{conformally invariant} in the following sense:
 	\begin{equation}\label{c-i}
 		\begin{split}
 		L_gv &= \rho^\frac{n+2}{n-2} L_{\rho^\frac{4}{n-2}g}(\rho^{-1}v) \hsp \text{ and}\\ B_gv &= \rho^\frac{n}{n-2} B_{\rho^\frac{4}{n-2}g}(\rho^{-1}v).
 		\end{split}
 	\end{equation}

\begin{lemma}\label{radial} For every $i=0,1,\ldots$, let us consider the following boundary eigenvalue problem: 
\begin{equation}\label{aux}
	\left\lbrace \begin{array}{ll} \gamma_i''+(n-1)\coth{t}\gamma_i'-\left(\frac{i(i+n-2)}{\sinh^2t}+n\right)\gamma_i= 0, \text{ for } 0<t<T,\\[0.15cm] \gamma_i'(T)-\mu\gamma_i(T)=0,
	\end{array}\right.
\end{equation}
with $\mu \in \R.$ Then the following hold true:
\begin{enumerate}
	\item[(i)] If $i=0$, the only bounded solutions are of the form $\gamma_0(t) = c_1 \cosh{t}$ for $c_1\in \R,$ and satisfy \eqref{aux} with $\mu = \mu_0 := \tanh{T}.$
	\item[(ii)] If $i=1$, the only bounded solutions can be written in the form $\gamma_1(t)=c_2 \sinh{t}$ with $c_2\in \R,$ and solve $\eqref{aux}$ with $\mu =\mu_1:= (\tanh{T})^{-1}$.
	\item[(iii)] If $i\geq 2$ and $\mu\leq\mu_1$, \eqref{aux} does not admit bounded solutions.
\end{enumerate}
\end{lemma}
\begin{proof}
The proofs for $(i)$ and $(ii)$ use the exact same argument, so for the sake of brevity we will only show the proof for $(ii)$ 

\medskip

Firstly, observe that $\sinh{t}$ solves the first equation of \eqref{aux} with $i=1,$ and it is positive and bounded in $[0,T].$ Therefore, by linear ODE theory, we can write any solution to the equation in the form $\gamma_1(t)=c(t)\sinh{t}$ for some function $c(t)$. Straightforward computations show that $c(t)$ must solve the following relation:
\begin{equation}\label{constant-var}
	c''(t)\sinh{t} + \left(2\cosh{t} + (n-1)\coth{t}\sinh{t}\right)c'(t)=0.
\end{equation} 
If $c(t)$ is nonconstant, \eqref{constant-var} can be integrated and its solutions can be calculated explicitly. For $t$ small enough, they present the assymptotic behaviour
$$c(t)=c_3\left(\frac{1}{t}-(n-1)\ln {t}+O(t)\right),\hsp \text{with } c_3\neq 0.$$ Thus, $c(t)$ must be constant. The second part of $(ii)$ can be proved by direct computation.

\medskip

Finally, let us prove $(iii)$. We will consider the unique solution to \eqref{aux} with $\gamma_i(T)=1$, so we study the following situation:
$$	\left\lbrace \begin{array}{lll} \gamma_i''+(n-1)\coth{t}\gamma_i'-\left(\frac{i(i+n-2)}{\sinh^2t}+n\right)\gamma_i= 0, \text{ for } 0<t<T,\\[0.15cm] \gamma_i'(T)=\mu, \\[0.15cm] \gamma_i(T)=1.
	\end{array}\right.
$$
Let $i\geq 2$ and define $u_i=\gamma_i - \gamma_1$, with $\gamma_1$ denoting the unique solution to \eqref{aux} with $i=1$, $\mu=\mu_1$ and $\gamma_1(T)=1.$ Then $u_i$ satisfies 
\begin{equation}\label{aux-3}
	\left\lbrace \begin{array}{lll} u_i''+(n-1)\coth{t}u_i'-\left(\frac{i(i+n-2)}{\sinh^2t}+n\right)u_i= \frac{(i-1)(i+n-1)}{\sinh^2t}\gamma_1,\\[0.15cm] u_i'(T)=\mu-\mu_1, \\[0.15cm] u_i(T)=0.
	\end{array}\right.
\end{equation} 
Firstly, we will show that $u_i'(T)\geq 0$, proving that $\mu\geq \mu_1$. Assume by contradiction that $u_i'(T)<0$. Then since $u(T)=0$, there exists a small interval $(t_0,T)$ where $u(t)>0$. By the first equation of \eqref{aux-3}, since $i\geq 2$,
\begin{equation}\label{aux-4}
u_i''(t)+(n-1)\frac{\cosh{t}}{\sinh{t}}u_i'(t)\geq 0, \hsp \text{for } t_0<t<T.
\end{equation}
Inequality \eqref{aux-4} can be written in the more convenient way
$$
\left((\sinh{t})^{n-1}u_i'\right)'\geq 0, \hsp \text{for } t_0<t<T.
$$
Consequently, $(\sinh{T})^{n-1}u_i'(T)\geq (\sinh{t_0})^{n-1}u_i'(t_0)$. In view of this, if $t_0=0,$ then $u_i'(T)\geq 0$, a contradiction. However, if $t_0>0$, then $u_i(t_0)=u_i(T)=0$ so there exists $t_1\in (t_0,T)$ with $u_i'(t_1)=0$, again a contradiction. 

\medskip 

To see that the inequality is strict we only need to show that there are no solutions for $i\geq 2$ and $\mu = \mu_1$. Let us define the sequence of linear operators
$$
	A_i(\phi)(t)=-\phi''(t)-(n-1)\coth{t} \,\phi'(t) + \left(\frac{i(i+n-2)}{\sinh^2t}\right)\phi(t),
$$
subject to the boundary conditions $\phi(T)=1$ and $\phi'(T)=\mu_1$. $(i)$ implies that $A_0$ admits no solution, while $(ii)$ gives us a positive function $\phi_1$ satisfying $A_1(\phi_1)=0$. Therefore, $A_1$ is a nonnegative operator. Now, notice that the following relation holds:
$$	A_{i} = A_1 + (i-1)(i+n-1).
$$
Consequently, $A_i$ is a positive operator if $i\geq 2$ and $A_i(\phi)=0$ only admits the trivial solution.
\end{proof}

\begin{lemma}\label{lemma-eigen} Let $n\geq 3$. Denote by $B_R$ the ball of radius $0<R<1$ centered at the origin of $\R^n$, equipped with the hyperbolic metric 
$$	\hg = \frac{4\abs{dx}^2}{\left(1-\abs{x}^2\right)^2}.
$$
	 The first eigenvalue of the Neumann boundary problem  
\begin{equation}\label{eig}
	\left\lbrace \begin{array}{ll} \Delta _{\mathbb H}\phi - n \phi = 0 & \text{ in } B_R, \\ \frac{\partial \phi}{\partial \nu} = \mu \phi & \text{ on } \partial B_R.
		\end{array}\right.
\end{equation}
is $\mu_0 = \frac{2R}{1+R^2}$, with corresponding eigenfunction given by $\phi_0(x) = \frac{1+\abs{x}^2}{1-\abs{x}^2}$. The second eigenvalue is $\mu_1 = \frac{1+R^2}{2R}$ and the corresponding eigenspace is $n-$dimensional and generated by the family of eigenfunctions $$\left\lbrace \phi_1^i(x) = \frac{\abs{x}x_i}{1-\abs{x}^2}:\:i=1,\ldots,n\right\rbrace.$$
\end{lemma}
\begin{proof} Let $d_\mathbb{H}$ denote the geodesic distance from the origin, given by $d_\mathbb{H}(x)=\ln {\frac{1+\abs{x}}{1-\abs{x}}}$, and let $(t,\theta)$ be the geodesic polar coordinates of a point in $B_R\backslash\{0\}$, where $0<t<T=\ln  \frac{1+R}{1-R}$ and $\theta\in \Sph^{n-1}.$ In these coordinates, the hyperbolic metric takes the form 
	\begin{equation*}
	\hg = dt^2+\sinh^2tg_{\Sph^{n-1}},
	\end{equation*}
 where $g_{\Sph^{n-1}}$ is the standard metric on $\Sph^{n-1},$ and \eqref{eig} is equivalent to the following problem:
	
\begin{equation}\label{eig-polar}
	\left\lbrace \begin{array}{ll} \frac{\partial^2\phi}{\partial t^2} +(n-1)\coth{t} \frac{\partial\phi}{\partial t} + \frac{\Delta_{\Sph^{n-1}}\phi}{\sinh^2t} - n \phi = 0 & \text{ in } B_T, \\[0.15cm] \frac{\partial \phi}{\partial t} = \mu \phi & \text{ on } \partial B_T.
	\end{array}\right.
\end{equation}
See \cite{punzo} for more details. Using the fact that spherical harmonics generate $L^2(\Sph^{n-1})$, we write $\phi(t,\theta)=\sum_i \gamma_i(t)\xi_i(\theta)$, with $\xi_i$ satisfying the equation $$-\Delta_{\Sph^{n-1}}\xi_i=i(i+n-2)\xi_i, \hsp i=0,1,\ldots$$Therefore, separating variables, we can rewrite \eqref{eig-polar} in the following form:
\begin{equation*}
	\left\lbrace \begin{array}{ll} \sum_i \left(\gamma_i''+(n-1)\coth{t}\gamma_i'-\left(\frac{i(i+n-2)}{\sinh^2t}+n\right)\gamma_i\right)\xi_i = 0, \\[0.15cm] \sum_i\left(\gamma_i'(T)-\mu\gamma_i(T)\right)\xi_i=0.
	\end{array}\right.
\end{equation*}
Since the functions $\xi_i$ are orthogonal, the consequence is that each $\gamma_i$ is a solution of \eqref{aux}. By Lemma \ref{radial}, if $\mu = \mu_0 = \tanh{T} = \frac{2R}{1+R^2}$, there exists a solution for \eqref{aux} associated to $i=0$, and consequently a solution for \eqref{eig-polar}: $$\phi_0(t,\theta)=\cosh t.$$ $\phi_0$ is nonnegative in $[0,T]$, so $\mu_0$ must be the first eigenvalue of \eqref{eig}. Again by Lemma \ref{radial}, for $\mu = \mu_1 = (\tanh T)^{-1} = \frac{1+R^2}{2R}$ there exists a solution for \eqref{aux} associated to $i=1$, which produces the family of solutions for \eqref{eig-polar}:
$$\left\lbrace\phi_1^i(t,\theta)=\xi_i(\theta)\sinh t:\:i=1,\ldots,n\right\rbrace.$$ The same result guarantees that any other solution of \eqref{aux} must have $\mu>\mu_1$, finishing the proof.  
\end{proof}
Finally, we are in position to prove Theorem \ref{nondegeneracy}.
\begin{proof}[Proof of Theorem \ref{nondegeneracy}] This proof follows the ideas of \cite[Lemma 2.2]{almaraz2011}, with the fundamental difference that our problem is equivalent to one on a geodesic ball in the Hyperbolic space and not in the Euclidean sphere.

\medskip Let us denote $g_{\star} = \abs{K}U^\frac{4}{n-2}g_0$. The scalar and boundary mean curvatures of $\R^n_+$ with respect to $g_{\star}$ are given by \eqref{pb0}:
\begin{align*}
		S_{\star} = -1, \hsp h_{\star} = \frac{\D_n(p)}{\sqrt{n(n-1)}}.
\end{align*}
	
\medskip By means of \eqref{c-i}, it is possible to rewrite \eqref{L} as follows:
$$
\left\lbrace\begin{array}{ll}
\Delta_{\star}\bar v - \frac{1}{n-1}\bar v = 0 & \text{in } \R^n_+, \\[0.15cm]
{\frac{\de\bar v}{\de \nu_{\star}}} - \frac{\D_n(p)}{\sqrt{n(n-1)}}\bar v =0 & \text{on } \de \R^n_+.
\end{array} \right.
$$
with $\bar v = \abs{K}^{-\frac{n-2}{4}} U^{-1} v.$ The differential operators are explicit and their expressions are given by:
\begin{align}\label{def*}
\Delta_{\star} \bar v = \frac{\left(1-\abs{\tilde x}^2-(x_n+\D_n(p))^2\right)^2}{4n(n-1)}\Delta \bar v &+\frac{n-2}{2n(n-1)}\left(1-\abs{\tilde x}^2-(x_n+\D_n(p))^2\right) \nabla \bar v\cdot \left(x+\D_n(p)e_n\right), \\[0.15cm] \frac{\de \bar v}{\de \nu_{\star}} &= \frac{1-\abs{\tilde x}^2-(x_n+\D_n(p))^2}{2\sqrt{n(n-1)}}\frac{\de \bar v}{\de \eta}. \label{def*2}
\end{align}
Now let us denote by $\Phi$ the map given by
\begin{equation}\label{map}
	\Phi=\mathcal{K}^{-1}\circ \tau_{\D_n(p)}:\R^n_+\to B_1(0)\subset \R^n,
\end{equation}
where $\tau_{\D_n(p)}$ is the translation $x\to x+\D_n(p)\mathfrak e_n$ and $\mathcal{K}$ is the \textit{Cayley transform}, which maps conformally the ball of radius $1$ centered at the origin of $\R^n$ to the half-space $\R^n_+$. 
It can be proved that, up to composing with a certain isometry of $\mathbb{H}^n$, $\text{Im}(\Phi)=B_R(0)$ with $R=\D_n(p)-\sqrt{\D_n(p)^2-1}$. Moreover, $\Phi$ is a conformal map and satisfies 
\begin{equation}\label{phiconformal}
	\Phi^*g_\mathbb{H} = \frac{\abs{K}}{n(n-1)}U^\frac{4}{n-2}g_0,\hsp \text{where}\hsp \hg = \frac{4\abs{dx}^2}{(1-\abs{x}^2)^2}\hsp\text{on}\hsp B_R.
\end{equation}

\medskip Multiplying \eqref{def*} by $n(n-1)$ and \eqref{def*2} by $\sqrt{n(n-1)}$ and applying \eqref{phiconformal}, one can see that  $\hat v = (\bar U^{-1}v)\circ \Phi^{-1}$ is in $H^1(B_R)$ (see \cite[Lemma 6]{GMP}) and satisfies the following problem:
$$
	\left\lbrace\begin{array}{ll}
	\Delta_{\mathbb{H}}\hat v - n\hat v= 0 & \text{in } B_R, \\[0.15cm]
	\frac{\de \hat v}{\de \nu_{\mathbb{H}}} = \D_n(p)\hat v & \text{on } \de B_R,
	\end{array} \right.
$$
being 
\begin{align*}
	\Delta_{\mathbb{H}}\hat v &= \frac{\left(1-\abs{x}^2\right)^2}{4}\Delta \hat v + \frac{n-2}{2}\nabla\hat v\cdot x, \hsp \text{and} \\ \frac{\de \hat v}{\de \nu_{\mathbb{H}}} &= \frac{1-\abs{x}^2}{2}\frac{\de \hat v}{\de \eta}
\end{align*}
the Laplace-Beltrami operator and normal derivative on $B_R$ considered with respect to the hyperbolic metric $g_\mathbb{H}$. Theorem \ref{nondegeneracy} follows from Lemma \ref{lemma-eigen}, taking into account that $ \D_n(p)=\frac{1+R^2}{2R} $ and
\begin{equation}\label{rel}
\hat {\mathfrak z}_i = c_i\phi^i_1 \hsp\text{ for every}\hsp i=1,\ldots,n.
\end{equation}
 
\end{proof}

\section{The building block}\label{buiblo}

Let $p\in\partial M$. The main ingredient to cook up our solutions are the bubbles defined in 
\eqref{bubble} together with the correction found out in Proposition \ref{vp}, i.e. the {\em building block} of the solutions we are looking for is
\begin{equation}\label{bb}
\mathcal W_p (\xi):= \chi\left(\left(\psi_{p}^\partial\right)^{-1}(\xi)\right)\left[\frac{1}{\delta ^{\frac{n-2}{2}}}U\left(\frac{\left(\psi_{p}^\partial\right)^{-1}(\xi)}{\delta}\right)+ \frac{1}{\delta^{\frac{n-4}{2}}}V_p\left(\frac{\left(\psi_{p}^\partial\right)^{-1}(\xi)}{\delta}\right)\right] \end{equation}
 where  $\psi_p^\partial: \mathbb R^n_+\to M$ are the Fermi coordinates in a neighborhood of $p$ and
$\chi$ is a radial cut-off function, with support in a ball of radius $R$. Here $U$ is the bubble defined in \eqref{U} and $V_p$  solves \eqref{L-f}.

\subsection{The correction of the bubble}
Let us introduce the correction term as the function $V_p:\mathbb R^n_+\to \mathbb R$ which is defined below.
\begin{proposition}\label{vp} Let $U$ be as in \eqref{U} and set
	$$
		\mathtt E_p(x)=\sum\limits_{i,j=1}^{n-1}\frac{8(n-1)}{n-2}h^{ij}(p)\frac{\de^2 U(x)}{\de x_i \de x_j}x_n,\ x\in \mathbb R^n_+
	$$
where $h^{ij}(p)$ are the coefficients of the second fundamental form of $M$ at the point $p\in\de M$. Then the problem
\begin{equation}\label{L-f}
	\left\lbrace \begin{array}{ll}
		-\frac{4(n-1)}{n-2}\Delta V + \frac{n+2}{n-2}\abs{K}U^\frac{4}{n-2}V=	\mathtt E_p & \text{ in } \R^n_+, \\[0.15cm]
		\frac{2}{n-2}\frac{\de V}{\de \nu} - \frac{n}{n-2}H U^\frac{2}{n-2}V=0 & \text{ on } \de \R^n_+,
	\end{array}\right.
\end{equation}
admits a solution $V_p$ satisfying the following properties:
\begin{enumerate}
	\item[(i)] $\int\limits_{\R^n_+}V_p(x)\mathfrak z_i(x)dx=0$ for any $ i=1,\ldots,n$ (see \eqref{Ji} and \eqref{Jn})
	\item[(ii)] $\abs{\nabla^\alpha V_p}(x) \lesssim \frac1{\left(1+\abs{x}\right)^{n-3+\alpha}}$ for any  $x\in\R^n_+ $ and $\alpha = 0,1,2$
	\item[(iii)]  
	\begin{equation*}
		|K|\int_{\R^n_+} U^\frac{n+2}{n-2}V_pdx = (n-1)H\int_{\de\R^n_+} U^\frac{n}{n-2}V_p\,d\tilde x.
	\end{equation*}
	\item[(iv)]   if  $n\geq 5$ $$\int_{\R^n_+}\left(-\frac{4(n-1)}{n-2}\Delta V_p+\frac{n+2}{n-2}\abs{K}U^\frac{4}{n-2}V_p\right)V_p\geq 0,$$	
	\item[(v)] the map $p\mapsto V_p$ is $C^2(\de M)$. 
\end{enumerate}
\end{proposition}
\begin{proof}

First, we will introduce some notation to reduce ourselves to the study of a problem similar to \eqref{L}. Let $\bar U = \abs{K}^\frac{n-2}{4}U$, then we can rewrite \eqref{L-f} as:
	
$$		\left\lbrace \begin{array}{ll}
			-\frac{4(n-1)}{n-2}\Delta v + \frac{n+2}{n-2}\bar U^\frac{4}{n-2}v=f & \text{ in } \R^n_+, \\[0.15cm]
			\frac{2}{n-2}\frac{\de v}{\de \nu} - \frac{n}{n-2}\frac{\D_n(p)}{\sqrt{n(n-1)}} \bar U^\frac{2}{n-2}v=0 & \text{ on } \de \R^n_+,
		\end{array}\right.
$$

	Let $\Phi$ be as in \eqref{map}. We set
	\begin{equation*}
	\hat f(\Phi^{-1}(x))=\frac{n(n-2)}{4}f(x)\bar U(x)^{-\frac{n+2}{n-2}}
	\end{equation*}
Arguing as in the proof of Theorem \ref{nondegeneracy}, we see that it is enough to consider the following problem for $\hat v = ({\bar U}^{-1}v)\circ \Phi^{-1}:$
\begin{equation}\label{hyperbolic-f}
	\left\lbrace\begin{array}{ll}
		\Delta_{\mathbb{H}}\hat v - n\hat v= \hat f & \text{in } B_R, \\[0.15cm]
		\frac{\de \hat v}{\de \nu_{\mathbb{H}}} = \D_n(p)\hat v & \text{on } \de B_R,
	\end{array} \right.
\end{equation}
By the area formula and \eqref{rel}:
\begin{align*}
\int_{B_r}\phi_1^k(z) \hat f(z)d\mu_{\hg} &= c_n\int_{\R^n_+} \phi_1^k(\Phi^{-1}(x))h^{ij}(p)\frac{\de^2 U(x)}{\de x_i \de x_j}x_nU^{-\frac{n+2}{n-2}}\abs{\jac \Phi^{-1}}dx \\ &= c_n \int_{\R^n_+} \mathfrak z_k(x)h^{ij}(p)\frac{\de^2 U(x)}{\de x_i \de x_j}x_ndx\\ &=c_n\sum_{\stackrel{i,j=1}{i\neq j}}^{n-1}\int_{0}^{+\infty}\int_{\R^{n-1}}\frac{x_ix_jp_k(\tilde x,x_n)}{\left(\abs{x}^2-1\right)^{n+1}}d\tilde x d x_n,
\end{align*}
being $p_k$ a polynomial in $x$ with $\deg p_k=1$ if $k=1,\ldots,n-1,$ and $\deg p_n=2$. To get the last identity we  have also used definitions \eqref{U}, \eqref{Ji}, \eqref{Jn} and the condition $\sum_i h^{ii}(p)=0$. Now, if we take polar coordinates in $\R^{n-1}$ and use the fact that 
$$\int_{\Sph_r^{n-2}}p^{\gamma} = \frac{r^2}{\gamma(\gamma+n-3)}\int_{\Sph_r^{n-2}}\Delta p^{\gamma}$$
for every homogeneous polynomial $p^\gamma$ of degree $\gamma$, we can check that
\begin{equation*}
	\int_{B_r}\phi_1^k(z) \hat f(z)d\mu_{\hg} = 0 \hsp \text{for all } k=1,\ldots,n.
\end{equation*}
By elliptic linear theory, there exists a solution $\hat v$ to \eqref{hyperbolic-f} which is orthogonal to $\{\phi_1^k\}_{k=1}^n$. Consequently, $v=\bar U (\hat v \circ \Phi)$ is a solution of \eqref{L-f} orthogonal to $\{\mathfrak z_k\}_{k=1}^n$. 

\medskip 

Given $z\in B_R$, let $G_{z_0}$ denote the Green's function solving the problem
$$
	\left\lbrace\begin{array}{ll}
		\Delta_{\mathbb{H}}G_{z}-nG_{z}=\delta_{z}-\sum_{k=1}^n \frac{\phi_1^k(z)\phi_1^k}{\norm{\phi_1^k}_{L^2}} & \text{in } B_R, \\
		\frac{\de G_{z}}{\de \nu_{\mathbb{H}}}-\D_n(p)G_{z}=0 & \text{on } \de B_R.
	\end{array}\right.
$$
Then, by Green's representation formula
\begin{align}
\psi(z)&=\sum_{k=1}^n\int_{B_R} \frac{\phi_1^k(z)\phi_1^k(w)}{\norm{\phi_1^k}_{L^2}} \psi(w)d\mu_{\hg}(w) - \int_{B_R}G_{z}(w)\Delta_{\mathbb{H}}\psi(w)d\mu_{\hg}(w)\nonumber \\ \label{green-rep}&-\int_{\de B_R}G_{z}(w)\left(\frac{\de}{\de \nu_{\mathbb{H}}}-\D_n(p)\right)\psi(w)d\mu_{\hg}(w)
\end{align}
Choosing $\psi = \hat v$ in \eqref{green-rep},
\begin{equation*}
\hat v(z)=-\int_{B_R}G_{z}(w)\hat{f}(w)d\mu_{\hg}(w),
\end{equation*}
then
$$\abs{\hat v(z)}\leq c_n \abs{h^{ij}(p)}\int_{B_R}\abs{w-z}^{2-n}\abs{w+\D_n(p)e_n}^{-3}d\mu_{\hg}(w).
$$
By \cite[Proposition 4.12]{aubinbook} with $\alpha = 2$ and $\alpha = n-3$, $$\hat v(z)\leq c_n\abs{h^{ij}(p)}\abs{z+\D_n(p)e_n}^{-1}.$$
Hence, $u=\bar U(\hat u \circ \Phi)$ satisfies the estimate (ii). To prove (iii), integrate by parts \eqref{hyperbolic-f} to obtain:
\begin{equation}\label{rel2}
n\int_{B_R}\hat v d\mu_{\hg} - \int_{B_R} \hat f d\mu_{\hg}= \D_n(p)\int_{\de B_R}\hat vds_{\hg}.
\end{equation}
By \eqref{phiconformal},
\begin{align}
	d\mu_{\hg}&=\left(n(n-1)\right)^{-\frac n2}\abs{K}^\frac{n}{2} U^{\frac{2n}{n-2}}\abs{dx}^2, \label{metrics1} \\ ds_{\hg} &=\left(n(n-1)\right)^{-\frac{n-1}{2}}\abs{K}^\frac{n-1}{2}U^\frac{2(n-1)}{n-2}\abs{d\tilde x}^2. \label{metrics2}
\end{align}
Therefore, by the area formula:
\begin{equation}\label{rel3}
	\int_{B_R} \hat f d\mu_{\hg} = c_n \int_{\R^n_+} h^{ij}(p)\frac{\de^2U(x)}{\de x_i\de x_j}x_nU(x)dx = 0.
\end{equation}
Combining \eqref{rel2} and \eqref{rel3} with the relations \eqref{metrics1} and \eqref{metrics2}, we get the desired equality.
\\

Finally, integrating by parts we obtain
\begin{equation}\label{rel4}
-\int_{B_R}(\Delta_{\mathbb{H}}\hat v)\hat vd\mu_{\hg} = \int_{B_R}\abs{\nabla_{\mathbb{H}}\hat v}^2d\mu_{\hg}-\D_n(p)\int_{\de B_R}\hat v^2 ds_{\hg}.
\end{equation}
By Lemma \ref{lemma-eigen}, we know that
\begin{equation*}
	\inf\left\lbrace \frac{\int_{B_R}\left(\abs{\nabla_{\mathbb{H}}\psi}^2+n \psi^2\right)d\mu_{\hg}}{\int_{\de B_R} \psi^2ds_{\hg}}: \int_{\de B_R}\psi \phi_0 ds_{\hg}=0\right\rbrace=\D_n(p).
\end{equation*}
If we showed that $\hat v$ is orthogonal to $\phi_0$ in $L^2(\de B_R)$, we would get:
\begin{equation}\label{rel5}
	\int_{B_R}\abs{\nabla_{\mathbb{H}}\hat v}^2d\mu_{\hg}+n	\int_{B_R}\hat v^2d\mu_{\hg} \geq \D_n(p)\int_{\de B_R}\hat v^2ds_{\hg}.
\end{equation}
Then, combining \eqref{rel4} and \eqref{rel5}, we obtain
\begin{equation}\label{rel6}
	-\int_{B_R}(\Delta_{\mathbb{H}}\hat v)\hat vd\mu_{\hg} + n	\int_{B_R}\hat v^2d\mu_{\hg} \geq 0.
\end{equation}
By the properties of the conformal Laplacian, we know that $$L_{\mathbb{H}}\psi = \frac{-4(n-1)}{n-2}\Delta_\mathbb{H}\psi-n(n-1)\psi = n(n-1)L_{\star}\phi,$$ with $\psi\circ\Phi^{-1}=\phi$. Thus, multiplying \eqref{rel6} by $\frac{4(n-1)}{n-2}$, we obtain
\begin{align*}
0&\leq n(n-1)\int_{\R^n_+}L_{\star}\left(\bar U^{-1}v\right)\bar U^{-1}v \:d\mu_{\star}+n(n-1)\left(1+\frac{4}{n-2}\right)\int_{\R^n_+}\bar U ^\frac{4}{n-2}v^2 dx \\ &= n(n-1)\left(\frac{4(n-1)}{n-2}\int_{\R^n_+}(\Delta v)v\:dx+\frac{n+2}{n-2}\int_{\R^n_+}\abs{K}U^\frac{4}{n-2}v^2\:dx\right).
\end{align*}
We conclude the proof by showing that $\int_{\de B_R}\hat v \phi_0 ds_{\hg}=0.$ We will use the fact that $\phi_0$ solves \eqref{eig} for $\mu=\D_n(p)^{-1}$ and that $\hat v$ is a solution of \eqref{hyperbolic-f}. Integrating by parts:
\begin{align*}
0=\int_{B_R}\hat f \phi_0 d\mu_{\hg} &= \int_{B_R}\left(\phi_0\Delta_{\mathbb{H}}\hat v-\hat v \Delta_{\mathbb{H}}\phi_0\right)d\mu_{\hg} \\ &= \int_{\de B_R}\left(\frac{\de \hat v}{\de \nu_{\mathbb{H}}}\phi_0-\frac{\de \phi_0}{\de \nu_{\mathbb{H}}}\hat v\right)ds_{\hg} = \left(\D_n(p)-\frac{1}{\D_n(p)}\right)\int_{\de B_R} \hat v \phi_0 ds_{\hg},
\end{align*} 
where the first identity can be proved using the same argument as in (i).\\ For the proof of (v) we can reason as in Proposition 7 of \cite{GMP}.
\end{proof}

We end this section by giving a more careful description of the function $V_p$. In particular, we need to detect the leading part of $V_p$ and
since its decay  changes as $n=4$ or $n\geq 5$ we have to distinguish the two cases.\\

{\bf Case $n=4$.}
We decompose $V_p$ into three parts: the main part $\bar w_p$ is almost a rational function, the second part $\zeta_p$ is a harmonic function with prescribed boundary condition and the third one $\psi_p$ is an higher order term. More precisely, let \beq\label{dec}
V_p=\bar w_p + \zeta_p+\psi_p\eeq where $\bar w_p$, $\zeta_p$ and $\psi_p$ solve respectively the following problems 
\beq\label{soleqn4}
-6\Delta \bar w_p=\mathtt E_p(x),\qquad \mbox{in}\,\,\ \mathbb R^4_+\eeq 
\beq\label{pbzetapn4}
\left\{\begin{aligned} &-6\Delta\zeta_p=0\quad &\mbox{in}\,\, &\mathbb R^4_+\\
&\frac{\partial \zeta_p}{\partial \nu}=2 H U \zeta_p +\left(2H U \bar w_p -\frac{\partial \bar w_p}{\partial \nu}\right),\quad &\mbox{on}\,\,\ &\partial\mathbb R^4_+\end{aligned}\right.\eeq
and
\beq\label{pbpsipn4}
\left\{\begin{aligned} &-6\Delta\psi_p+3|K|U^2\psi_p=-3|K|U^2 (\bar w_p+\zeta_p)\quad &\mbox{in}\,\, &\mathbb R^4_+\\
&\frac{\partial \psi_p}{\partial \nu}=2 H U \psi_p,\quad &\mbox{on}\,\,\ &\partial\mathbb R^4_+\end{aligned}\right.\eeq
The following holds:
\begin{lemma}\label{barw0}
Set
$$\bar w^0_p(x):= \sum_{\substack{i,j=1\atop i\not=j}}^3M_{ij}(p)\frac{x_4x_ix_j}{\left(\abs{\tilde x}^2+(x_4+\D_4)^2-1\right)^2},\
\hbox{with}\ 
 M_{ij}(p)=\frac{2h^{ij}(p)\alpha_4}{\abs{K}^\frac12}.$$
Then
\beq\label{expwpn4} \bar w_p(x)- \bar w^0_p(x)= \bigo{\frac{1}{1+\abs{x} ^2}}\ \hbox{and}\
|\nabla \bar w_p(x)-  \nabla\bar w^0_p(x)|= \bigo{\frac{1}{1+\abs{x}^3}},
\eeq 
\beq \label{caso4zeta}
|\zeta_p|\lesssim \frac{1}{1+|x|}\quad\mbox{and}\quad |\nabla \zeta_p|\lesssim \frac{1}{1+|x|^2}
\eeq
\beq\label{caso4psi}
|\psi_p|\lesssim \frac{1}{1+|x|^3}\quad\mbox{and}\quad |\nabla\psi_p|\lesssim \frac{1}{1+|x|^4}.
\eeq
\end{lemma}
\begin{proof}
First we observe that the estimates \eqref{caso4zeta} and \eqref{caso4psi} follows by using the same arguments of Proposition \ref{vp} applied to problems \eqref{pbzetapn4} and \eqref{pbpsipn4}.\\ Now it remains to show \eqref{expwpn4}.\\
We remark that we can write $$\bar w_p=2h^{ij}(p)\partial^2_{ij}z_p,\qquad i, j=1, \ldots, 3\,\,\ i\neq j$$ where $z_p$ solves the problem
\begin{equation}\label{aux-z}
-\Delta z_p = U(x)x_4=\frac{\alpha_4}{|K|^{\frac 12}}\frac{x_4}{|\tilde x|^2+(x_4+\D_4)^2-1}\qquad \mbox{in}\,\,\ \mathbb R^4_+.
\end{equation}
The aim is then to understand the main term of the solution $z_p$ of \eqref{aux-z}.\\
It holds that 
$$\frac{x_4+\D_4}{|\tilde x|^2+(x_4+\D_4)^2-1}=\frac 12\ln(|\tilde x|^2+(x_4+\D_4)^2-1).$$ Thus if we take $\Phi_0$ a solution of 

\begin{equation}\label{aux-phi0}
-\Delta \Phi_0 = \ln \left(\abs{\tilde x}^2+(x_4+\D_4)^2-1\right),\quad\mbox{in}\,\,\ \mathbb R^4_+
\end{equation}
and $\Phi_1$ a solution of 
\begin{equation}\label{aux-phi1}
-\Delta \Phi_1 = \frac{\D_4}{\abs{\tilde x}^2+(x_4+\D_4)^2-1}\quad\mbox{in}\,\,\ \mathbb R^4_+
\end{equation}
then, $$z_p=\frac{\alpha_4}{|K|^{\frac 12}}\left(\frac{1}{2}\frac{\de\Phi_0}{\de x_4}-\Phi_1\right)$$ solves \eqref{aux-z}. \\The advantage of \eqref{aux-phi0} and \eqref{aux-phi1} over \eqref{aux-z} is that, under a change of variables $$-\Delta [\Phi_0(\tilde x, x_4-\D_4)]=\ln (|\tilde x|^2+x_4^2-1)=\ln(|x|^2-1)\quad \mbox{in}\,\, \mathbb R^4_+$$ and 
$$-\Delta [\Phi_1(\tilde x, x_4-\D_4)]=\frac{\D_4}{|\tilde x|^2+x_4^2-1}=\frac{\D_4}{|x|^2-1}\quad \mbox{in}\,\, \mathbb R^4_+.$$ If we assume that $\Phi_0(\tilde x, x_4-\D_4)$ and $\Phi_1(\tilde x, x_4-\D_4)$ are radially symmetric, i.e. $\tilde \Phi_0(|x|)=\Phi_0(\tilde x, x_4-\D_4)$ and $\tilde \Phi_1(|x|)=\Phi_1(\tilde x, x_4-\D_4)$ then it is reduced to solve the equations
$$-\tilde \Phi_0''-\frac{N-1}{r}\tilde \Phi_0'=\ln(r^2-1)\quad \mbox{in}\,\, (1, +\infty)$$ and
$$-\tilde \Phi_1''-\frac{N-1}{r}\tilde \Phi_1'=\frac{\D_4}{r^2-1}\quad \mbox{in}\,\, (1, +\infty).$$
The general solutions are expressed as 
\begin{align*}
\tilde\Phi_0(r)&=\frac{c_1+3r^4-2(r^2-1)^2\ln (r^2-1)}{16r^2}, \\
\tilde\Phi_1(r) &= \frac{c_2}{r^2}+\frac{\D_4\ln (r^2-1)}{4r^2}-\frac{\D_4\ln (r^2-1)}{4},
\end{align*}
with $c_1,c_2\in \R$. Using the symmetries of the coefficients $h^{ij}$ (with the aid of  computer assisted proof), we get
\begin{equation*}
	\bar w_p(x)= \frac{2x_4\displaystyle\sum_{\substack{i,j=1\\i<j}}^3M_{ij}x_ix_j}{\left(\abs{\tilde x}^2+(x_4+\D_4)^2-1\right)^2}+ \bigo{\frac{1}{(1+\abs{x})^2}}\ \hbox{in a $C^1-$sense}.\end{equation*}
	That concludes the proof.\\
\end{proof}

{\bf Case $n\geq 5$.} We can decompose $V_p=w_p+\psi_p$ where $w_p$ solves
\begin{equation}\label{casongeq5-2}
-c_n\Delta w_p + c_n\frac{n(n+2)}{(|\tilde x|^2+(x_n+\mathfrak D_n)^2-1)^2}w_p=\mathtt E_p(x)\quad \mbox{in}\,\, \mathbb R^n_+,\end{equation}

and $\psi_p$ solves

\begin{equation}\label{casongeq5-3}
\left\{\begin{aligned}
&-c_n\Delta \psi_p + c_n\frac{n(n+2)}{(|\tilde x|^2+(x_n+\mathfrak D_n)^2-1)^2}\psi_p=0, \quad &\mbox{in}\,\, &\mathbb R^n_+\\
&\frac{\partial \psi_p}{\partial \nu}=\frac {n \mathfrak D_n}{(|\tilde x|^2+\mathfrak D_n^2-1)} \psi_p+\left(\frac {n \mathfrak D_n}{(|\tilde x|^2+\mathfrak D_n^2-1)}w_p-\frac{\partial w_p}{\partial \nu}\right)\quad &\mbox{on}\,\, &\partial\mathbb R^n_+\end{aligned}\right.\end{equation}

We claim that 
\beq\label{wp}
w_p(x)=\frac{\beta_n}{4n}\frac{\sum_{i, j=1\atop j\neq i}^{n-1}h^{ij}(p)x_i x_j (x_n-\mathfrak D_n)}{(|\tilde x|^2+(x_n+\mathfrak D_n)^2-1)^{\frac n 2}}.\eeq
Indeed, we look for a solution of \eqref{casongeq5-2} of the form $$w_p(x)=\frac{q(x)}{(|\tilde x|^2+(x_n+\mathfrak D_n)^2-1)^{\frac n 2}},$$ with $q(x)$ a polynomial function. 
Straightforward computations show that $q(x)$ has to verify the equation
\begin{equation*}
\mathcal{L}(q(x)) = \beta_n  x_n\sum_{i, j=1\atop i\neq j}^{n-1} h^{ij}(p) x_i x_j,
\end{equation*}
being 
$$\mathcal L(q)=-(|\tilde x|^2+(x_n+\mathfrak D_n)^2-1)\Delta q +2n \nabla q \cdot (x+\mathfrak D_n \mathfrak e_n)-2n q$$
with $\beta_n:=\frac{2n(n-2)\alpha_n}{|K|^{\frac{n-2}{4}}}.$\\
Observe that it is possible to write $$q(x)=\beta_n\sum_{i, j=1\atop i\neq j}^{n-1}h^{ij}(p)q^{ij}(x),$$ where every $q^{ij}$ is a polynomial solving $\mathcal L(q^{ij})= x_i x_j x_n$. We note that $\mathcal L(x_i x_j x_n)=4 n x_i x_j x_n +2n x_i x_j \mathfrak D_n$ and $\mathcal L(x_i x_j)=2n x_i x_j$, so 
$$\mathcal L\left(\frac{1}{4n}x_ix_j (x_n-\mathfrak D_n)\right) = x_ix_jx_n.$$
Therefore,   \eqref{wp} follows.\\

\subsection{The energy of the building block}
Let us define the energy $J_\varepsilon: H^1(M)\to\mathbb R$ 
\begin{equation}\label{energia}\begin{aligned}
J_\varepsilon(u)&:=\int_M \left(\frac{c_n}{2}|\nabla_g u|^2+\frac 12\mathcal S_g u^2-K \mathfrak G(u)\right)\,d\nu_g-c_n\frac{n-2}{2}\int_{\partial M} H \mathfrak F(u)\, d\sigma_g\\ & \quad+(n-1)\varepsilon\int_{\partial M} u^2\,d\sigma_g\end{aligned}\end{equation}
where $$\mathfrak G(s):=\int_0^s \mathfrak g(t)\, dt,\ \mathfrak g(t):=(t^+)^{n+2\over n-2}\quad \hbox{and}\quad\mathfrak F(s):=\int_0^s \mathfrak f(t)\,dt,\ \mathfrak f(t):=(t^+)^{n \over n-2}.$$ 
 It is useful to introduce the
integral quantities  whose properties are listed in Appendix \ref{app0}:
\begin{equation}\label{ima}I_m^\alpha:=\int_0^{+\infty}\frac{\rho^\alpha}{(1+\rho^2)^m}\, d\rho\end{equation}
and if $p\in\partial M$
\begin{equation}\label{fip}
 \varphi_m(p):=\int_{\mathfrak D_n}^{+\infty}\frac{1}{(t^2-1)^m}\, dt\quad\hbox{and}\quad {\hat\varphi}_m(p):=\int_{\mathfrak D_n}^{+\infty}\frac{(t-\mathfrak D_n)^2}{(t^2-1)^m}\, dt .\end{equation}
We will assume  that $H$ and $K$ are constant functions.
We remark that $\mathfrak D_n,$  $\varphi_m$ and $\hat  \varphi_m$ are also constant functions, so we will omit the dependance on $p.$\\

In the following result we compute the energy of the {\em building block} \eqref{bb} (the proof is quite technical and is postponed in  Appendix \ref{app-energy}). 

\begin{proposition}\label{energy-bubble}
It holds true that
$$J_\epsilon(\mathcal W_p )=\mathfrak E  -\zeta_n(\delta) \Big[ \mathfrak {b}_n\|\pi(p)\|^2+o'_n(1)\Big]+\varepsilon\delta\Big[{\mathfrak c}_n +o''_n(1)\Big]$$
 where (see \eqref{ima} and \eqref{fip})
 \begin{equation}\label{En}\mathfrak E  :=\frac{\mathfrak a_n}{|K|^{\frac{n-2}{2}}} \left[-(n-1)\varphi_{\frac{n+1}{2}} +\frac{\mathfrak D_n}{(\mathfrak D_n^2-1)^{\frac{n-1}{2}}}\right],\   \mathfrak a_n:=\alpha_n^\dsh \omega_{n-1}I^n_{n-1}\frac{n-3}{(n-1)\sqrt{n(n-1)}},\end{equation}
moreover (see Proposition \ref{intestn5} for the definition of $\mathfrak f_n$)
 
\begin{equation}\label{bn}{\mathfrak b}_n :=\frac 12\mathfrak f_n+{n-2\over n-1}\alpha_n^2\omega_{n-1}I^n_{n-1}\frac1{|K|^{n-2\over2}}\left(4(n-3){\hat\varphi}_{n-1\over2}+\varphi_{n-3\over2} \right),\,\, n\geq 5\end{equation}
\beq\label{b4} {\mathfrak b}_4:=\frac{192\pi^2}{|K|}+\frac{\alpha_4^2\omega_3 I_3^4}{|K|}\eeq 
and
\begin{equation}\label{Cn}
{\mathfrak c}_n :=2(n-2)\omega_{n-1}\alpha_n^2\frac{1}{|K|^{\frac{n-2}{2}}\left(\mathfrak D_n^2-1\right)^{\frac{n-3}{2}}}I^{n}_{n-1}.
\end{equation} 
Moreover
$$\zeta_4(\delta):=\delta^2|\ln \delta|\quad \hbox{and}\quad \zeta_n(\delta):=\delta^2\ \hbox{if}\ n\geq5.$$
and
\begin{equation}\label{on}o'_n(1)=\left\{\begin{aligned}&\mathcal O(\delta ) \ \hbox{if}\ n\geq6,\\
 & \mathcal O(\delta |\ln\delta|) \ \hbox{if}\ n=5,\\
 & \mathcal O(|\ln\delta|^{-1}) \ \hbox{if}\ n=4\end{aligned}\right.\ \hbox{and}\ 
 o''_n(1)=\left\{\begin{aligned}&\mathcal O(\delta ) \ \hbox{if}\ n\geq5,\\ & \mathcal O(\delta |\ln\delta|) \ \hbox{if}\ n=4.\end{aligned}\right.\eeq

\end{proposition}

\section{Proof of Theorem \ref{main}}\label{dimo}

\subsection{Preliminaries}
Since $(M, g)$ belongs to the positive Escobar class   (i.e. the quadratic part of the Euler functional associated to the problem is positive definite), we can  provide the Sobolev space $H^1(M)$ with the scalar product $$\langle u, v\rangle :=\int_M \left(c_n \nabla_g u\nabla_g v + \mathcal S_g uv\right)\, d\nu_g$$ where $d\nu_g$ is the volume element of the manifold. We let $\|\cdot\|$ be the norm induced by $\langle \cdot, \cdot\rangle$. \\ Moreover, for any $u\in L^q(M)$ (or $u\in L^q(\partial M)$) we denote the $L^q-$ norm of $u$ by $\|u\|_{L^q(M)}:=(\int_M |u|^q\, d\nu_g)^{\frac 1 q}$ (respectively $\|u\|_{L^q(\partial M)}:=(\int_{\partial M}|u|^q\, d\sigma_g)^{\frac 1q}$ where $d\sigma_g$ is the volume element of $\partial M$.) \\
We have the well-known embedding continuous maps
$$\begin{aligned} &\mathfrak i_{\partial M}: H^1(M)\to L^t(\partial M)\qquad\qquad & \mathfrak i_M: H^1(M) \to L^{\frac{2n}{n-2}}(M)\\ &\mathfrak i^*_{\partial M}: L^{t'}(\partial M)\to H^1(M)\qquad\qquad & \mathfrak i^*_M: L^{\frac{2n}{n+2}}(M)\to H^1(M)\end{aligned}$$ for $1\leq t\leq \frac{2(n-1)}{n-2}$.\\
Now given $\mathfrak f\in L^{\frac{2(n-1)}{n}}(\partial M)$ the function $w_1=\mathfrak i^*_{\partial M}(\mathfrak f)$ in $H^1(M)$ is the unique solution of the equation \beq\label{w1}\left\{\begin{aligned}&-c_n\Delta_g w_1+\mathcal  S_g w_1=0\quad &\mbox{in}\,\, M\\ &\frac{\partial w_1}{\partial \nu}=\mathfrak f\quad &\mbox{on}\,\, \partial M.\end{aligned}\right.\eeq
Moreover, if we let $\mathfrak g\in L^{\frac{2n}{n+2}}(M)$, the function $w_2=\mathfrak i^*_M(\mathfrak g)$ is the unique solution of the equation
\beq\label{w2}\left\{\begin{aligned}&-c_n\Delta_g w_2+\mathcal S_g w_2=\mathfrak g\quad &\mbox{in}\,\, M\\ &\frac{\partial w_2}{\partial \nu}=0\quad &\mbox{on}\,\, \partial M.\end{aligned}\right.\eeq
By continuity of $\mathfrak i_M, \mathfrak i_{\partial M}$ we get
$$\|\mathfrak i^*_{\partial M}(\mathfrak f)\|\leq C_1 \|\mathfrak f\|_{L^{\frac{2(n-1)}{n}}(\partial M)}\quad\quad \|\mathfrak i^*_M(\mathfrak g)\|\leq C_2 \|\mathfrak g\|_{L^{\frac{2n}{n+2}}(M)}$$ for some $C_1>0$ and independent of $\mathfrak f$ and some $C_2>0$ and independent of $\mathfrak g$.\\
Then we rewrite the problem \eqref{pb} as 
\begin{equation*}\label{pb1} u=\mathfrak i^*_M(K\mathfrak g(u))+\mathfrak i^*_{\partial M}\left(\frac{n-2}{2}\left(H \mathfrak f(u)-\varepsilon u \right)\right),\ \hbox{ with $\mathfrak g(u):=(u^+)^{\frac{n+2}{n-2}}$ and $\mathfrak f(u)=(u^+)^{\frac{n}{n-2}}$}\end{equation*}\\

\subsection{The ansatz}
Having in mind   Proposition \ref{energy-bubble}, we fix a non-umbilic and non-degenerate minimum point $p\in\partial M$ of the function $\|\pi(\cdot)\|^2$ with   and  we choose 
\begin{equation}\label{d0}d_0:= \frac{{\mathfrak c}_n}{2\mathfrak b_n\|\pi(p)\|^2}\end{equation} 
where ${\mathfrak b}_n$ and ${\mathfrak c}_n$ are  positive constants   defined in \eqref{bn}, \eqref{b4} and \eqref{Cn}.
For any integer $k\geq 1$, we look for solutions of \eqref{pb}  of the form 
\beq\label{ue}u_\varepsilon(\xi):=\underbrace{\sum_{j=1}^k\mathcal W_j(\xi)}_{:=\mathcal W(\xi)}+\Phi_\varepsilon(\xi)\quad\xi\in M\eeq
where 
$$\mathcal W_j(\xi)=\chi\left(\left(\psi_{p}^\partial\right)^{-1}(\xi)\right)W_j(\xi)$$
and
$$ W_j(\xi):= {\frac{1}{\delta_j^{\frac{n-2}{2}}}U\left(\frac{\left(\psi_{p}^\partial\right)^{-1}(\xi)-\eta(\e)\tau_j}{\delta_j}\right)} +\delta_j {\frac{1}{\delta_j^{\frac{n-2}{2}}}V_p\left(\frac{\left(\psi_{p}^\partial\right)^{-1}(\xi)-\eta(\e)\tau_j}{\delta_j}\right)} . $$

Here
 $\chi$ is a radial cut-off function with support in a ball of radius $R$, the bubble $U$ is defined in \eqref{U} and $V_p$ solves \eqref{L-f}.
Moreover, 
\begin{equation}\label{conf}\tau_j\in \mathcal C:=\left\{(\tau_1, \ldots, \tau_k)\in\mathbb R^{(n-1)k}\,\, :\,\, \tau_i\neq \tau_j\ \hbox{if}\ i\neq j\right\}\end{equation}
and given $d_0$ as in \eqref{d0} the concentration parameters $\delta_j$  and the rate of the concentration points $\eta(\varepsilon)$ are choosen as follows:
\begin{equation}\label{deltaj}
\delta_j:=\varepsilon\left(d_0+\eta(\varepsilon) d_j\right),\   d_j\in[0, +\infty)\  \hbox{and}\ \eta(\varepsilon):=\varepsilon^\alpha\ \hbox{with}\ \alpha:=\frac{n-4}n \quad \hbox{if}\ n\geq5
\end{equation}
or   
\begin{equation}\label{deltaj4}
\delta_j:=\rho(\e)\left(d_0+\eta(\e) d_j\right)\  d_j\in[0, +\infty) \  \hbox{and}\ \eta(\varepsilon): =\frac{1}{|\ln \rho(\e)|^{\frac 14}}
 \quad \hbox{if}\ n=4
\end{equation}
where $\rho$  is the inverse function of 
$\ell:(0, e^{-\frac 12})\to \left(0, \frac{e^{-1}}{2}\right)$ defined by $\ell(s)=-s\ln  s$.   We remark that $\rho(\e)\to 0$  as $\e\to 0$.\\
Finally, the remainder term $\Phi_\varepsilon(\xi)$ belongs to $\mathcal K^\bot$ defined as follows.\\
Let us define for i$=1, \ldots, n,$ and $ j=1,\ldots, k$
$$
\mathcal Z_{j, i}(\xi)=\chi\left(\left(\psi_p^\partial\right)^{-1}(\xi)\right)Z_{j, i}(\xi),\ \hbox{with}\ Z_{j, i}(\xi):=\frac{1}{\delta_j^{\frac{n-2}{2}}}\mathfrak z_i\left(\frac{\left(\psi_p^\partial\right)^{-1}(\xi)-\eta(\e)\tau_j}{\delta_j}\right)$$ where $\mathfrak z_i$ are given in \eqref{Ji}and \eqref{Jn}.\\
We decompose $H^1(M)$ in the direct sum of the following two subspaces $$\mathcal K={\rm span}\left\{\mathcal Z_{j, i}\,\,:\,\, i=1, \ldots, n,\,\, j=1, \ldots, k\right\}$$ and $$\mathcal K^\bot:=\left\{\psi\in H^1(M)\,:\, \langle \psi,\mathcal Z_{j, i}\rangle=0,\quad i=1, \ldots, n,\,\, j=1, \ldots, k\right\}.$$ 
\subsection{The reduction process} We define the projections $$\Pi: H^1(M)\to \mathcal K\qquad \Pi^\bot: H^1(M)\to \mathcal K^\bot.$$
Therefore solving \eqref{pb1} is equivalent to solve the couple of equations 
\beq\label{eaux}\Pi^\bot\left\{u_\varepsilon -\mathfrak i^*_M(K \mathfrak g(u_\varepsilon))-\mathfrak i^*_{\partial M}\left(\frac{n-2}{2}\left(H \mathfrak f(u_\varepsilon)-\varepsilon u_\varepsilon\right)\right)\right\}=0\eeq
\beq\label{bif}\Pi\left\{u_\varepsilon -\mathfrak i^*_M(K \mathfrak g(u_\varepsilon))-\mathfrak i^*_{\partial M}\left(\frac{n-2}{2} \left(H \mathfrak f(u_\varepsilon)-\varepsilon  u_\varepsilon\right)\right)\right\}=0\eeq where $u_\varepsilon$ is defined in \eqref{ue}.\\

\subsection{Solving the equation (\ref{eaux}): the remainder term} 
\noindent We shall find the remainder term $\Phi_\varepsilon\in\mathcal K^\bot$ in \eqref{ue}. Let us rewrite the equation \eqref{eaux} as 
\beq\label{aux1}
\mathcal E+\mathcal L(\Phi_\varepsilon)+\mathcal N(\Phi_\varepsilon)=0\eeq
where the error term $\mathcal E$ is  
$$\mathcal E:=\Pi^\bot\left\{\mathcal W-\mathfrak i^*_M\left(K \mathfrak g\left(\mathcal W\right)\right)-\mathfrak i^*_{\partial M}\left(\frac{n-2}{2} \left(H \mathfrak f\left(\mathcal W\right)-\varepsilon \mathcal W \right)\right)\right\},$$
the linear operator $\mathcal L$ is  
$$\mathcal L(\Phi_\varepsilon):=\Pi^\bot\left\{\Phi_\varepsilon-\mathfrak i^*_M\left(K \mathfrak g'\left(\mathcal W \right)\Phi_\varepsilon\right)-\mathfrak i^*_{\partial M}\left(\frac{n-2}{2}\left(H \mathfrak f'\left(\mathcal W \right)\Phi_\varepsilon -\varepsilon \Phi_\varepsilon\right)\right)\right\}$$
 and  the quadratic term $\mathcal N (\Phi_\varepsilon)$ is 
$$\begin{aligned}
\mathcal N (\Phi_\varepsilon)&:=\Pi^\bot\Big\{-\mathfrak i^*_M\left[K\left(\mathfrak g\left(\mathcal W +\Phi_\varepsilon\right)-\mathfrak g\left(\mathcal W \right)-\mathfrak g'\left(\mathcal W \right)\Phi_\varepsilon\right)\right] \\
&\qquad\qquad\left.-\mathfrak i^*_{\partial M}\left[\frac{n-2}{2}H\left(\mathfrak f\left(\mathcal W +\Phi_\varepsilon\right)-\mathfrak f\left(\mathcal W \right)-\mathfrak f'\left(\mathcal W \right)\Phi_\varepsilon\right)\right]\right\}.\end{aligned}$$

The following result holds true.
\begin{proposition}\label{esistphi}
For any compact subset $\mathcal A\subset (0, +\infty)^k\times \mathcal C$ (see \eqref{conf}) there exists $\e_0>0$  such that for any $\e\in (0, \e_0)$ and for any $(d_1, \ldots, d_k, \tau_1, \ldots, \tau_k)\in \mathcal A$ there exists a unique function $\Phi_\e\in\mathcal K^\bot$ which solves equation \eqref{aux1}. 
Moreover, the map $(d_1, \ldots, d_k, \tau_1, \ldots, \tau_k) \mapsto \Phi_\e (d_1, \ldots, d_k, \tau_1, \ldots, \tau_k)$ is of class $C^1$ and 
$$\|\Phi_\e\|\lesssim\left\{\begin{aligned}&\e^2 \quad &\hbox{if}\,\,\ n\geq 7\\ &\e^2|\ln \e|^{\frac 23} \quad &\hbox{if}\,\,\ n=6\\
& \e^{\frac 32}  \quad &\hbox{if}\,\,\ n=5\\
&\rho(\e)\quad &\hbox{if}\,\,\ n=4.\end{aligned}\right.
$$\end{proposition}

We omit the proof because it is standard and relies on the following two key results whose proof is given in Appendix \ref{app-c} and Appendix \ref{app-d} respectively. First, 
we estimate the size of the error term $\mathcal E$.
\begin{lemma}\label{errorsize}
Let $n\geq 4$. For any compact subset $\mathcal A\subset (0, +\infty)^k\times \mathcal C$ there exists $\e_0>0$  such that for any $\e\in (0, \e_0)$ and for any $(d_1, \ldots, d_k, \tau_1, \ldots, \tau_k)\in \mathcal A$ it holds
$$
\|\mathcal E\|\lesssim\left\{\begin{aligned}&\e^2  \quad &\hbox{if}\,\,\ n\geq 7\\ &\e^2|\ln \e|^{\frac 23} \quad &\hbox{if}\,\,\ n=6\\
& \e^{\frac 32}  \quad &\hbox{if}\,\,\ n=5\\
&\rho(\e)\quad &\hbox{if}\,\,\ n=4.\end{aligned}\right.
$$

\end{lemma}
Next, we  study the invertibility of the linear operator $\mathcal L$.
 \begin{lemma}\label{lineareinv}
Let $n\geq 4$. For any compact subset $\mathcal A \subset (0, +\infty)^k \times \mathcal C$ there exist a positive constant $C>0$ and $\e_0>0$ such that for any $\e\in(0, \e_0)$ and any $(d_1, \ldots, d_k, \tau_1, \ldots, \tau_k)\in\mathcal A$ it holds
$$\|\mathcal L(\phi)\|\geq C\|\phi\|.$$\end{lemma}
Next, we have to
\subsection{Solving equation (\ref{bif}): the reduced problem} 
We know that solutions to problem
  \eqref{pb}  are critical points of the energy functional $J_\e$ defined in \eqref{energia}.
Let us introduce the so-called {\em reduced energy}
\begin{equation}\label{RE}
\mathfrak J_\e(d_1, \ldots, d_k, \tau_1, \ldots, \tau_k):= J_\e(\mathcal W+\Phi_\e)\end{equation} where
the remainder term $\Phi_\e$ has been found in Proposition \ref{esistphi}.\\ 
We shall prove   that a critical point of the reduced energy provides a solution to our problem.
\begin{proposition}\label{ridotto} 
It holds true that
\begin{enumerate}
\item If $(d_1, \ldots, d_k, \tau_1, \ldots, \tau_k)\in [0, +\infty)^k\times (\mathbb R^n)^k$ is a critical point of the reduced energy \eqref{RE}, then $\mathcal W+\Phi_\e$ is a critical point of $J_\e$ and so it solves  \eqref{pb}.
\item  The following expansion holds true
$$
\begin{aligned}
&\mathfrak J_\e(d_1, \ldots, d_k, \tau_1, \ldots, \tau_k)=k\mathfrak E  + k\theta_n(\e)\left({\mathfrak c}_nd_0-{\mathfrak b}_n\|\pi(p)\|^2d_0^2\right)\\
&+\Theta_n(\e)\underbrace{\left(-\mathfrak b_n\sum_{i=1}^k \mathfrak Q(p)(\tau_i, \tau_i)-{\mathfrak b}_n\|\pi(p)\|^2\sum_{i=1}^k d_i^2-\frac{{\mathfrak d}_n}{|K|^{n-2\over2}}\sum_{i<j}\frac{d_0^{n-2}}{|\tau_i-\tau_j|^{n-2}}\right)}_{=:\mathfrak F_n(d_1, \ldots, d_k, \tau_1, \ldots, \tau_k)}\\
&+o\left(\Theta_n(\e)\right),\end{aligned}
$$
$C^0-$ uniformly with respect to $(d_1, \ldots, d_k, \tau_1, \ldots, \tau_k)$ in a compact set of $(0, +\infty)^k\times \mathcal C$. 
Here  $\mathfrak Q(p)$ is the quadratic form associated with the second derivative of $p\to\|\pi(p)\|^2$ (being zero the first derivative),  $\mathfrak E $, ${\mathfrak c}_n$ are constants defined in \eqref{En} and \eqref{Cn}, respectively, and  
$${\mathfrak d}_n:=\alpha_n^{2^\star}\omega_{n-1}I_{\frac{n+2}{2}}^{n-2}.$$ Moreover 
$$\theta_4(\e)=\Theta_4(\e):=\rho^2(\e)|\ln \rho(\e)|\ \hbox{if}\ n=4\quad \hbox{and}\quad \theta_n(\e):=\e^2,\ \Theta_n(\e):=\e^{ \frac{4(n-2)}n}\ \hbox{if}\ n\geq 5.$$
\end{enumerate}
\end{proposition}

As we claimed above, we  shall postpone in Appendix \ref{app-e} the proof  because even if it relies on standard arguments, it requires a lot of new  elaborated and technical estimates.
\\

\subsection{Proof of Theorem \ref{main}: completed}
The  claim immediately follows by Proposition \ref{ridotto} taking into account that the function  
 $\mathfrak F_n(d_1, \ldots, d_k, \tau_1, \ldots, \tau_k)$  has a maximum point which is stable under $C^0-$ perturbations.

\appendix
\section{Auxiliary results}\label{app0}

We have (see \cite{almaraz2011} Lemmas 9.4 and 9.5) the following results:
$$\begin{aligned}
&I_m^\alpha:=\int_0^{+\infty}\frac{\rho^\alpha}{(1+\rho^2)^m}\, d\rho=\frac{2m}{\alpha+1}I_{m+1}^{\alpha+2},\quad\hbox{for}\,\, \alpha+1<2m\\
&I_m^\alpha=\frac{2m}{2m-\alpha-1}I_{m+1}^{\alpha},\quad\hbox{for}\,\, \alpha+1<2m+2\\
&I_m^\alpha=\frac{2m-\alpha-3}{\alpha+1}I_{m}^{\alpha+2},\quad\hbox{for}\,\, \alpha+3<2m.\end{aligned}$$
In particular, if $n\geq 4$
\begin{equation}\label{In}
I_n^n=I_n^{n-2}=\frac{n-3}{2(n-1)}I_{n-1}^{n},\quad I_{n-1}^{n-2}=\frac{n-3}{n-1}I_{n-1}^{n},\quad I_{n-2}^{n-2}=\frac{2(n-2)}{n-1}I_{n-1}^{n}.\end{equation}

\begin{lemma}

It holds:

\begin{equation}\label{Jalpham}
\int_{\mathbb R^n_+}\frac{|\tilde x|^\alpha}{(|\tilde x|^2+(x_n+\mathfrak D_n)^2-1)^m}\, dx= \omega_{n-1}I_m^{n-2+\alpha}\varphi_{\frac{2m-n-\alpha+1}{2}},\qquad \hbox{for}\,\, n+\alpha<2m
\end{equation}
\begin{equation}\label{JalphamI}
\int_{\mathbb R^{n-1}}\frac{|\tilde x|^\alpha}{(|\tilde x|^2+\mathfrak D_n^2-1)^m}\, dx= \omega_{n-1}(\mathfrak D_n^2-1)^{\frac{n+\alpha-1-2m}{2}}I_m^{n-2+\alpha},\qquad \hbox{for}\,\, n-1+\alpha<2m
\end{equation}
\begin{equation}\label{JalphamII}
\int_{\mathbb R^n_+}\frac{x_n^2|\tilde x|^\alpha}{(|\tilde x|^2+(x_n+\mathfrak D_n)^2-1)^m}\, dx= \omega_{n-1}I_m^{n-2+\alpha}{\hat\varphi}_{\frac{2m-n-\alpha+1}{2}},\qquad \hbox{for}\,\, n+2+\alpha<2m
\end{equation}

\end{lemma}

\begin{proof}
Let us show \eqref{Jalpham}.
$$\begin{aligned}
&\int_{\mathbb R^n_+}\frac{|\tilde x|^\alpha}{(|\tilde x|^2+(x_n+\mathfrak D_n)^2-1)^m}\, dx=\int_0^{+\infty}\int_{\mathbb R^{n-1}}\frac{|\tilde x|^\alpha}{(|\tilde x|^2+(x_n+\mathfrak D_n)^2-1)^m}\, dx\\
&\quad=\int_{\sqrt{\mathfrak D_n^2-1}}^{+\infty}\frac{\lambda}{\sqrt{\lambda^2+1}}\int_{\mathbb R^{n-1}}\frac{|\tilde x|^\alpha}{(|\tilde x|^2+\lambda^2)^m}\, dx\\
&\quad=\omega_{n-1}\int_{\sqrt{\mathfrak D_n^2-1}}^{+\infty}\frac{\lambda^{n+\alpha-2m}}{\sqrt{\lambda^2+1}}\int_{0}^{+\infty}\frac{r^{n-2+\alpha}}{(r^2+1)^m}\, dr\\
&\quad= \omega_{n-1}I_{m}^{n-2+\alpha}\int_{\mathfrak D_n}^{+\infty}\frac{1}{(t^2-1)^{\frac{2m-n-\alpha+1}{2}}}\, dt\\
&\quad=\omega_{n-1}I_m^{n-2+\alpha}\varphi_{\frac{2m-n-\alpha+1}{2}}.
\end{aligned}$$
For what concerning \eqref{JalphamI} we have
$$\begin{aligned}
&\int_{\mathbb R^{n-1}}\frac{|\tilde x|^\alpha}{(|\tilde x|^2+\mathfrak D_n^2-1)^m}\, dx=\omega_{n-1}\int_{0}^{+\infty}\frac{r^{n-2+\alpha}}{(r^2+\mathfrak D_n^2-1)^m}\, dr\\
&\quad=\omega_{n-1}(\mathfrak D_n^2-1)^{\frac{n+\alpha-1-2m}{2}}\int_{0}^{+\infty}\frac{r^{n-2+\alpha}}{(r^2+1)^m}\, dr\\
&\quad= \omega_{n-1}I_{m}^{n-2+\alpha}(\mathfrak D_n^2-1)^{\frac{n+\alpha-1-2m}{2}}.
\end{aligned}$$
Finally 
$$\begin{aligned} &\int_{\mathbb R^n_+}\frac{x_n^2|\tilde x|^\alpha}{(|\tilde x|^2+(x_n+\mathfrak D_n)^2-1)^m}\, dx=\int_{\sqrt{\mathfrak D_n^2-1}}^{+\infty}\frac{\lambda\left(\sqrt{\lambda^2+1}-\mathfrak D_n\right)^2}{\sqrt{\lambda^2+1}}\int_{\mathbb R^n_+}\frac{|\tilde x|^\alpha}{(|\tilde x|^2+\lambda^2)^m}\, d\tilde x\, d\lambda\\
&=\omega_{n-1}I_m^{n-2+\alpha}\int_{\sqrt{\mathfrak D_n^2-1}}^{+\infty}\frac{\lambda^{n+\alpha-2m}(\sqrt{\lambda^2+1}-\mathfrak D_n)^2}{\sqrt{\lambda^2+1}}\, d\lambda\\
&=\omega_{n-1}I_m^{n-2+\alpha}\int_{\mathfrak D_n}^{+\infty}\frac{(t-\mathfrak D_n)^2}{(t^2-1)^{\frac{2m-n-\alpha+1}{2}}}\, dt\\
&=\omega_{n-1}I_m^{n-2+\alpha}{\hat\varphi}_{\frac{2m-n-\alpha+1}{2}}.\end{aligned}$$
\end{proof}

\begin{lemma}\label{varphi}
It holds:
$$\varphi_{\frac{n+1}{2}}=\frac{1}{n-1}\frac{\mathfrak D_n}{(\mathfrak D_n^2-1)^{\frac{n-1}{2}}}-\frac{n-2}{n-1}\varphi_{\frac{n-1}{2}}$$ and
$$\varphi_{\frac{n-1}{2}}=\frac{1}{n-3}\frac{\mathfrak D_n}{(\mathfrak D_n^2-1)^{\frac{n-3}{2}}}-\frac{n-4}{n-3}\varphi_{\frac{n-3}{2}}$$
Moreover
$${\hat\varphi}_m:=\varphi_{m-1}+(\mathfrak D^2_n+1)\varphi_m-\frac{1}{m-1}\frac{\mathfrak D_n}{(\mathfrak D_n^2-1)^{m-1}}.$$
\end{lemma}
\begin{proof} It is enough to integrate by parts.
\end{proof}

It is also useful to remind the expansion of the metric given in \cite{E1}.
\begin{lemma}\label{metricexpansion} Let $(M,g)$ be a compact Riemannian manifold with boundary. If $x=(\tilde x, x_n)=(x_1\ldots,x_n)$ are the Fermi coordinates centered at a point $p\in \de M$, then the following expansion holds:
	\begin{itemize}
	\item $ \sqrt{\abs{g(x)}} = 1 - \frac{1}{2}\left(\norm{\pi(p)}^2+\text{Ric}_\nu(p)\right)x_n^2-\frac{1}{6}\bar R_{ij}(p)x_ix_j +O(\abs{x}^3) ,$
	\item
$ g^{ij}(x)=\delta_{ij}+2h^{ij}(p)x_n+\frac13 \bar R_ {ikjl}(p)x_kx_l+2\frac{\de h^{ij}}{\de x_k}(p)x_kx_n  +\left(R_{injn}(p)+3h_{ik}(p)h_{kj}(p)\right){x_n}^2+O(\abs{x}^3) $
\item $g^{an}(x)=\delta_{an}$ 
\item $\Gamma_{ij}^k(x)=\bigo{\abs{x}} $
	\end{itemize}
where $  \pi(p)$ is the second fundamental form at $p$, $h^{ij}(p)$ are its coefficients, $\bar R_ {ikjl}(p)$  and
$ R_ {abcd}(p)$ are the curvature tensor of the boundary $\partial M $ and $ M,$ respectively, $\bar R_{ij}(p) = \bar R_{ikjk}(p) $ are
the coefficients of the Ricci tensor, and  $\text{Ric}_\nu(p) = R_{inin}(p)= R_{nn}(p).$ Here the indices $i,j,k=1,\dots,n-1$ and $a,b=1,\dots,n.$
\end{lemma}

Finally, we also remind the useful estimate:  
\beq\label{yyl}||s+t|^q-s^q|\lesssim \left\{\begin{aligned}&\min\{s^{q-1}|t|,|t|^q\}\ \hbox{if}\ 0<q\leq1 \\ &
s^{q-1}|t|+|t|^q \ \hbox{if}\ q>1\end{aligned}\right.\quad \hbox{ for any $s>0$ and $t\in\mathbb R$}.
\eeq

\section{Proof of Proposition \ref{energy-bubble}}\label{app-energy}
First we need two technical propositions in which we compute the contribution of  correction term $V_p$ to the energy.
\begin{proposition}\label{intestn4}
Let $n=4$ and $\bar w_p$ the solution of \eqref{soleqn4} (the first term of the expansion of $V_p$ in \eqref{dec}). Then
 \beq\label{intestn41}
 \int_{B^+_{\frac R\delta}}|\nabla \bar w_p|^2= \frac{64\pi^2}{\abs{K}} \|\pi(p)\|^2 |\ln \delta|+\mathcal O(1).\eeq
\end{proposition}
\begin{proof}
We first reduce the integral into \eqref{intestn41} to one in a simpler domain. Let $Q^+_r$ denote the upper half of the ball of radius $r$ in the $\norm{\cdot}_{\infty}$ of $\R^4$, that is,
\begin{equation*}
	Q^+_r= \left\lbrace x\in \R^4:\,x_4\geq 0\hsp\text{and } -r\leq x_i\leq r,\:i=1,2,3.\right\rbrace
\end{equation*}
and let $A^+(\frac{R}{4\delta},\frac{R}{\delta})$ the upper half of the annulus with radii $r_1=\frac{R}{4\delta}$ and $r_2=\frac{R}{\delta}.$\\
Then, we can write $$B^+_{\frac{R}{\delta}} = Q^+_{\frac{R}{2\delta}}\sqcup\Omega_\delta,$$
with $\Omega_\delta:=B^+_{\frac{R}{\delta}} \setminus Q^+_{\frac{R}{2\delta}}$. Notice that $\Omega_\delta$ satisfies $\Omega_\delta\subset A^+(\frac{R}{4\delta},\frac{R}{\delta})$. Then, by using also Lemma \ref{barw0}, we get
\begin{align*}
	\int_{B^+_{\frac R \delta}}\abs{\nabla \bar w_p}^2 = \int_{Q^+_{\frac R {2\delta}}}\abs{\nabla \bar w_p}^2 + \int_{\Omega_\delta}\abs{\nabla \bar w_p}^2=\int_{Q^+_{\frac R {2\delta}}}\abs{\nabla \bar w_p^0}^2 + \int_{\Omega_\delta}\abs{\nabla \bar w_p}^2+\mathcal O(1),
\end{align*}
and
\begin{align*}
	\int_{\Omega_\delta}\abs{\nabla \bar w_p}^2 \leq \int_{A^+(\frac{R}{4\delta},\frac{R}{\delta})}\abs{\nabla \bar w_p}^2 \leq C \int_{\frac{R}{4\delta}}^{\frac{R}{\delta}}\left(1+r\right)^{-4}r^{3}dr = \mathcal O(1).
\end{align*}
Then $$\int_{B^+_{\frac R \delta}}\abs{\nabla \bar w_p}^2=\int_{Q^+_{\frac R {2\delta}}}\abs{\nabla \bar w_p^0}^2+\mathcal O(1).$$
The latter integral can be calculated explicitly with the help of mathematical software. Firstly, we compute 
\begin{equation*}
\int_{0}^{\frac{R}{2\delta}}\int_{\frac{-R}{2\delta}}^{\frac{R}{2\delta}}\int_{\frac{-R}{2\delta}}^{\frac{R}{2\delta}}\int_{\frac{-R}{2\delta}}^{\frac{R}{2\delta}} \abs{\frac{\de \bar w_p^0}{\de x_i}}^2dx_1dx_2dx_3dx_4 = \frac{\pi^2}{30}\left(\sum_{\substack{k>j\\k\neq i}}^3\sum_{\substack{j=1 \\ j\neq i}}^3 \left(3M_{ij}^2+{M_{jk}^2}\right)\right)\abs{\ln \delta} + \bigo{1},
\end{equation*}
for $i=1,2,3.$ Similarly,
\begin{equation*}
	\int_{0}^{\frac{R}{2\delta}}\int_{\frac{-R}{2\delta}}^{\frac{R}{2\delta}}\int_{\frac{-R}{2\delta}}^{\frac{R}{2\delta}}\int_{\frac{-R}{2\delta}}^{\frac{R}{2\delta}} \abs{\frac{\de \bar w_p^0}{\de x_4}}^2dx_1dx_2dx_3dx_4 = \frac{\pi^2}{30}\left(\sum_{\substack{i,j=1 \\ i<j}}^3 3M_{ij}^2\right)\abs{\ln \delta} + \bigo{1}.
\end{equation*}
Hence, by the definition of $M_{ij}(p)=\frac{8\sqrt{3}}{\abs{K}^\frac{1}{2}}h^{ij}(p):$
\begin{equation*}
	\int_{Q^+_{\frac R {2\delta}}} \abs{\nabla \bar w^0_p}^2 = \frac{64\pi^2}{\abs{K}}\norm{\pi(p)}^2\abs{\ln \delta}+\bigo1,
 \end{equation*}
being $\norm{\pi(p)}^2=h_{12}^2(p)+h_{13}^2(p)+h_{23}^2(p).$
\end{proof}

\begin{proposition}\label{intestn5}
Let $n\geq 5$. Let $V_p$ a solution of \eqref{L-f}, then there exists a nonnegative constant $\mathfrak{f}_n$ depending only on $n$ and $\D_n$ such that
$$
	\int_{\R^n_+}\left(-c_n\Delta V_p+\frac{n+2}{n-2}\abs{K}U^{4\over n-2}V_p\right)V_p = \mathfrak{f}_n \norm{\pi(p)}^2.
$$
\end{proposition}
\begin{proof}
Let decompose $V_p=w_p+\psi_p$ where $w_p$ solves
\eqref{casongeq5-2} 
and $\psi_p$ solves \eqref{casongeq5-3}.\\
For sake of convenience, let us define
\begin{equation*}
b(\tilde x, 0) = \frac {n \mathfrak D_n}{(|\tilde x|^2+\mathfrak D_n^2-1)}w_p(\tilde x,0)+\left.\frac{\partial w_p}{\partial x_n}\right\vert_{x_n=0}.
\end{equation*}
Since $V_p=w_p+\psi_p$ ($w_p$ and $\psi_p$ are defined in \eqref{casongeq5-2} and in \eqref{casongeq5-3} respectively) then 
$$\begin{aligned}&\int_{\mathbb R^n_+}\left(-\Delta V_p + \frac{n(n+2)}{(|\tilde x|^2+(x_n+\mathfrak D_n)^2-1)^2}V_p\right)V_p\, dx=\beta_n h^{ij}(p)\int_{\mathbb R^n_+}\frac{x_i x_j x_n }{(|\tilde x|^2+(x_n+\mathfrak D_n)^2-1)^{\frac{n+2}{2}}}V_p\, dx \\
	&=\underbrace{\beta_n h^{ij}(p)\int_{\mathbb R^n_+}\frac{x_i x_j x_n }{(|\tilde x|^2+(x_n+\mathfrak D_n)^2-1)^{\frac{n+2}{2}}}w_p\, dx}_{I_{w_p}}+\underbrace{\beta_n h^{ij}(p)\int_{\mathbb R^n_+}\frac{x_i x_j x_n }{(|\tilde x|^2+(x_n+\mathfrak D_n)^2-1)^{\frac{n+2}{2}}}\psi_p\, dx}_{I_{\psi_p}}\end{aligned}$$
Let us evaluate separately $I_{w_p}$ and $I_{\psi_p}$.
\begin{equation}\label{casongeq5-6}
\begin{split}
I_{w_p}=&\beta_n h^{ij}(p)\int_{\mathbb R^n_+}\frac{x_i x_j x_n }{(|\tilde x|^2+(x_n+\mathfrak D_n)^2-1)^{\frac{n+2}{2}}}w_p\, dx \\ &=\frac{\beta_n^2}{4n} h^{ij}(p)h^{k\ell}(p)\underbrace{\int_{\mathbb R^n_+}\frac{x_i x_jx_k x_\ell x_n(x_n-\mathfrak D_n) }{(|\tilde x|^2+(x_n+\mathfrak D_n)^2-1)^{n+1}}\, dx}_{I_{ijk\ell}}
\end{split}
\end{equation}

Notice that by symmetry reasons and the fact that $h^{ii}(p)=0$ for every $i=1,\ldots,n-1$, we can write
\begin{equation}\label{casongeq5-5}
\frac{\beta_n^2}{4n}\sum_{\substack{i,j=1\\i\neq j}}^{n-1}\sum_{\substack{k,\ell=1\\k\neq l}}^{n-1} h^{ij}(p)h^{k\ell}(p)I_{ijk\ell} = \frac{\beta_n^2}{n} \sum_{\substack{i,j=1\\i< j}}^{n-1}\sum_{\substack{k,\ell=1\\k< l}}^{n-1} h^{ij}(p)h^{k\ell}(p)I_{ijk\ell}.
\end{equation}
In view of \eqref{casongeq5-5}, if $(i,j)\neq (k,\ell)$, there exists an index, let say $i$, such that $i\notin\{j,k,\ell\}$. In that case, it is easy to see that 
\begin{align*}
&h^{ij}(p)h^{k\ell}(p)\int_{\mathbb R^n_+}\frac{x_i x_jx_k x_\ell x_n(x_n-\mathfrak D_n) }{(|\tilde x|^2+(x_n+\mathfrak D_n)^2-1)^{n+1}}\, dx \\ &= h^{ij}(p)h^{k\ell}(p)\int_{\R^{n-1}_+}\int_{-\infty}^{+\infty} \frac{\de}{\de x_i}\left(\frac{-1}{2n}\frac{x_jx_kx_lx_n(x_n-\D_n)}{(\abs{\tilde x}^2+(x_n+\D_n)^2-1)^{n+1}}\right)dx_i\prod^n_{\substack{\alpha=1\\ \alpha\neq i}}dx_\alpha = 0.
\end{align*}
Consequently,
\begin{equation*}
h^{ij}(p)h^{k\ell}(p)I_{ijk\ell} = \sum^{n-1}_{\substack{i,j=1}}h^{ij}(p)^2\int_{\R^n_+}\frac{x_i^2x_j^2x_n(x_n-\D_n)}{(\abs{\tilde x}^2+(x_n+\D_n)^2-1)^{n+1}}dx.
\end{equation*}
For every $i\neq j$, using polar coordinates in $\R^{n-3}$, we can see that
\begin{align*}
&\int_{\R^n_+}\frac{x_i^2x_j^2x_n(x_n-\D_n)}{(\abs{\tilde x}^2+(x_n+\D_n)^2-1)^{n+1}}dx = \int_0^{+\infty}\int_{-\infty}^{+\infty}\int_{-\infty}^{+\infty} x_i^2x_j^2x_n(x_n-\D_n)\\ &\times \int_0^{+\infty}\frac{\omega_{n-4}r^{n-4}dr}{(r^2+{x_i}^2+{x_j}^2+(x_n+\D_n)^2-1)^{n+1}}dx_idx_jdx_n \\ &= \frac{\omega_{n-4}\Gamma\left(\frac{n-3}{2}\right)\Gamma\left(\frac{n+5}{2}\right)}{2n!}\int_0^{+\infty}\int_{-\infty}^{+\infty}\int_{-\infty}^{+\infty} \frac{{x_i}^2{x_j}^2x_n(x_n-\D_n)}{(x_i^2+x_j^2+(x_n+\D_n)^2-1)^\frac{n+5}{2}}dx_idx_jdx_n \\ &=\frac{\omega_{n-4}\pi\Gamma\left(\frac{n-3}{2}\right)\Gamma\left(\frac{n+5}{2}\right)}{n!(n-1)(n+1)(n+3)}\underbrace{\int_0^{+\infty}\frac{x_n(x_n-\D_n)}{\left((x_n+\D_n)^2-1\right)^{\frac{n-1}{2}}}}_{\mathfrak{f}_1(n,\D_n)}.
\end{align*}
Substituting in \eqref{casongeq5-6}:
\begin{equation*}
I_{w_p}= \frac{\omega_{n-4}\pi{\beta_n}^2\Gamma\left(\frac{n-3}{2}\right)\Gamma\left(\frac{n+5}{2}\right)}{ 4n(n-1)(n+3)(n+1)!}\mathfrak{f}_1(n,\D_n)\norm{\pi(p)}^2.
\end{equation*}

Let us now study the term $I_{\psi_p}$. Multiplying \eqref{L-f} by $\psi_p$ and integrating by parts we obtain:
$$\begin{aligned}&\beta_n h^{ij}(p)\int_{\mathbb R^n_+}\frac{x_i x_j x_n }{(|\tilde x|^2+(x_n+\mathfrak D_n)^2-1)^{\frac{n+2}{2}}}\psi_p\, dx\\
&=\int_{\mathbb R^n_+}\left(\nabla v_p \nabla\psi_p +\frac{n(n+2)}{(|\tilde x|^2+(x_n+\mathfrak D_n)^2-1)^2}v_p \psi_p\right)\, dx -\int_{\partial\mathbb R^n_+}\frac{\partial v_p}{\partial \eta}\psi_p\, d\tilde x\\
&=\int_{\mathbb R^n_+}\left(\nabla w_p \nabla\psi_p +\frac{n(n+2)}{(|\tilde x|^2+(x_n+\mathfrak D_n)^2-1)^2}w_p \psi_p\right)\, dx\\
&+\int_{\mathbb R^n_+}\left(| \nabla\psi_p|^2 +\frac{n(n+2)}{(|\tilde x|^2+(x_n+\mathfrak D_n)^2-1)^2} \psi^2_p\right)\, dx -\int_{\partial\mathbb R^n_+}\frac{n\mathfrak D_n}{(|\tilde x|^2+\mathfrak D_n^2-1)}w_p\psi_p\, d\tilde x\\ &-\int_{\partial\mathbb R^n_+}\frac{n\mathfrak D_n}{(|\tilde x|^2+\mathfrak D_n^2-1)}\psi^2_p\, d\tilde x\\
&=\int_{\partial\mathbb R^n_+}\frac{\partial \psi_p}{\partial\eta}w_p-\int_{\partial\mathbb R^n_+}\frac{n\mathfrak D_n}{(|\tilde x|^2+\mathfrak D_n^2-1)}w_p\psi_p\, d\tilde x\\
&+\int_{\mathbb R^n_+}\left(| \nabla\psi_p|^2 +\frac{n(n+2)}{(|\tilde x|^2+(x_n+\mathfrak D_n)^2-1)^2} \psi^2_p\right)\, dx-\int_{\partial\mathbb R^n_+}\frac{n\mathfrak D_n}{(|\tilde x|^2+\mathfrak D_n^2-1)}\psi^2_p\, d\tilde x\\
&=\int_{\partial\mathbb R^n_+}\frac{n\mathfrak D_n}{(|\tilde x|^2+\mathfrak D_n^2-1)}w_p\psi_p\, d\tilde x +\int_{\partial\mathbb R^n_+}b(\tilde x, 0) w_p(\tilde x, 0)\, d\tilde x-\int_{\partial\mathbb R^n_+}\frac{n\mathfrak D_n}{(|\tilde x|^2+\mathfrak D_n^2-1)}w_p\psi_p\, d\tilde x\\
&+\int_{\mathbb R^n_+}\left(| \nabla\psi_p|^2 +\frac{n(n+2)}{(|\tilde x|^2+(x_n+\mathfrak D_n)^2-1)^2} \psi^2_p\right)\, dx-\int_{\partial\mathbb R^n_+}\frac{n\mathfrak D_n}{(|\tilde x|^2+\mathfrak D_n^2-1)}\psi^2_p\, d\tilde x
\end{aligned}$$
We will study separately the terms with and without $\psi_p$. By \eqref{wp}
$$b(\tilde x, 0)=\frac{\beta_n}{4n} \frac{\sum_{i, j=1\atop i\neq j}^{n-1}h^{ij}(p)x_ix_j }{(|\tilde x|^2+\mathfrak D_n^2-1)^{\frac{n}{2}}}$$
and arguing as before:
\begin{align*}
&	\int_{\partial\mathbb R^n_+}b(\tilde x, 0) w_p(\tilde x, 0)\, d\tilde x = -\frac{\D_n\beta_n^2}{16n^2} \sum_{i,j=1}^{n-1}h^{ij}(p)^2\int_{\R^{n-1}}\frac{{x_i}^2{x_j}^2}{(\abs{\tilde x}^2+{\D_n}^2-1)^n}d\tilde x.
\end{align*}
Now, for $i\neq j$ fixed,
\begin{align*}
	&\int_{\R^{n-1}}\frac{x_i^2x_j^2}{(\abs{\tilde x}^2+\D_n^2-1)^n}d\tilde x \\&= \int_{-\infty}^{+\infty}\int_{-\infty}^{+\infty}x_i^2x_j^2\int_0^{+\infty}\frac{\omega_{n-4}r^{n-4}dr}{(r^2+x_i^2+x_j^2+\D_n^2-1)}dx_i dx_j \\ &=\frac{\omega_{n-4}\Gamma\left( \frac{n-3}{2}\right)\Gamma\left(\frac{n+3}{2}\right)}{2(n-1)!}\int_{-\infty}^{+\infty}\int_{-\infty}^{+\infty}\frac{{x_i}^2{x_j}^2dx_idx_j}{(x_i^2+x_j^2+{\D_n}^2-1)^\frac{n+3}{2}}\\&=\frac{\omega_{n-4}\pi\Gamma\left( \frac{n-3}{2}\right)\Gamma\left(\frac{n+3}{2}\right)(\D_n^2-1)^\frac{3-n}{2}}{(n-1)!(n-3)(n-1)(n+1)}.
\end{align*}
Finally, 
\begin{equation*}
	\int_{\partial\mathbb R^n_+}b(\tilde x, 0) w_p(\tilde x, 0)\, d\tilde x = -\frac{\D_n{\beta_n}^2\omega_{n-4}\pi\Gamma\left( \frac{n-3}{2}\right)\Gamma\left(\frac{n+3}{2}\right)(\D_n^2-1)^\frac{3-n}{2}}{16n(n-1)(n-3)(n+1)!}\norm{\pi(p)}^2.
\end{equation*}

Finally we address the terms with $\psi_p$. 

Since $\psi_p$ solves \eqref{casongeq5-3} we can write
$$\psi_p=\frac{\beta_n}{4n}\sum\limits_{i, j=1\atop i\neq j}^{n-1}h^{ij}(p) \psi_{ij}$$
where $\psi_{ij}$ solves
\begin{equation*}
\left\{\begin{aligned}
&-\Delta \psi_{ij} + \frac{n(n+2)}{(|\tilde x|^2+(x_n+\mathfrak D_n)^2-1)^2}\psi_{ij}=0, \quad &\mbox{in}\,\, &\mathbb R^n_+\\
&\frac{\partial \psi_{ij}}{\partial \eta}-\frac {n \mathfrak D_n}{(|\tilde x|^2+\mathfrak D_n^2-1)} \psi_{ij}= \frac{x_ix_j }{(|\tilde x|^2+\mathfrak D_n^2-1)^{\frac{n}{2}}}\quad &\mbox{on}\,\, &\partial\mathbb R^n_+\end{aligned}\right.\end{equation*}
It is not difficult to check that $\psi_{ij}$ is odd in $x_i$ and $x_j$ and even in all the other variables $x_\ell$, $\ell=1,\dots,n-1$, and so
\begin{equation}\label{casongeq5-7}
\int\limits_{\mathbb R^{n-1}}\frac{x_ix_j }{(|\tilde x|^2+\mathfrak D_n^2-1)^{\frac{n}{2}}}\psi_{\ell\kappa}(\tilde x,0)d\tilde x=0\ \hbox{if}\ (i,j)\not=(\ell,\kappa).
\end{equation}
Moreover it holds that $\psi_{ij}=\psi_{ji}$ and $\psi_{ij}=\psi_{12}(\sigma_{ij}x)$, where $\sigma_{ij}$ permutes the $x_i$ and $x_j$ variables, i.e.
$$\sigma_{ij}(x_1,\dots,x_i,\dots,x_j,\dots,x_n)=(x_1,\dots,x_j,\dots,x_i,\dots,x_n).$$
Multiplying by $\psi_p$ in \eqref{casongeq5-3}, integrating by parts and using \eqref{casongeq5-7} we immediately see that
\begin{align*}
	&\int_{\R^n_+}\abs{\nabla\psi_p}^2+\int_{\R^n_+}\frac{n(n+2)}{(\abs{\tilde x}^2+(x_n+\D_n)^2-1)^2}\psi_p ^2 - \int_{\de\R^n_+}\frac{n\D_n}{\abs{\tilde x}^2+(x_n+\D_n)^2-1}\psi_p^2 \\ &= \frac{\beta_n}{4n} \int_{\de \R^n_+}\frac{h^{ij}(p)x_ix_j}{(\abs{\tilde x}^2+(x_n+\D_n)^2-1)^\frac{n}{2}}\psi_p 
\\&=\frac{\beta_n}{4n}
\sum\limits_{i,j=1\atop i\not=j}^{n-1}h^{ij}(p)\sum\limits_{\ell,\kappa=1\atop \ell\not=\kappa}^{n-1} h^{\ell\kappa}(p)\int\limits_{\mathbb R^{n-1}}\frac{x_ix_j }{(|\tilde x|^2+\mathfrak D_n^2-1)^{\frac{n}{2}}}\psi_{\ell\kappa}(\tilde x,0)d\tilde x\\
&=\sum\limits_{i,j=1\atop i\not=j}^{n-1}h^{ij}(p)^2 \int\limits_{\mathbb R^{n-1}}\frac{x_ix_j }{(|\tilde x|^2+\mathfrak D_n^2-1)^{\frac{n}{2}}}\psi_{ij}(\tilde x,0)d\tilde x\\
&=\sum\limits_{i,j=1\atop i\not=j}^{n-1}h^{ij}(p)^2\underbrace{\left(\ \int\limits_{\mathbb R^{n-1}}\frac{x_1x_2 }{(|\tilde x|^2+\mathfrak D_n^2-1)^{\frac{n}{2}}}\psi_{12}(\tilde x,0)d\tilde x\right)}_{\mathfrak t_2(n,\D_n)}\\
&=\mathfrak f_2(n,\D_n)\|\pi(p)\|^2,
\end{align*}
where $\mathfrak f_2$ only depends on $n$ and $\D_n$ because the functions $\psi_{ij}$ do not depend on the point $p$.

\medskip 

Collecting all the previous estimates, we get a constant $\mathfrak f_n$, which only depends on $n$ and $\D_n$, such that 
\begin{equation*}
	\int_{\R^n_+}\left(-c_n\Delta V_p+\frac{n+2}{n-2}\abs{K}U^{4\over n-2}V_p\right)V_p = \mathfrak{f}_n \norm{\pi(p)}^2.
\end{equation*}
Notice that $\mathfrak f_n$ needs to be nonnegative because $V_p$ satisfies Proposition \ref{vp}-(iv).
\end{proof}

\begin{proof}[\bf Proof of Proposition \ref{energy-bubble}]
We write $\mathcal W_p=\chi\left(\left(\psi_{p}^\partial\right)^{-1}(\xi)\right)W$ with
 $$ W(\xi):=\underbrace{\frac{1}{\delta^{\frac{n-2}{2}}}U\left(\frac{\left(\psi_{p}^\partial\right)^{-1}(\xi)}{\delta}\right)}_{:= U}+\delta \underbrace{\frac{1}{\delta^{\frac{n-2}{2}}}V_p\left(\frac{\left(\psi_{p}^\partial\right)^{-1}(\xi)}{\delta}\right)}_{:=V}. $$  
and also $\mathcal W_p=\mathcal W=\mathcal U+\delta\mathcal V$ with
$$\mathcal U (\xi)=\chi\left(\left(\psi_{p}^\partial\right)^{-1}(\xi)\right)U (\xi)\ \hbox{and}\ \mathcal V(\xi)=\chi\left(\left(\psi_{p}^\partial\right)^{-1}(\xi)\right)V (\xi).$$
We have
 $$\begin{aligned}
J_\e(\mathcal W)=&\underbrace{\frac{c_n}{2}\int_M |\nabla_g(\mathcal U+\delta  \mathcal V)|^2}_{I_1}+\underbrace{\frac 12 \int_M \mathcal S_g (\mathcal U+\delta \mathcal V)^2}_{I_2}+\underbrace{(n-1)\e\int_{\partial M} (\mathcal U+\delta\mathcal V)^2}_{I_3}\\
 &\underbrace{-(n-2)\int_{\partial M} H \left[\left((\mathcal U +\delta  \mathcal V )^+\right)^{\frac{2(n-1)}{n-2}}-\mathcal U ^{\frac{2(n-1)}{n-2}}\right]}_{I_4}\underbrace{-(n-2)\int_{\partial M}H \mathcal U ^{\frac{2(n-1)}{n-2}}}_{I_5}\\
 &\underbrace{-\frac{n-2}{2n}\int_M K\left[\left((\mathcal U +\delta  \mathcal V )^+\right)^{\frac{2n}{n-2}}-\mathcal U ^{\frac{2n}{n-2}}\right]}_{I_6}\underbrace{-\frac{n-2}{2n}\int_M K \mathcal U  ^{\frac{2n}{n-2}}}_{I_7}
 \end{aligned}$$
{\em Estimate of $I_2$} By \eqref{Jalpham} (with $\alpha=0$ and $m=n-2$) and \eqref{In}, if $n\geq 5$
 $$\begin{aligned}
 I_2&:=\frac 12\d^2\int_{\mathbb R^n_+}\mathcal S_g(\delta x)\left(U(x)\chi(\delta x)+\delta V_p(x)\chi(\delta x)\right)^2|g(\delta x)|^{\frac 12}\, dx\\
 &=\frac 12 \delta^2\mathcal S_g(p)\int_{\mathbb R^n_+}U^2(x)\, dx+\left\{\begin{aligned}&\mathcal O(\delta^3)\quad &\hbox{if}\,\, & n\geq 6\\ & \mathcal O(\delta^3|\ln \d|)\quad &\hbox{if}\,\, & n=5\end{aligned}\right.\\
 &=\frac 12 \delta^2\frac{\mathcal S_g(p)}{|K|^{\frac{n-2}{2}}}\int_{\mathbb R^n_+}\frac{\alpha_n^2}{\left(|\tilde x|^2+(x_n+\mathfrak D_n)^2-1\right)^{n-2}}\, dx+\left\{\begin{aligned}&\mathcal O(\delta^3)\quad &\hbox{if}\,\, & n\geq 6\\ & \mathcal O(\delta^3|\ln \d|)\quad &\hbox{if}\,\, & n=5\end{aligned}\right.\\
 &=\delta^2 \frac 12 \alpha_n^2 \omega_{n-1}\frac{2(n-2)}{n-1}I_{n-1}^n \frac{\mathcal S_g(p)}{|K|^{\frac{n-2}{2}}}\varphi_{\frac{n-3}{2}}+\left\{\begin{aligned}&\mathcal O(\delta^3)\quad &\hbox{if}\,\, & n\geq 6\\ & \mathcal O(\delta^3|\ln \d|)\quad &\hbox{if}\,\, & n=5\end{aligned}\right. \end{aligned}$$
and if $n=4$ 
 $$\begin{aligned}
 I_2&=\frac{\alpha_4^2}{2} \d^2\frac{\mathcal S_g(p)}{|K|}\int_{B^+_{\frac R \d}}\frac{1}{\left(|\tilde x|^2+(x_4+\mathfrak D_4(p))^2-1\right)^{2}}\, dx+\mathcal O(\d^2)\\
 &=\frac{\alpha_4^2\omega_3}{2} \d^2\frac{\mathcal S_g(p)}{|K|}I_2^2 \int_{\sqrt{\mathfrak D_4^2(p)-1}}^{\sqrt{\left(\frac R \d +\mathfrak D_4(p)\right)^2-1}}\frac{1}{\sqrt{\lambda^2+1}}\, d\lambda+\mathcal O(\d^2)\\
 &= -\frac{2\alpha_4^2\omega_3}{3} \frac{\mathcal S_g(p)}{|K|}I_3^4\d^2\ln \d +\mathcal O(\d^2)\\\end{aligned}$$
{\em Estimate of $I_3$} By
 \eqref{JalphamI} (with $\alpha=0$ and $m=n-2$) and \eqref{In}, if $n\geq 4$ that
\begin{equation*}\begin{split}
 I_3&:=(n-1)\e \delta\int_{\mathbb R^{n-1}}\left(U(\tilde x, 0)\chi(\delta\tilde x, 0)+\delta V_p (\tilde x, 0)\chi(\delta\tilde x, 0)\right)^2|g(\delta\tilde x, 0)|^{\frac 12}\, d\tilde x\\
 &=(n-1)\e \delta\int_{\mathbb R^{n-1}}U^2(\tilde x, 0)\, d\tilde x +\left\{\begin{aligned}&\mathcal O(\e \delta^2)\quad &\hbox{if}\,\, &n \geq 5\\ &\mathcal O(\e \d^2|\ln \d|)\quad &\hbox{if}\,\, &n=4.\end{aligned}\right.
\\
 &=(n-1)\e \delta\frac{1}{|K|^{\frac{n-2}{2}}} \alpha_n^2\int_{\mathbb R^{n-1}}\frac{1}{\left(|\tilde x|^2+\mathfrak D_n^2-1\right)^{n-2}}\, d\tilde x+\left\{\begin{aligned}&\mathcal O(\e \delta^2)\quad &\hbox{if}\,\, &n \geq 5\\ &\mathcal O(\e \d^2|\ln \d|)\quad &\hbox{if}\,\, &n=4.\end{aligned}\right.
\\
 &=\e\d\underbrace{ {2(n-2)\omega_{n-1}\alpha_n^2\frac{1}{|K|^{\frac{n-2}{2}}\left(\mathfrak D_n^2-1\right)^{\frac{n-3}{2}}}I^{n}_{n-1}}}_{:={\mathfrak c}_n} +\left\{\begin{aligned}&\mathcal O(\e \delta^2)\quad &\hbox{if}\,\, &n \geq 5\\ &\mathcal O(\e \d^2|\ln \d|)\quad &\hbox{if}\,\, &n=4.\end{aligned}\right.
 \end{split}\end{equation*}

{\em Estimate of $I_5$}
By \eqref{JalphamI} (with $\alpha=2$ and $m=n-1$) and Lemma \ref{metricexpansion}, if $n\geq 4$ 
 $$\begin{aligned} I_5&:=-(n-2)\int_{\mathbb R^{n-1}} H \left(U(\tilde x, 0)\chi(\delta\tilde x, 0)\right)^{\dsh}|g(\delta \tilde x, 0)|^{\frac 12}\, d\tilde x\\
 &=-(n-2)H\int_{\partial\mathbb R^{n}_+}U^{\dsh}(\tilde x, 0)\, d\tilde x +\frac{n-2}{6}\delta^2\overline R_{ij}(p)H\int_{\mathbb R^{n-1}}U^{\dsh}(\tilde x, 0)\tilde x_i\tilde x_j\, d\tilde x+\mathcal O(\delta^3)\\
 &=-(n-2)H\int_{\partial\mathbb R^{n}_+}U^{\dsh}(\tilde x, 0)\, d\tilde x\\
 &+\delta^2\frac{\alpha_n^{\dsh}(n-2)}{6(n-1)}\frac{H\overline R_{ii}(p)}{|K|^{\frac{n-1}{2}}}\int_{\mathbb R^{n-1}}\frac{|\tilde x|^2}{\left(|\tilde x|^2+\mathfrak D_n^2-1\right)^{n-1}}\, d\tilde x+\mathcal O(\delta^3)\\
 &=-(n-2)H\int_{\partial\mathbb R^{n}_+}U^{\dsh}(\tilde x, 0)\, d\tilde x\\
 &+\delta^2\frac{\alpha_n^{\dsh}(n-2)}{6(n-1)}\omega_{n-1}\frac{H\overline R_{ii}(p)}{|K|^{\frac{n-1}{2}}(\mathfrak D_n^2-1)^{\frac{n-3}{2}}}I_{n-1}^n+\mathcal O(\delta^3).\\
  \end{aligned}$$
 
 {\em Estimate of $I_7$} By \eqref{Jalpham} (with $\alpha=2$ and $m=n$), \eqref{JalphamII} (with $\alpha=0$ and $m=n$), Lemma \ref{metricexpansion}  and \eqref{In}, if $n\geq 4$
$$\begin{aligned} I_7&:=-\frac{n-2}{2n}\int_{\mathbb R^n_+} K \left(U(x)\chi(\delta x)\right)^{\dst}|g(\delta x)|^{\frac 12}\, dx\\
&=\frac{n-2}{2n}|K|\int_{\mathbb R^n_+}U^{\dst}(x)\, dx-\frac{n-2}{4n}\left(\|\pi(p)\|^2+{\rm Ric}_\nu(p)\right)\delta^2|K|\int_{\mathbb R^n_+}U^{\dst}x_n^2\, dx\\
&-\frac{n-2}{12n}\overline{R}_{ij}(p)\delta^2 |K|\int_{\mathbb R^n_+}U^{\dst}(x)\tilde x_i\tilde x_j\, dx+\mathcal O(\delta^3)\\
&=\frac{n-2}{2n}|K|\int_{\mathbb R^n_+}U^{\dst}(x)\, dx\\
&-\delta^2\frac{n-2}{4n}\alpha_n^{\dst}\frac{\left(\|\pi(p)\|^2+{\rm Ric}_\nu(p)\right)|K|}{|K|^{\frac n 2}}\int_{\mathbb R^n_+}\frac{x_n^2}{(|\tilde x|^2+(x_n+\mathfrak D_n)^2-1)^n}\, dx\\
&-\delta^2\alpha_n^{\dst}\frac{n-2}{12n(n-1)}\frac{\overline{R}_{ii}(p) |K|}{|K|^{\frac n 2}}\int_{\mathbb R^n_+}\frac{|\tilde x|^2}{(|\tilde x|^2+(x_n+\mathfrak D_n)^2-1)^n}\, dx+\mathcal O(\delta^3)\\
&=\frac{n-2}{2n}|K|\int_{\mathbb R^n_+}U^{\dst}(x)\, dx\\
&-\delta^2\frac{n-2}{4n}\frac{n-3}{2(n-1)}I_{n-1}^n\omega_{n-1}\alpha_n^{\dst}\frac{\left(\|\pi(p)\|^2+{\rm Ric}_\nu(p)\right)}{|K|^{\frac{ n-2}{ 2}}}{\hat\varphi}_{\frac{n+1}{2}}\\
&-\delta^2\alpha_n^{\dst}\frac{n-3}{2(n-1)}I_{n-1}^n\omega_{n-1}\frac{n-2}{12n(n-1)}\frac{\overline{R}_{ii}(p)}{|K|^{\frac{ n-2}{ 2}}}\varphi_{\frac{n-1}{2}}+\mathcal O(\delta^3)
\end{aligned}$$

{\em Estimate of $I_4$ and $I_6$} By Lemma \ref{metricexpansion}-(i), if $n\geq 4$
$$\begin{aligned} I_4&=-(n-2)\int_{\partial \mathbb R^n_+}H\left[\left((U+\delta V_p)^+\right)^{\frac{2(n-1)}{n-2}}-U^{\frac{2(n-1)}{n-2}}\right](\tilde x, 0)|g(\delta\tilde x, 0)|^{\frac 12}\, d\tilde x\\
&=-2(n-1)\delta H\int_{\partial \mathbb R^n_+}U^{\frac{n}{n-2}}V_p\, d\tilde x -\frac{n(n-1)}{n-2}\delta^2 H\int_{\partial \mathbb R^n_+}U^{\frac{2}{n-2}}V_p^2\, d\tilde x+\mathcal O(\delta^3)
\end{aligned}$$
and similarly
 $$\begin{aligned}
 I_6&= |K|\delta\int_{\mathbb R^n_+}U^{\frac{n+2}{n-2}}V_p+\frac{n+2}{2(n-2)}\delta^2|K|\int_{\mathbb R^n_+}U^{\frac{4}{n-2}}V_p^2+\mathcal O(\delta^3).
 \end{aligned}$$
 
 {\em Estimate of $I_1$}
 First we have 
 $$I_1:=\underbrace{\frac{c_n}{2}\int_M|\nabla_g \Ud|^2\,d\nu_g}_{:=I_1^1}+\underbrace{c_n\delta\int_M \nabla_g\Ud\nabla_g\Vd\, d\nu_g}_{:=I_1^2}+\underbrace{\frac{c_n}{2}\delta^2\int_M |\nabla_g\Vd|^2\, d\nu_g}_{:=I_1^3}$$ and we   separately estimate the terms $I_1^i$ with $i=1, 2, 3$. Set $B_\delta:=\{x\in \mathbb R^n_+\ :\   |\delta x|\leq R \}$.\\
 
{\em Estimate of $I_1^1$} By  Lemma \ref{metricexpansion}-(ii)-(iii) we get
 $$\begin{aligned}
 I_1^1&=\frac{c_n}{2}\int_{\mathbb R^n_+}g^{ab}(\delta x)\frac{\partial}{\partial x_a}(U(x)\chi(\delta x))\frac{\partial}{\partial x_b}(U(x)\chi(\delta x))|g(\delta x)|^{\frac 12}\, dx\\
 &=c_n\int_{B_\delta}\left[\frac{|\nabla U|^2}{2}+\left(\delta h_{ij}x_n +\frac{\delta^2}{6} \overline R_{ikj\ell}x_kx_\ell +\delta^2 \frac{\partial h_{ij}}{\partial x_k}x_k x_n+\frac{\delta^2}{2}\left(R_{injn}+3h_{ik}h_{kj}\right)x_n^2\right)\frac{\partial U}{\partial x_i}\frac{\partial U}{\partial x_j}\right]\\
 &\quad\quad\times \left(1-\frac{\delta^2}{2}\left(\|\pi\|^2+{\rm Ric}_\nu\right)x_n^2-\frac{\delta^2}{6} \overline R_{\ell m}x_\ell x_m\right)\, d\tilde x\, d x_n +\mathcal O(\delta^3)\\
 &=\int_{B_\delta}\left(\frac{c_n}{2}|\nabla U|^2+\delta^2\frac{c_n}{2}\left(R_{injn}+3 h_{ik}h_{kj}\right)x_n^2\frac{\partial U}{\partial x_i}\frac{\partial U}{\partial x_j}\right)\\
 &\qquad\times \left(1-\frac{\delta^2}{2}\left(\|\pi\|^2+{\rm Ric}_\nu\right)x_n^2-\frac{\delta^2}{6} \overline R_{\ell m}x_\ell x_m\right)\, d\tilde x\, d x_n +\mathcal O(\delta^3).\\
 \end{aligned}$$
Moreover, by  \eqref{JalphamII} (with $\alpha=0$ or $\alpha=2$ and $m=n-1$ and with $\alpha=0$ and $m=n$ secondly), \eqref{Jalpham} (with $\alpha=2$ and $m=n-1$ first and with $\alpha=2$ and $m=n$ secondly) and \eqref{In}, if $n\geq5$
$$\begin{aligned}
 I_1^1&=\frac{c_n}{2}\int_{\mathbb R^n_+}|\nabla U|^2-\frac{c_n}{4}\delta^2\left(\|\pi(p)\|^2+{\rm Ric}_\nu(p)\right)\int_{\mathbb R^n_+}x_n^2|\nabla U|^2-\frac{c_n}{12}\delta^2\frac{\overline R_{\ell\ell}(p)}{n-1}\int_{\mathbb R^n_+}|\tilde x|^2|\nabla U|^2\\
 &+\frac{c_n\alpha_n^2}{2}\frac{(n-2)^2}{n-1}\delta^2\frac{\left(3\|\pi(p)\|^2+{\rm Ric}_\nu(p)\right)}{|K|^{\frac{n-2}{2}}}\int_{\mathbb R^n_+}\frac{x_n^2|\tilde x|^2}{(|\tilde x|^2+(x_n+\mathfrak D_n)^2-1)^n}\, dx+\mathcal O(\delta^3)\\
 &=\frac{c_n}{2}\int_{\mathbb R^n_+}|\nabla U|^2-\frac{c_n\alpha_n^2(n-2)^2}{4}\d^2\frac{(\|\pi(p)\|^2+{\rm Ric}_\nu(p))}{|K|^{\frac{n-2}{2}}}\int_{\mathbb R^n_+}\frac{x_n^2}{(|\tilde x|^2+(x_n+\mathfrak D_n)^2-1)^{n-1}}\, dx\\
 &-\frac{c_n\alpha_n^2(n-2)^2}{4}\d^2\frac{(\|\pi(p)\|^2+{\rm Ric}_\nu(p))}{|K|^{\frac{n-2}{2}}}\int_{\mathbb R^n_+}\frac{x_n^2}{(|\tilde x|^2+(x_n+\mathfrak D_n)^2-1)^{n}}\\
 &-\frac{c_n\alpha_n^2(n-2)^2}{12(n-1)}\d^2\frac{\overline R_{\ell \ell}(p)}{|K|^{\frac{n-2}{2}}}\int_{\mathbb R^n_+}\frac{|\tilde x|^2}{(|\tilde x|^2+(x_n+\mathfrak D_n)^2-1)^{n-1}}\\
 &-\frac{c_n\alpha_n^2(n-2)^2}{12(n-1)}\d^2\frac{\overline R_{\ell\ell}(p)}{|K|^{\frac{n-2}{2}}}\int_{\mathbb R^n_+}\frac{|\tilde x|^2}{(|\tilde x|^2+(x_n+\mathfrak D_n)^2-1)^{n}}\\
 &+\frac{c_n\alpha_n^2}{2}\frac{(n-2)^2}{n-1}\delta^2\frac{\left(3\|\pi(p)\|^2+{\rm Ric}_\nu(p)\right)}{|K|^{\frac{n-2}{2}}}\int_{\mathbb R^n_+}\frac{x_n^2|\tilde x|^2}{(|\tilde x|^2+(x_n+\mathfrak D_n)^2-1)^n}\, dx+\mathcal O(\delta^3)\\
 &=\frac{c_n}{2}\int_{\mathbb R^n_+}|\nabla U|^2-\frac{c_n\alpha_n^2(n-2)^2(n-3)}{4(n-1)}\omega_{n-1}I_{n-1}^n\d^2\frac{(\|\pi(p)\|^2+{\rm Ric}_\nu(p))}{|K|^{\frac{n-2}{2}}}{\hat\varphi}_{\frac{n-1}{2}}\\
 &-\frac{c_n\alpha_n^2(n-2)^2(n-3)}{8(n-1)}\omega_{n-1}I_{n-1}^n\d^2\frac{(\|\pi(p)\|^2+{\rm Ric}_\nu(p))}{|K|^{\frac{n-2}{2}}}{\hat\varphi}_{\frac{n+1}{2}}\\
 &-\frac{c_n\alpha_n^2(n-2)^2}{12(n-1)}\omega_{n-1}I_{n-1}^n\d^2\frac{\overline R_{\ell \ell}(p)}{|K|^{\frac{n-2}{2}}}\varphi_{\frac{n-3}{2}}-\frac{c_n\alpha_n^2(n-2)^2(n-3)}{24(n-1)^2}\omega_{n-1}I_{n-1}^n\d^2\frac{\overline R_{\ell\ell}(p)}{|K|^{\frac{n-2}{2}}}\varphi_{\frac{n-1}{2}}\\
&+\frac{c_n\alpha_n^2}{4}\frac{(n-2)^2(n-3)}{(n-1)^2}\omega_{n-1}I_{n-1}^n\delta^2\frac{\left(3\|\pi(p)\|^2+{\rm Ric}_\nu(p)\right)}{|K|^{\frac{n-2}{2}}}{\hat\varphi}_{\frac{n-1}{2}}+\mathcal O(\delta^3)\\
 \end{aligned}$$
and if $n=4$
 $$\begin{aligned}I_1^1&= \frac{c_4}{2}\int_{\mathbb R^4_+}|\nabla U|^2-c_4\alpha_4^2\frac{\left(\|\pi(p)\|^2+{\rm Ric}_\nu(p)\right)}{|K|}\d^2 \int_{B_\delta}\frac{x_4^2}{(|\tilde x|^2+(x_4+\mathfrak D_4)^2-1)^3}\, dx\\
 &-\frac 1 9 c_4\alpha_4^2\frac{\overline R_{\ell\ell}(p)}{|K|}\d^2 \int_{B_\d}\frac{|\tilde x|^2}{(|\tilde x|^2+(x_4+\mathfrak D_4)^2-1)^3}\, dx\\
 &+\frac 2 3c_4\alpha_4^2\frac{\left(3\|\pi(p)\|^2+{\rm Ric}_\nu(p)\right)}{|K|}\d^2\int_{B_\d}\frac{x_4^2|\tilde x|^2}{(|\tilde x|^2+(x_4+\mathfrak D_4)^2-1)^4}\, dx+\mathcal O(\d^2)\\
 &= \frac{c_4}{2}\int_{\mathbb R^4_+}|\nabla U|^2+\frac 13 c_4\omega_3\alpha_4^2I_3^4\frac{\left(\|\pi(p)\|^2+{\rm Ric}_\nu(p)\right)}{|K|}\d^2\ln \d +\frac 1 9 c_4\omega_3 \alpha_4^2I_3^4 \frac{\overline R_{\ell\ell}(p)}{|K|}\d^2\ln \d \\
 &-\frac{1}{54} c_4\omega_3\alpha_4^2\frac{\left(3\|\pi(p)\|^2+{\rm Ric}_\nu(p)\right)}{|K|}I_3^4\d^2\ln \d+\mathcal O(\d^2).
\end{aligned}$$
%

{\em Estimate of $I_1^2$ and $I_1^3$ if $n\geq5$}

We have
$$\begin{aligned} I_1^2&=\delta c_n \int_{\mathbb R^n_+}g^{ab}(\delta x)\frac{\partial}{\partial x_a}(U(x)\chi(\delta x))\frac{\partial}{\partial x_b}(V_p(x)\chi(\delta x))|g(\delta x)|^{\frac12}\, dx\\
&=\delta c_n \int_{\mathbb R^n_+}\nabla U \nabla V_p\, dx +\delta^2 2 c_n h^{ij}(p)\int_{\mathbb R^n_+}x_n \frac{\partial U}{\partial x_i}\frac{\partial V_p}{\partial x_j}\, dx+\mathcal O(\delta^3)\\
&=-\delta \frac{n+2}{n-2}|K|\int_{\mathbb R^n_+}U^{\frac{n+2}{n-2}}V_p\, dx+c_n \frac n 2 \delta H\int_{\partial \mathbb R^n_+}U^{\frac{n}{n-2}}V_p\\
&-c_n\delta^2\underbrace{\int_{\mathbb R^n_+}|\nabla V_p|^2}_{<\infty\,\hbox{\tiny for}\,\, n\geq 5}+c_n\frac n 2 \delta^2H\int_{\partial\mathbb R^n_+}U^{\frac{2}{n-2}}V_p^2-\frac{n+2}{n-2}|K|\delta^2\int_{\mathbb R^n_+}U^{\frac{4}{n-2}}V_p^2+\mathcal O(\delta^3)\end{aligned}$$ since
$$\begin{aligned} c_n\int_{\mathbb R^n_+}\nabla U\nabla V_p &=-c_n \int_{\mathbb R^n_+} U\Delta V_p+c_n\int_{\partial\mathbb R^n_+} U(\tilde x, 0)\frac{\partial V_p}{\partial \nu}\\
&=-\frac{n+2}{n-2}|K|\int_{\mathbb R^n_+}U^{\frac{n+2}{n-2}}V_p\, dx+ c_n\frac n 2 H\int_{\partial \mathbb R^n_+}U^{\frac{n}{n-2}}(\tilde x, 0)V_p\, d\tilde x\\
&+\underbrace{8 \frac{n-1}{n-2}\int_{\mathbb R^n_+}h^{ij}(p)\frac{\partial^2 U}{\partial x_i\partial x_j}x_n U}_{=0}\end{aligned}$$
and
$$\begin{aligned} 2c_n h^{ij}(p)\int_{\mathbb R^n_+}x_n \frac{\partial U}{\partial x_i}\frac{\partial V_p}{\partial x_j}\, dx &= -2 c_n h^{ij}(p)\int_{\mathbb R^n_+}x_n \frac{\partial^2 U}{\partial x_i\partial x_j}V_p\, dx\\
&=-\int_{\mathbb R^n_+}\left(-c_n\Delta V_p+\frac{n+2}{n-2}|K| U^{\frac{4}{n-2}}V_p\right) V_p\, dx\\
&=-c_n\int_{\mathbb R^n_+}|\nabla V_p|^2+c_n\int_{\partial\mathbb R^n_+}V_p\frac{\partial V_p}{\partial\nu}-\frac{n+2}{n-2}|K|\int_{\mathbb R^n_+}U^{\frac{4}{n-2}}V_p^2\\
&=-c_n\int_{\mathbb R^n_+}|\nabla V_p|^2+c_n\frac n 2 H\int_{\partial\mathbb R^n_+}U^{\frac{2}{n-2}}V_p^2-\frac{n+2}{n-2}|K|\int_{\mathbb R^n_+}U^{\frac{4}{n-2}}V_p^2\end{aligned}$$
and also
$$I_1^3:=\frac{c_n}{2}\delta^2\int_{\mathbb R^n_+}|\nabla V_p|^2+\mathcal O(\delta^3).$$

{\em Estimate of $I_1^2$ and $I_1^3$ if $n=4$}
 
Let $V_p=\bar w_p+\zeta_p+\psi_p$ as in \eqref{dec}  and set   $w_p=\bar w_p+\zeta_p$ so that $w_p$ solves the problem 
\beq\label{wp4}
\left\{\begin{aligned} &-6\Delta w_p=\mathtt E_p(x)\quad &\mbox{in}\,\, &\mathbb R^4_+\\
&\frac{\partial w_p}{\partial \nu}=2 H U w_p,\quad &\mbox{on}\,\,\ &\partial\mathbb R^4_+\end{aligned}\right.\eeq
Then
\begin{equation}\label{I12-1}
\begin{split} 
I_1^2&=\delta c_4 \int_{\mathbb R^4_+}g^{ab}(\delta x)\frac{\partial}{\partial x_a}(U(x)\chi(\delta x))\frac{\partial}{\partial x_b}(V_p(x)\chi(\delta x))|g(\delta x)|^{\frac12}\, dx\\
&=-\delta \frac{n+2}{n-2}|K|\int_{\mathbb R^4_+}U^{3}V_p\, dx+2c_4  \delta H\int_{\partial \mathbb R^4_+}U^{2}V_p\\
&+\delta^2 12 h^{ij}(p)\int_{B^+_{\frac R \delta}}x_4 \frac{\partial U}{\partial x_i}\frac{\partial V_p}{\partial x_j}\, dx+\mathcal O(\d^2)\\
&=-\delta \frac{n+2}{n-2}|K|\int_{\mathbb R^4_+}U^{3}V_p\, dx+2c_4  \delta H\int_{\partial \mathbb R^4_+}U^{2}V_p\\
&+\delta^2 12 h^{ij}(p)\int_{B^+_{\frac R \delta}}x_4 \frac{\partial U}{\partial x_i}\frac{\partial w_p}{\partial x_j}\, dx+\mathcal O(\d^2).
\end{split}
\end{equation} 
Here we have used the fact that
$$\int_{B^+_{\frac R \delta}}x_4 \frac{\partial U}{\partial x_i}\frac{\partial \psi_p}{\partial x_j}\, dx=\mathcal O(1).$$
Let us study the last integral term of \eqref{I12-1}. Integrating by parts in $x_j$ and using the equation \eqref{wp4}:
\begin{equation}\label{I12-2}
\begin{split}
12h^{ij}(p)\int_{B^+_{\frac R \delta}}x_4\frac{\de U}{\de x_i}\frac{\de w_p}{\de x_j}dx &= 12h^{ij}(p)\int_{\de^+B^+_{\frac R \delta}}x_4\frac{\de U}{\de x_i}\frac{x_j}{\abs{x}}w_p
 - \int_{B^+_{\frac R \delta}}\mathtt E_p w_p \\ &=12h^{ij}(p)\int_{\de^+B^+_{\frac R \delta}}x_4\frac{\de U}{\de x_i}\frac{x_j}{\abs{x}}w_p-6\int_{B^+_{\frac R \delta}}\abs{\nabla w_p}^2 \\ &+ 12\int_{\de'B^+_{\frac R \delta}}HU w_p^2+6\int_{\de^+B^+_{\frac R \delta}}\nabla w_p \cdot \frac{x}{\abs{x}} w_p \\ &= -6\int_{B^+_{\frac R \delta}}\abs{\nabla w_p}^2 + \mathcal O(1),
\end{split}
\end{equation}
since
\begin{align*}
\int_{\de'B^+_{\frac R \delta}}HUw_p^2 &= \mathcal O(1), \\ \int_{\de^+B^+_{\frac R \delta}}x_4\frac{\de U}{\de x_i}\frac{x_j}{\abs{x}}w_p &\lesssim \frac{R}{\delta} \left(1+\frac{1}{\delta}\right)^{-3}\left(1+\frac{1}{\delta}\right)^{-1}\frac{\omega_{n-1}}{2}\frac{R^3}{\delta^3}=\mathcal O(1), \\ \int_{\de^+B^+_{\frac R \delta}}\nabla w_p \cdot \frac{x}{\abs{x}} w_p &\lesssim \left(1+\frac{R}{\delta}\right)^{-2}\left(1+\frac{R}{\delta}\right)^{-1}\frac{\omega_{n-1}}{2}\frac{R^3}{\delta^3} = \mathcal O(1).
\end{align*}
By \eqref{I12-1} and \eqref{I12-2},
\begin{equation}\label{I12-3}
	I_1^2 = -6\delta^2\int_{B^+_{\frac R \delta}}\abs{\nabla w_p}^2 +\mathcal O(\d^2).
\end{equation}
Analogously,
\begin{equation}\label{I13-1}
I_1^3:=\frac{c_4}{2}\delta^2\int_{B^+_{\frac R \delta}}|\nabla V_p|^2+\mathcal O(\delta^2)=3\delta^2\int_{B^+_{\frac R \delta}}|\nabla w_p|^2+\mathcal O(\delta^2).
\end{equation}
Finally, combining \eqref{I12-3} and \eqref{I13-1}, we obtain
\begin{equation}\label{I12-I13}
	I_1^2+I_1^3 = -3\delta^2\int_{B^+_{\frac R \delta}}\abs{\nabla w_p}^2 + \bigo{\delta^2}.
\end{equation}
Now, using the fact that $w_p=\bar w_p+\zeta_p$ and the decay estimate \eqref{caso4zeta} we get 
\begin{equation*}
\begin{split}
\int_{B^+_{\frac R\delta}} |\nabla w_p|^2&=\int_{B^+_{\frac R\delta}}|\nabla \bar w_p|^2 +\int_{B^+_{\frac R\delta}}|\nabla \zeta_p|^2+\int_{B^+_{\frac R\delta}}\nabla \bar w_p \nabla\zeta_p\\
&=\int_{B^+_{\frac R\delta}}|\nabla \bar w_p|^2 +\int_{B^+_{\frac R\delta}}|\nabla \zeta_p|^2+ \int_{\partial^+B^+_{\frac R \delta}}\nabla\zeta_p \cdot\frac{x}{|x|}\bar w_p\\ &+ \int_{\de'B^+_{\frac R\delta}}\left(2 H U\zeta_p + \left(2H U\bar w_p -\frac{\partial \bar w_p}{\partial \nu}\right)\right)\bar w_p\\
&=\int_{B^+_{\frac R\delta}}|\nabla \bar w_p|^2+\mathcal O(1)\\
\end{split}
\end{equation*}
since again, by using the decay properties, we get
$$\int_{\de'B^+_{\frac R\delta}}\left(2 H U\zeta_p + \left(2H U\bar w_p -\frac{\partial \bar w_p}{\partial \nu}\right)\right)\bar w_p=\mathcal O(1)$$ and
$$\int_{\de^+B^+_{\frac R\delta}}\nabla \zeta_p\cdot \frac{x}{\abs{x}}\bar w_p=\mathcal O(1)$$
and
by using the problem solved by $\zeta_p$, i.e. \eqref{pbzetapn4}, we get
$$\begin{aligned} 0&=\int_{B^+_{\frac R \delta}}|\nabla\zeta_p|^2-\int_{\partial^+ B^+_{\frac R \delta}}\nabla \zeta_p\cdot \frac{x}{|x|}\zeta_p+\int_{\partial'B^+_{\frac R \delta}}\left(2H U \zeta_p +\left(2HU\bar w_p-\frac{\partial \bar w_p}{\partial \nu}\right)\right)\zeta_p\\
&=\int_{B^+_{\frac R \delta}}|\nabla\zeta_p|^2+\mathcal O(1)\end{aligned}$$
from which it follows that $$\int_{B^+_{\frac R \delta}}|\nabla\zeta_p|^2=\mathcal O(1).$$
By Proposition \ref{intestn4} we get 
$$\int_{B^+_{\frac R\delta}}|\nabla \bar w_p|^2= \frac{64\pi^2}{\abs{K}} \|\pi(p)\|^2 |\ln \delta|+\mathcal O(1)$$
Then \eqref{I12-I13} reduces to 
\begin{equation*}
	I_1^2+I_1^3 = -3\delta^2\int_{Q^+_{\frac R {2\delta}}}\abs{\nabla \bar w_p}^2 +\mathcal O(\delta^2)=-3\delta^2\int_{Q^+_{\frac R {2\delta}}}\abs{\nabla \bar w_p^0}^2 +\mathcal O(\delta^2)=-3 \frac{64\pi^2}{\abs{K}}\|\pi(p)\|^2 \d^2 |\ln \d|+\mathcal O(\d^2).
\end{equation*}

{\em Conclusion.}\\

We collect all the previous estimates and we take into account that
\begin{itemize}
\item[-] the terms of order $\d$ cancel because of Proposition \ref{vp}-(iii)
\item[-] the higher order terms which contain ${\rm Ric}_\nu(p)$ and $\overline R_{\ell\ell}(p)$ (di order $\d^2$ if $n\geq5$ and $\d^2|\ln\d|$ if $n=4$) cancel, because 
by Lemma \ref{varphi} and the fact that $\mathcal S_g(p)=2{\rm Ric}_\nu(p)+\overline R_{\ell\ell}(p)+\|\pi(p)\|^2$
$$  \d^2\alpha_n^2\omega_{n-1}I_{n-1}^n\frac{n-2}{n-1}\frac{{\rm Ric_\nu(p)}}{|K|^{\frac{n-2}{2}}}\left(2\varphi_{\frac{n-3}{2}}-(n-3)(n-1){\hat\varphi}_{\frac{n+1}{2}} \right)=0$$
 and
 $$ \d^2\alpha_n^2\frac{n-2}{3(n-1)}\omega_{n-1}I_{n-1}^n\frac{\overline R_{\ell\ell}(p)}{|K|^{\frac{n-2}{2}}}\left(-(n-4)\varphi_{\frac{n-3}{2}}-(n-3)\varphi_{\frac{n-1}{2}}+\frac{\mathfrak D_n}{(\mathfrak D_n-1)^{\frac{n-3}{2}}}\right)=0.$$
\end{itemize}
Finally, we have if $n=4$
 $$J_\e(\mathcal W)=\mathfrak E-\d^2|\ln \d|\underbrace{\left(\frac{192\pi^2}{|K|}+\alpha_4^2\omega_3 I_3^4 \frac{1}{|K|}\right)}_{:={\mathfrak b}_4}\|\pi(p)\|^2+\e \d  \mathfrak c_4 +\mathcal O(\delta^2). $$
and if $n\geq5$
$$J_\e(\mathcal W)=\mathfrak E-\d^2(\underbrace{\frac 12\mathfrak f_n+\mathfrak f_n^1}_{:=\mathfrak b_n})\|\pi(p)\|^2+\e \d{\mathfrak c}_n +\left\{\begin{aligned}&\mathcal O(\delta^3)\quad &\hbox{if}\,\, & n\geq 6\\ & \mathcal O(\delta^3|\ln \d|)\quad &\hbox{if}\,\, & n=5\end{aligned}\right. ,$$ 
because
   the higher order terms which contain $\|\pi(p)\|^2$ reduces to
\begin{equation*}
\begin{split}&\delta^2\alpha_n^2\frac{n-2}{n-1}\omega_{n-1}I_{n-1}^n\frac{\|\pi(p)\|^2}{|K|^{\frac{n-2}{2}}}\left(\varphi_{\frac{n-3}{2}}-(n-1)(n-3){\hat\varphi}_{\frac{n+1}{2}}-(n-1)(n-3){\hat\varphi}_{\frac{n-1}{2}}+3(n-3){\hat\varphi}_{\frac{n-1}{2}}\right)\\
&=-\delta^2\underbrace{\alpha_n^2\frac{n-2}{n-1}\omega_{n-1}I_{n-1}^n\frac{1}{|K|^{\frac{n-2}{2}}}\left(4(n-3){\hat\varphi}_{\frac{n-1}{2}}+\varphi_{\frac{n-3}{2}}\right)}_{:=\mathfrak f_n^1}\|\pi(p)\|^2\end{split}\end{equation*}
and, by Proposition \ref{intestn5},
$$\frac12 \int\limits_{\mathbb R^n_+}\left(-c_n\Delta V_p+{n+2\over n-2}|K|U^{4\over n-2}V_p\right)V_p=\frac 12 \mathfrak f_n \|\pi(p)\|^2$$
Here the energy of the bubble $\mathfrak E$ is   constant and is computed in the remark below.\end{proof}
\begin{remark}
The energy of the bubble is given by
$$\mathfrak E=\frac{c_n}{2}\int_{\mathbb R^+_n}|\nabla U|^2-\frac{n-2}{2n}K\int_{\mathbb R^n_+}U^{2^*}-(n-2)H\int_{\partial\mathbb R^n_+}U^{\dsh}$$
where $c_n:=\frac{4(n-1)}{n-2}$ and $$U(\tilde x, x_n):=\frac{\alpha_n}{|K|^{\frac{n-2}{4}}}\frac{1}{\left(|\tilde x|^2+(x_n+\mathfrak D_n)^2-1\right)^{\frac{n-2}{2}}}$$ and $\alpha_n:=(4n(n-1))^{\frac{n-2}{4}}$ and $\mathfrak D_n:=\sqrt{n(n-1)}\frac{H}{\sqrt{|K|}}$.\\ We recall that $U$ satisfies \eqref{limpb}.
Hence
$$\frac{c_n}{2}\int_{\mathbb R^n_+}|\nabla U|^2=\frac{c_n(n-2)}{4}H\int_{\partial\mathbb R^n_+}U^\dsh-\frac 12 |K|\int_{\mathbb R^n_+}U^{\dst}$$
Then
$$\begin{aligned}\mathfrak E&=-\frac 1 n |K|\int_{\mathbb R^n_+}U^\dst+H\int_{\partial\mathbb R^n_+}U^\dsh\\
&=-\frac 1 n |K| \frac{\alpha_n^\dst}{|K|^{\frac n 2}}\int_{\mathbb R^n_+}\frac{1}{(|\tilde x|^2+(x_n+\mathfrak D_n)^2-1)^n}\, d\tilde x\, dx_n\\
&+H\frac{\alpha_n^\dsh}{|K|^{\frac{n-1}{2}}}\int_{\partial\mathbb R^n_+}\frac{1}{(|\tilde x|^2+\mathfrak D_n^2-1)^{n-1}}\, d\tilde x\end{aligned}$$
Now, by using \eqref{JalphamI} with $\alpha=0$ and $m=n-1$ we get for $n\geq 4$
$$\begin{aligned} \int_{\partial\mathbb R^n_+}\frac{1}{(|\tilde x|^2+\mathfrak D_n^2-1)^{n-1}}\, d\tilde x
=\omega_{n-1}\frac{n-3}{n-1}\frac{I^{n}_{n-1}}{(\mathfrak D_n^2-1)^{\frac{n-1}{2}}}.\end{aligned}$$


Instead, by using \eqref{Jalpham} with $\alpha=0$ and $m=n$ we get 
$$\begin{aligned}
\int_{\mathbb R^n_+}\frac{1}{(|\tilde x|^2+(x_n+\mathfrak D_n)^2-1)^n}\, d\tilde x\, dx_n
=\omega_{n-1}\frac{n-3}{2(n-1)}I_{n-1}^{n}\varphi_{\frac{n+1}{2}}.
\end{aligned}$$
Collecting all the previous terms, by Lemma \ref{varphi} 
$$\begin{aligned}\mathfrak E &=-\frac{\alpha_n^\dst(n-3)}{2 n(n-1)}\omega_{n-1}I^n_{n-1}\frac{\varphi_{\frac{n+1}{2}}}{|K|^{\frac{n-2}{2}}}+\frac{\alpha_n^{\dsh}(n-3)}{n-1}\omega_{n-1}I^{n}_{n-1}\frac{H}{|K|^{\frac{n-1}{2}}(\mathfrak D_n^2-1)^{\frac{n-1}{2}}}\\
&=\alpha_n^\dsh \omega_{n-1}I^n_{n-1}\frac{n-3}{n-1}\frac{1}{|K|^{\frac{n-2}{2}}}\left[-\frac{ \alpha_n^{\dst-\dsh}}{2 n}\varphi_{\frac{n+1}{2}}+\frac{H}{|K|^{\frac{1}{2}}(\mathfrak D_n^2-1)^{\frac{n-1}{2}}}\right]\\
&=\underbrace{\alpha_n^\dsh \omega_{n-1}I^n_{n-1}\frac{n-3}{(n-1)\sqrt{n(n-1)}}}_{:=\mathfrak a_n}\frac{1}{|K|^{\frac{n-2}{2}}}\left[-(n-1)\varphi_{\frac{n-1}{2}}+\frac{\mathfrak D_n}{(\mathfrak D_n^2-1)^{\frac{n-1}{2}}}\right].\\
\end{aligned} $$ 
\end{remark}

\section{Proof of Lemma \ref{errorsize}}\label{app-c}
In the following we use the following notation
 $$ W_j(\xi):=\underbrace{\frac{1}{\delta_j^{\frac{n-2}{2}}}U\left(\frac{\left(\psi_{p}^\partial\right)^{-1}(\xi)-\eta(\e)\tau_j}{\delta_j}\right)}_{:= U_j}+\delta_j\underbrace{\frac{1}{\delta_j^{\frac{n-2}{2}}}V_p\left(\frac{\left(\psi_{p}^\partial\right)^{-1}(\xi)-\eta(\e)\tau_j}{\delta_j}\right)}_{:=V_j}. $$  
and 
$$\mathcal U_j(\xi)=\chi\left(\left(\psi_{p}^\partial\right)^{-1}(\xi)\right)U_j(\xi),\qquad \mathcal V_j(\xi)=\chi\left(\left(\psi_{p}^\partial\right)^{-1}(\xi)\right)V_j(\xi).$$

 \begin{proof} 
Let $$\gamma_M:=\mathfrak i^*_M\left(K \mathfrak g\left(\mathcal W\right)\right)\,\,\hbox{ and}\,\, \gamma_{\partial M}:=\mathfrak i^*_{\partial M}\left(\frac{n-2}{2}\left(H \mathfrak f\left(\mathcal W\right)-\e \mathcal W\right)\right).$$
By using the equations that $\gamma_M$ and $\gamma_{\partial M}$ satisify (see \eqref{w1}, \eqref{w2}) we get

$$\begin{aligned}\|\mathcal E\|^2&=c_n \int_M \left|\nabla_g \left(\mathcal W-\gamma_M-\gamma_{\partial M}\right)\right|^2\, d\nu_g+\int_M\mathcal S_g \left(\mathcal W-\gamma_M-\gamma_{\partial M}\right)^2\, d\nu_g\\
&=-c_n\int_M \left[\Delta_g \left(\mathcal W-\gamma_M-\gamma_{\partial M}\right)\left(\mathcal W-\gamma_M-\gamma_{\partial M}\right)\right]\, d\nu_g\\
&+\int_M \mathcal S_g \left(\mathcal W-\gamma_M-\gamma_{\partial M}\right)^2\, d\nu_g\\
&+c_n\int_{\partial M}\frac{\partial}{\partial\nu}\left(\mathcal W-\gamma_M-\gamma_{\partial M}\right)\left(\mathcal W-\gamma_M-\gamma_{\partial M}\right)\,d\sigma_g\\
&=\underbrace{\sum_{j=1}^k \left[\int_M\left(-c_n\Delta_g\mathcal W_j +\mathcal S_g \mathcal W_j-K \mathfrak g(\mathcal W_j)\right)\mathcal E\, d\nu_g\right.}_{ }\\
& \qquad\qquad\underbrace{\left.  +c_n \int_{\partial M}\left(\frac{\partial}{\partial\nu}\mathcal W_j-\frac{n-2}{2}H f(\mathcal W_j)+\frac{n-2}{2}\e \mathcal W_j\right)\mathcal E\, d\nu_g\right] }_{=:(I)}\\
&+\underbrace{\int_M K\left(\sum_{j=1}^k\mathfrak g(\mathcal W_j)-\mathfrak g\left(\sum_{j=1}^k \mathcal W_j\right)\right)\E\, d\nu_g}_{=:(II)}\\
&+\underbrace{\frac{n-2}{2}c_n \int_{\partial M}H\left(\sum_{j=1}^k \mathfrak f(\mathcal W_j)-\mathfrak f\left(\sum_{j=1}^k \mathcal W_j\right)\right)\mathcal E\, d\nu_g}_{=:(III)}\\
\end{aligned}$$
\begin{itemize}
\item Let us estimate $(I),$  which is the sum of   the contribution of each peak. We estimate each term in the sum and for  the sake of simplicity, we replace $\mathcal W_j$ by $\mathcal W$. Each term looks like
\begin{align*}
& \underbrace{ \int_M \left(-c_n\Delta_g \mathcal U -K\mathcal U^{\frac{n+2}{n-2}}+K\left(\mathfrak g (\mathcal U)-\mathfrak g(\mathcal U+\delta\mathcal V)\right)-c_n\delta\Delta_g \mathcal V\right)\E}_{(I_1)} \\ &+ \underbrace{\int_{\de M}\left(c_n\frac{\de \mathcal U}{\de \nu}-2(n-1)H\mathcal U^{\frac{n}{n-2}}\right)\E }_{(I_2)}\\& +\underbrace{2(n-1)\int_{\de M} \left(H\left(\mathfrak f(\mathcal U)-\mathfrak f(\mathcal U+\delta\mathcal V)\right)+ \frac{2\delta}{n-2}\frac{\de \mathcal V}{\de\nu}\right)\E}_{(I_3)} \\&+\underbrace{\int_M \mathcal S_g\mathcal W\E}_{(I_4)}\\ & +\underbrace{2(n-1)\e \int_{\de M} \: \mathcal W \E}_{(I_5)}.
\end{align*}

{\em Estimate of $(I_1)$}.  
We have
\begin{align*}
&|(I_1)| \lesssim \norm{\E}_{H^1(M)}\norm{A_\delta}_{L^\frac{2n}{n+2}(M)}
\end{align*}
with 
\begin{equation*}
	A_\delta =-c_n\Delta_g\:(\mathcal U+\delta\mathcal V )-K\:\mathcal U^{\frac{n+2}{n-2}}+K\left(\mathfrak g(\mathcal U)-\mathfrak g(\mathcal U+\d \mathcal V)\right).
\end{equation*}
Now, in local coordinates, the Laplace-Beltrami operator reads as:
\beq\label{lp}
\Delta_g\:\phi = \Delta\phi + \left(g^{ij}-\delta^{ij}\right)\de^2_{ij}\phi-g^{ij}\Gamma^k_{ij}\de_k\phi.
\eeq
and so by the decay of $U$ and $V_p$ (see Proposition \ref{vp}) and by Lemma \ref{metricexpansion} in variables $x=\delta y$ with
$|\delta y|\leq R$
\begin{align*}
 A_\delta(y)&= -\delta^{-\frac{n+2}{2}}c_n\Delta U(y)\chi(\delta y)-\frac{8(n-1)}{n-2}\delta^{-\frac n2}h^{ij}(p)\de^2_{ij}U(y)\chi(\delta y)y_n- \delta^{-\frac{n+2}{2}}K \chi^{\frac{n+2}{n-2}}(\delta y)U^{\frac{n+2}{n-2}}(y)\\ &+\delta^{-\frac{n+2}{2}} K \chi^\frac{n+2}{n-2}(\delta y)\left(\mathfrak g(U)-\mathfrak g(U+\delta V_p)\right) -\delta^{-\frac{n}{2}}c_n\Delta V_p+\delta^{-\frac{n-2}{2}}\Lambda(y)\\ &= \delta^{-\frac{n+2}{2}}K\left(  \chi(\delta y)-\chi^{\frac{n+2}{n-2}}(\delta y)\right)U^\frac{n+2}{n-2}(y)\\ &+\delta^{-\frac{n+2}{2}} K\chi^\frac{n+2}{n-2}(\delta y)\left(\mathfrak g(U)-\mathfrak g(U+\d V_p)\right) +  \delta^{-\frac n2}K \chi(\delta y)\mathfrak g'(U)V_p(y)+\delta^{-\frac{n-2}{2}} \Lambda(y)
\end{align*}
where 
$$|\Lambda(y)|\lesssim {1 \over1+|y|^{n-2}}\ \hbox{if}\ |\delta y|\leq R.$$
Finally,
\begin{align*}
\norm{A_\delta}_{L^\frac{2n}{n+2}(M)} &\lesssim \left\| K\left(  \chi(\delta y)-\chi^{\frac{n+2}{n-2}}(\delta y)\right)U^\frac{n+2}{n-2}(y)\right\|_{L^\frac{2n}{n+2}(\mathbb R^n_+)} \\ &+\delta\left\| K\left(  \chi(\delta y)-\chi^{\frac{n+2}{n-2}}(\delta y)\right)U^\frac{4}{n-2}(y)V_p(y)\right\|_{L^\frac{2n}{n+2}(\mathbb R^n_+)}\\ &
+\left\|  K\chi^\frac{n+2}{n-2}(\delta y)\left(\mathfrak g(U+\d V_p)  -\mathfrak g(U)-  \delta\mathfrak g'(U)V_p(y)\right)\right\|_{L^\frac{2n}{n+2}(\mathbb R^n_+)}\\ &+\delta^2\left\|  \Lambda \right\|_{L^\frac{2n}{n+2}(B(0,R/\delta))}
\\ &\lesssim \left\{\begin{aligned}&  \delta ^{ 2}\ \hbox{if}\ n\geq7\\
& \delta^2|\ln\delta|^\frac23\ \hbox{if}\ n=6\\
&\delta^{n-2\over2}\ \hbox{if}\ n= 4,5,
\end{aligned}\right. \end{align*}
because by  by \eqref{yyl}
$$\left|   \left(\mathfrak g(U+\d V_p)  -\mathfrak g(U)-  \delta\mathfrak g'(U)V_p(y)\right)\right |
\lesssim \left\{\begin{aligned}& U^{6-n\over n-2}\left(\delta V_p\right)^2 \ \hbox{if}\ n\geq6\\
& U^{6-n\over n-2}\left(\delta V_p\right)^{2}+\left(\delta V_p\right)^{n+2\over n-2}\ \hbox{if}\ n=4,5\\
\end{aligned}\right. $$
which implies
$$\left\|   \left(\mathfrak g(U+\d V_p)  -\mathfrak g(U)-  \delta\mathfrak g'(U)V_p(y)\right)\right \|_{L^\frac{2n}{n+2}(\mathbb R^n_+)}
\lesssim \left\{\begin{aligned}&  \delta ^{n+2\over n-2}\ \hbox{if}\ n\geq6\\
& \delta^2\ \hbox{if}\ n=4, 5.
\end{aligned}\right. $$
and also
$$\delta^2\left\|  \Lambda \right\|_{L^\frac{2n}{n+2}(B(0,R/\delta))}\lesssim \left\{\begin{aligned}&  \delta ^{ 2}\ \hbox{if}\ n\geq7\\
& \delta^2|\ln\delta|^\frac23\ \hbox{if}\ n=6\\
&\delta^{n-2\over2}\ \hbox{if}\ n=4,5.
\end{aligned}\right. $$

{\em Estimate of $(I_2)$}.
We have
\begin{align*}
(I_2)&=2(n-1)\int_{\de M}\left(\frac{2}{n-2}\frac{\de \mathcal U}{\de\nu}-H\mathcal U^{\frac{n}{n-2}}\right)\E \\[0.2cm]&\lesssim \norm{\E}_{H^1(M)}\norm{\frac{2}{n-2}\frac{\de \mathcal U}{\de\nu}-H\mathcal U^{\frac{n}{n-2}}}_{L^{\frac{2(n-1)}{n}}(\de M)}.
\end{align*}
and
\begin{align*}
	&\norm{\frac{2}{n-2}\frac{\de \mathcal U}{\de\nu}-H\mathcal U^{\frac{n}{n-2}}}_{L^{\frac{2(n-1)}{n}}(\de M)}\\[0.2cm]& = \left(\int_{\de M} \left(\frac{2}{n-2}\frac{\de\mathcal U}{\de\nu}-H\mathcal U^\frac{n}{n-2}\right)^\frac{2(n-1)}{n}\right)^\frac{n}{2(n-1)}\\[0.2cm]&\lesssim \left(\int_{\de\R^n_+}\left(\frac{2}{n-2}\frac{\de U(\tilde y)}{\de \nu}\chi(\delta\tilde y)-H U^\frac{n}{n-2}(\tilde y)\chi^\frac{n}{n-2}(\delta \tilde y)\right)\sqrt{\abs{g(\delta \tilde y)}}d\tilde y\right)^\frac{n}{2(n-1)} \\[0.2cm]&\lesssim\left(\int_{\de \R^n_+}HU^{\frac{2(n-1)}{n-2}}(\tilde y)\left(\chi(\delta \tilde y)- \chi^\frac{n}{n-2}(\delta \tilde y)\right)^\frac{2(n-1)}{n}d\tilde y\right)^\frac{n}{2(n-1)} 
	\\[0.2cm]& \lesssim \delta^{ 2}
\end{align*}

{\em Estimate of $(I_3)$}.
We have
\begin{align*}
|(I_3)| &\lesssim \norm{\E}_{H^1(M)} \norm{H\left(\mathfrak f(\mathcal U)-\mathfrak f(\mathcal U+\d \mathcal V)\right)+\frac{2\delta}{n-2}\frac{\de \mathcal V}{\de \nu}}_{L^\frac{2(n-1)}{n}(\de M)}
\end{align*}
and by the decay of $V_p$ and \eqref{yyl}
\begin{align*}
	&\norm{H\left(\mathfrak f(\mathcal U)-\mathfrak f(\mathcal U+\d \mathcal V)\right)+\frac{2\delta}{n-2}\frac{\de \Vd}{\de \nu}}_{L^\frac{2(n-1)}{n}(\de M)} \\ &\lesssim \norm{H \left(\mathfrak f(U)-\mathfrak f(U+\delta V_p)\right)\chi^{\frac{n}{n-2}}(\d \tilde y)+\frac{2\d}{n-2}\frac{\de V_p}{\de \nu}\chi(\d \tilde y)}_{L^\frac{2(n-1)}{n}(\de \mathbb R^n_+)} \\ 
	&\lesssim \d \norm{H \mathfrak f'(U+\d \theta V_p)V_p \chi^{\frac{n}{n-2}}(\d \tilde y)-\frac{n}{n- 2}H U^{\frac{2}{n-2}}V_p \chi(\delta \tilde y)}_{L^\frac{2(n-1)}{n}(\de \mathbb R^n_+)}  \\ 
	&\lesssim \d \norm{H\left( \chi^{\frac{n}{n-2}}(\d \tilde y)- \chi(\d \tilde y)\right) \mathfrak f'(U+\d\theta V_p)V_p}_{L^\frac{2(n-1)}{n}(\de \mathbb R^n_+)} \\ 
&+\d \norm{H \chi(\d\tilde y)\left(\mathfrak f'(U+\d \theta V_p)) -\mathfrak f'(U)\right)V_p}_{L^\frac{2(n-1)}{n}(\de \mathbb R^n_+)}\\
&\lesssim \left\{\begin{aligned}&  \delta ^{2}\ \hbox{if}\ n\geq5\\
& \delta^2|\ln\delta|^\frac23\ \hbox{if}\ n=4.
\end{aligned}\right. 
\end{align*}
because by   \eqref{yyl}
$$\left|   \left(\mathfrak f'(U+\d \theta V_p)) -\mathfrak f'(U)\right)V_p\right |
\lesssim  \delta U^{4-n\over n-2}V_p^2.$$

{\em Estimate of $(I_4)$}. By Hölder's inequality 
\begin{align*}
|(I_4)|&\lesssim \int_M |\mathcal U||\mathcal E|\, \ d\nu_g +\delta \int_M |\mathcal V||\mathcal E|\, d\nu_g\lesssim \norm{\E}\left(\norm{\mathcal U}_{L^{\frac{2n}{n+2}}(M)}+\delta\norm{\mathcal V}_{L^2(M)}\right),
\end{align*}
with
\begin{align*}
\norm{\mathcal U}_{L^{\frac{2n}{n+2}}(M)} &\lesssim  \left\{\begin{aligned} &\delta^2\quad &\mbox{if}\,\, &n\geq 7\\
& \delta^2|\ln \delta|^{\frac 23}\quad &\mbox{if}\,\, &n= 6\\
&\delta^{\frac{n-2}{2}}\quad &\mbox{if}\,\, &n= 4, 5\end{aligned}\right.
\end{align*}
and
\begin{align*}
\norm{\mathcal V}_{L^{2}(M)} &\lesssim  \left\{\begin{aligned} &\delta^2\quad &\mbox{if}\,\, &n\geq 7\\
& \delta^2|\ln \delta|^{\frac 12}\quad &\mbox{if}\,\, &n= 6\\
&\delta^{\frac{n-2}{2}}\quad &\mbox{if}\,\, &n= 4, 5\end{aligned}\right.
\end{align*}

{\em Estimate of $(I_5)$}. By Hölder's inequality 
\begin{align*}
|(I_5)|&\lesssim \e\left(\norm{\mathcal U}_{L^{\frac{2(n-1)}{n}}(\de M)}+\delta\norm{\mathcal V}_{L^2(\de M)}\right)\norm{\E},
\end{align*}
with
\begin{align*}
\norm{\mathcal U}_{L^{\frac{2(n-1)}{n}}(\de M)} & \lesssim   \left\{\begin{aligned} &\delta\quad &\mbox{if}\,\, &n\geq 5\\
& \delta|\ln \delta|^{\frac 23}\quad &\mbox{if}\,\, &n= 4\end{aligned}\right.
\end{align*}
and
\begin{align*}
\norm{\mathcal V}_{L^{2}(\de M)} &\lesssim  \left\{\begin{aligned} &\delta^{\frac 32}\quad &\mbox{if}\,\, &n\geq 6\\
& \delta^{\frac 32}|\ln \delta|^{\frac 12}\quad &\mbox{if}\,\, &n=5 \\
&\delta\quad &\mbox{if}\,\, &n= 4.\end{aligned}\right.
\end{align*}

Finally, collecting all the previous estimates, by the choice of $\delta_j$ in \eqref{deltaj} and \eqref{deltaj4}
$$ |(I)|\lesssim   \left\{\begin{aligned} &\e^{2} \quad &\mbox{if}\,\, &n\geq 7\\
&\e^{2}|\ln\e|^\frac23 \quad &\mbox{if}\,\, &n\geq 6\\
& \e^{\frac 32}  \quad &\mbox{if}\,\, &n=5 \\
&\rho(\e)\quad &\mbox{if}\,\, &n= 4.\end{aligned}\right.$$

\item Let us estimate the interaction terms $(II)$ and $(III)$.\\ Set for any $h=1, \ldots, k$ $B_h^+=B^+(\eta(\e)\tau_h, \eta(\e) \sigma/2)$ where $\sigma>0$ is small enough and $\partial' B_h^+=B^+_h\cap\de \mathbb R^n_+$. Since $\sigma$ is small then $B^+_h\subset B^+_R$ and $\de'B_h^+\subset \de' B^+_R$ and they are disjoints.\\
We remark that in $B^+_{2R}(0)$ we get
$$W_i(x)\lesssim \frac{\d_i^{\frac{n-2}{2}}}{|x-\eta(\e)\tau_i|^{n-2}}.$$
Then
$$\left|\int_M K\left(\sum_j \mathfrak g(\mathcal W_j)-\mathfrak g\left(\sum_j \mathcal W_j\right)\right)\mathcal E\right|\lesssim \norm{\mathfrak g\left(\sum_j \mathcal W_j\right)-\sum_j \mathfrak g(\mathcal W_j)}_{L^{\frac{2n}{n+2}}(M)}\norm{\mathcal E}$$
Hence
$$\begin{aligned} &\|\ldots\ldots\|_{L^{\frac{2n}{n+2}}(M)}\lesssim \left[\int_{B^+_{2R}\setminus B^+_R}|\ldots\ldots|^{\frac{2n}{n+2}}|g(x)|^{\frac 12}\, dx\right]^{\frac{n+2}{2n}}\\
&+\left[\int_{B^+_R\setminus \bigcup_h B_h^+}|\ldots\ldots|^{\frac{2n}{n+2}}|g(x)|^{\frac 12}\, dx\right]^{\frac{n+2}{2n}}+\sum_{h=1}^k\left[\int_{B^+_h}|\ldots\ldots|^{\frac{2n}{n+2}}|g(x)|^{\frac 12}\, dx\right]^{\frac{n+2}{2n}} \\
&\leq \sum_{i=1}^k \left[\int_{B_R^+}(1-\chi^\dst(|x|))|W_i|^\dst|g(x)|^{\frac 12}\, dx\right]^{\frac{n+2}{2n}}+ \left[\int_{B_R^+\setminus \bigcup_h B_h^+}|W_i|^\dst|g(x)|^{\frac 12}\, dx\right]^{\frac{n+2}{2n}}\\
&+\sum_{h=1}^k \left[\int_{B_h^+}\left||W_h|^{\dst-2}\sum_{i\neq h}W_i\right|^{\frac{2n}{n+2}}|g(x)|^{\frac 12}\, dx\right]^{\frac{n+2}{2n}}+\sum_{h=1}^k \left[\int_{B_h^+}\left|\sum_{i\neq h}W_i\right|^\dst|g(x)|^{\frac 12}\, dx\right]^{\frac{n+2}{2n}}\end{aligned}$$
Let us estimate each term.
$$\begin{aligned}&\sum_{i=1}^k \left[\int_{B_{2R}^+\setminus B_R^+}(1-\chi^\dst(|x|))|W_i|^\dst|g(x)|^{\frac 12}\, dx\right]^{\frac{n+2}{2n}}\lesssim \sum_{i=1}^k \left[\int_{\mathbb R^n_+\setminus B_R^+}\frac{\delta_i^n}{|x-\eta(\e)\tau_i|^{2n}}\, dx\right]^{\frac{n+2}{2n}}\\
&\lesssim \sum_{i=1}^k\frac{\delta_i^{\frac{n+2}{2}}}{\eta(\e)^{\frac{n+2}{2}}}\left[\int_{\mathbb R^n_+\setminus B^+_{\frac{R}{\eta(\e)}}}\frac{1}{|y-\tau_i|^{2n}}\, dy\right]^{\frac{n+2}{2n}}\lesssim \left\{\begin{aligned}& \e^{(1-\alpha)\frac{n+2}{2}} \quad &\hbox{if}\,\, &n\geq 5\\ 
& (\rho(\e))^3|\ln \rho(\e)|^{\frac 34} \quad &\hbox{if}\,\, &n=4.\end{aligned}\right.\end{aligned}$$
$$\begin{aligned}&\left[\int_{B_R^+\setminus \bigcup_h B_h^+}|W_i|^\dst|g(x)|^{\frac 12}\, dx\right]^{\frac{n+2}{2n}}
\lesssim \left[\int_{B_R^+\setminus \bigcup_h B_h^+}\frac{\delta_i^n}{|x-\eta(\e)\tau_i|^{2n}}\, dx\right]^{\frac{n+2}{2n}}\\
&\lesssim \sum_{i=1}^k \frac{\d_i^{\frac{n+2}{2}}}{\eta(\e)^{\frac{n+2}{2}}}\left[\int_{\mathbb R^n_+\setminus B^+_{\frac{R}{\eta(\e)}}}\frac{1}{|y-\tau_i|^{2n}}\, dy\right]^{\frac{n+2}{2n}}\lesssim\left\{\begin{aligned}& \e^{(1-\alpha)\frac{n+2}{2}} \quad &\hbox{if}\,\, &n\geq 5\\ 
& (\rho(\e))^3|\ln \rho(\e)|^{\frac 34} \quad &\hbox{if}\,\, &n=4.\end{aligned}\right.\end{aligned}$$
Now for $n\geq 7$
$$\begin{aligned}&\sum_{h=1}^k \left[\int_{B_h^+}\left||W_h|^{\dst-2}\sum_{i\neq h}W_i\right|^{\frac{2n}{n+2}}|g(x)|^{\frac 12}\, dx\right]^{\frac{n+2}{2n}}\\
&\lesssim\sum_{h=1}^k\sum_{i\neq h}\left[\int_{B_h^+}\frac{\delta_h^{\frac{4n}{n+2}}}{|x-\e^\alpha\tau_h|^{\frac{8n}{n+2}}}\frac{\delta_i^{\frac{n(n-2)}{n+2}}}{|x-\e^\alpha\tau_i|^{\frac{2n(n-2)}{n+2}}}\,dx\right]^{\frac{n+2}{2n}}\\
&\lesssim\sum_{h=1}^k\sum_{i\neq h}\delta_h^2\delta_i^{\frac{n-2}{2}}\left(\int_{B_h^+}\frac{1}{|x-\e^\alpha\tau_h|^{\frac{8n}{n+2}}}\, dx\right)^{\frac{n+2}{2n}}\\
&\lesssim \sum_{h=1}^k\sum_{i\neq h}\delta_h^2\delta_i^{\frac{n-2}{2}}\left(\int_{B^+_{\e^\alpha\sigma/2}} \frac{1}{|y|^{\frac{8n}{n+2}}}\, dy\right)^{\frac{n+2}{2n}}\lesssim\e^{\frac{n+2}{2}+\alpha\frac{n-6}{2}} \end{aligned}$$
while for $n=4, 5, 6$ 
$$\begin{aligned}&\sum_{h=1}^k \left[\int_{B_h^+}\left||W_h|^{\dst-2}\sum_{i\neq h}W_i\right|^{\frac{2n}{n+2}}|g(x)|^{\frac 12}\, dx\right]^{\frac{n+2}{2n}}\\
&\lesssim\sum_{h=1}^k\sum_{i\neq h}\frac{\delta_h^2\delta_i^{\frac{n-2}{2}}}{\eta(\e)^{\frac{6-n}{2}}}\left(\int_{B_{\frac \sigma 2}^+}\frac{1}{(|\tilde y|^2 + (y_n+\mathfrak D_n\delta_h\eta(\e)^{-1})^2-\delta_h^2\eta(\e)^{-2})^{\frac{4n}{n+2}}}\, dx\right)^{\frac{n+2}{2n}}\\
&\lesssim \sum_{h=1}^k\sum_{i\neq h}\frac{\delta_h^2\delta_i^{\frac{n-2}{2}}}{\eta(\e)^{\frac{6-n}{2}}}\left(\int_{B^+_{\frac \sigma 2}} \frac{1}{(|y|^2+(\delta_h\eta(\e)^{-1})^2(\mathfrak D^2_n(p)-1))^{\frac{4n}{n+2}}}\, dy\right)^{\frac{n+2}{2n}}\\
&\lesssim\left\{\begin{aligned}& \e^{4}|\ln \e|^{\frac 23} \quad &\mbox{if}\quad n=6\\ & \e^{3} \quad &\mbox{if}\quad n=5\\
& (\rho(\e))^3|\ln \rho(\e)| \quad &\mbox{if}\quad n=4.\end{aligned}\right.\end{aligned}$$

At the end
$$\begin{aligned}\sum_{h=1}^k \left[\int_{B_h^+}\left|\sum_{i\neq h}W_i\right|^\dst|g(x)|^{\frac 12}\, dx\right]^{\frac{n+2}{2n}}&\lesssim \sum_{h=1}^k\sum_{i\neq h}\left[\int_{B_h^+}\frac{\delta_i^n}{|x-\eta(\e)\tau_i|^{2n}}\, dx\right]^{\frac{n+2}{2n}}\\
&\lesssim\left\{\begin{aligned}& \e^{(1-\alpha)\frac{n+2}{2}} \quad &\hbox{if}\,\, &n\geq 5\\ 
& (\rho(\e))^3|\ln \rho(\e)|^{\frac 34}  \quad &\hbox{if}\,\, &n=4.\end{aligned}\right.\end{aligned}$$
For the term $(III)$ we get
$$|(III)|\lesssim \norm{\sum_j \mathfrak f(\mathcal W_j)-\mathfrak f\left(\sum_j \mathcal W_j\right)}_{L^{\frac{2(n-1)}{n}}(\partial M)}\norm{\E}$$
Hence
$$\begin{aligned} &\|\ldots\ldots\|_{L^{\frac{2(n-1)}{n}}(\de M)}\lesssim \left[\int_{\de'B^+_{2R}\setminus \de'B^+_R}|\ldots\ldots|^{\frac{2(n-1)}{n}}|g(\tilde x, 0)|^{\frac 12}\, dx\right]^{\frac{n}{2(n-1)}}\\
&+\left[\int_{\de'B^+_R\setminus \bigcup_h \de' B_h^+}|\ldots\ldots|^{\frac{2(n-1)}{n}}|g(\tilde x, 0)|^{\frac 12}\, dx\right]^{\frac{n}{2(n-1)}}+\sum_{h=1}^k\left[\int_{\de' B^+_h}|\ldots\ldots|^{\frac{2(n-1)}{n}}|g(\tilde x, 0)|^{\frac 12}\, dx\right]^{\frac{n}{2(n-1)}} \\
&\leq \sum_{i=1}^k \left[\int_{\de'B_R^+}(1-\chi^\dsh(\tilde x, 0))|W_i|^\dsh|g(\tilde x, 0)|^{\frac 12}\, dx\right]^{\frac{n}{2(n-1)}}+ \left[\int_{\de' B_R^+\setminus \bigcup_h \de' B_h^+}|W_i|^\dsh|g(\tilde x, 0)|^{\frac 12}\, dx\right]^{\frac{n}{2(n-1)}}\\
&+\sum_{h=1}^k \left[\int_{\de' B_h^+}\left||W_h|^{\dsh-2}\sum_{i\neq h}W_i\right|^{\frac{2(n-1)}{n}}|g(\tilde x, 0)|^{\frac 12}\, dx\right]^{\frac{n}{2(n-1)}}+\sum_{h=1}^k \left[\int_{\de' B_h^+}\left|\sum_{i\neq h}W_i\right|^\dsh|g(\tilde x, 0)|^{\frac 12}\, dx\right]^{\frac{n}{2(n-1)}}\end{aligned}$$
Now
$$\begin{aligned}&\sum_{i=1}^k \left[\int_{\de'B_{2R}^+\setminus\de'B^+_R }|W_i|^\dsh|g(\tilde x, 0)|^{\frac 12}\, dx\right]^{\frac{n}{2(n-1)}}\lesssim \sum_{i=1}^k \left[\int_{\mathbb R^{n-1}\setminus \de'B_R^+}\frac{\d_i^{n-1}}{|\tilde x -\eta(\e) \tilde\tau_i|^{\frac{2(n-1)(n-3)}{n-2}}}\, d\tilde x\right]^{\frac{n}{2(n-1)}}\\
&\lesssim \sum_{i=1}^k\frac{\d_i^{\frac n 2}}{\eta(\e)^{\frac{n}{2}}}\left[\int_{\mathbb R^{n-1}\setminus \de'B^+_{\frac {R}{\eta(\e)}}}\frac{1}{|\tilde y -\tilde\tau_i|^{2(n-1)}}\, d\tilde x\right]^{\frac{n}{2(n-1)}}\\&\lesssim\left\{\begin{aligned}& \e^{(1-\alpha)\frac n 2} \quad&\hbox{if}\,\, &n\geq 5\\
& (\rho(\e))^2|\ln \rho(\e)|^{\frac 23} \quad&\hbox{if}\,\, &n=4\end{aligned}\right.\end{aligned}$$
Similarly
$$\left[\int_{\de' B_R^+\setminus \bigcup_h \de' B_h^+}|W_i|^\dsh|g(\tilde x, 0)|^{\frac 12}\, dx\right]^{\frac{n}{2(n-1)}}\lesssim\left\{\begin{aligned}& \e^{(1-\alpha)\frac n 2}\quad&\hbox{if}\,\, &n\geq 5\\
& (\rho(\e))^2|\ln \rho(\e)|^{\frac 23} \quad&\hbox{if}\,\, &n=4\end{aligned}\right.$$
Now for $n\geq 5$
$$\begin{aligned}&\sum_{h=1}^k \left[\int_{\de' B_h^+}\left||W_h|^{\dsh-2}\sum_{i\neq h}W_i\right|^{\frac{2(n-1)}{n}}|g(\tilde x, 0)|^{\frac 12}\, dx\right]^{\frac{n}{2(n-1)}}\\
&\lesssim \sum_{h=1}^k\sum_{i\neq h}\d_h\d_i^{\frac{n-2}{2}}\left[\int_{\de' B^+_h}\frac{1}{|\tilde x -\eta(\e)\tilde\tau_h|^{\frac{4(n-1)}{n}}}\, d\tilde x\right]^{\frac{n}{2(n-1)}}\\
&\lesssim \sum_{h=1}^k\sum_{i\neq h}\d_h\d_i^{\frac{n-2}{2}}\left[\int_{\de' B^+_{\eta(\e)\sigma/2}}\frac{1}{|\tilde y|^{\frac{4(n-1)}{n}}}\, d\tilde x\right]^{\frac{n}{2(n-1)}}\lesssim\e^{\frac n 2 +\alpha\frac{n-4}{2}}\end{aligned}$$ and for $n=4$
$$\begin{aligned}&\sum_{h=1}^k \left[\int_{\de' B_h^+}\left||W_h|\sum_{i\neq h}W_i\right|^{\frac 3 2}|g(\tilde x, 0)|^{\frac 12}\, dx\right]^{\frac 2 3}\lesssim\rho(\e) .\end{aligned}$$
Finally
$$\sum_{h=1}^k \left[\int_{\de' B_h^+}\left|\sum_{i\neq h}W_i\right|^\dsh|g(\tilde x, 0)|^{\frac 12}\, dx\right]^{\frac{n}{2(n-1)}}\lesssim\left\{\begin{aligned}& \e^{(1-\alpha)\frac n 2} \quad&\hbox{if}\,\, &n\geq 5\\
& (\rho(\e))^2|\ln \rho(\e)|^{\frac 23} \quad&\hbox{if}\,\, &n=4\end{aligned}\right..$$
\end{itemize}
We collect all the above estimates and the claim follows.
\end{proof}
\section{Proof of Lemma \ref{lineareinv}}\label{app-d}
\begin{proof}
We argue by contradiction. Assume there exist sequences $\e_m\to0$, $p_m\to p\in\de M$, $\tau_m\in \R^n\times\stackrel{k}{\cdots}\times\R^n$, $ d_m\in\R_+\times\stackrel{k}{\cdots}\times\R_+$, $\phi_m$, $\psi_m \in \mathcal K^\perp$ such that $ \tau_{m\,j}\to\tau_j$, $d_{m\,j}\to d_j>0$ for every $j=1,\ldots,k$ and
\begin{equation*}
\mathcal L(\phi_m)=\psi_m,\hsp \norm{\phi_m}_{H^1}=1,\hsp \norm{\psi_m}_{H^1}\to 0.
\end{equation*}
We will write $\W_m := \W( d_m,\tau_m, \e_m)$, and $$\phi_{m\,l}(\xi):={\delta_{m\,l}}^\frac{n-2}{2}\phi_m(\psi^\de_{p_m}(\delta_{m\,l}y+\eta(\e_m)\tau_{m\,l}))\chi(\delta_{m\,l}y+\eta(\e_m)\tau_{m\,l}).$$ Since $\norm{\phi_m}_{H^1}=1$, $\phi_{m\,l}$ is bounded in $\mathcal D^{1,2}(\R^n_+)$, there exists $\phi_l$ such that
\begin{equation}\label{inv:conver}
\begin{array}{ll}
\phi_{m\,l}\weakto\phi_l & \text{weakly in}\hsp \mathcal D^{1,2}(\R^n_+), \\
\phi_{m\,l}\to\phi_l & \text{weakly in} \hsp L^\frac{2n}{n-2}(\R^n_+), \\
\phi_{m\,l}\weakto\phi_l & \text{strongly in}\hsp L_{\text{loc}}^\frac{2(n-1)}{n-2}(\de \R^n_+), 
\end{array}
\end{equation}

Firtly, notice that by definition of $\mathcal L$,
\begin{equation}\label{inv:eq1}
\psi_m-\phi_m-\iM(K\sgg'(\W_m)\phi_m)-\idM\left(\frac{n-2}{2}H\sff'(\W_m)\phi_m-\e_m\phi_m\right) = \sum_{\substack{j=1,\ldots,k \\ i=1,\ldots,n}}C_m^{ji}\Z_{j,i}.
\end{equation}
We will show that $C_m^{ji}=o_m(1)$ for every $j,i$. Multiply equation \eqref{inv:eq1} by $\Z_{l,q}$ and integrate by parts to obtain
\begin{align*}
C_m^{ji}\scal{\Z_{j,i},\Z_{l,q}} &= \underbrace{\scal{\psi_m,\Z_{l,q}}}_{=0}-\underbrace{\scal{\phi_m,\Z_{l,q}}}_{=0}-\scal{\iM(K\sgg'(\W)\phi_m),\Z_{l,q}}\\&-\scal{\idM\left(\frac{n-2}{2}H\sff'(\W_m)\phi_m\right),\Z_{l,q}} +\e_m\scal{\idM(\phi_m),\Z_{l,q}} \\ &= \frac{n+2}{n-2}\int_M\abs{K}\abs{\W_m}^\frac{4}{n-2}\phi_m\Z_{l,q} - c_n\frac{n}{2}\int_{\de M}H\abs{\W_m}^\frac{2}{n-2}\phi_m\Z_{l,q} \\ &+\e_m\int_{\de M}\phi_m\Z_{l,q}.
\end{align*}
Observe that
\begin{align*}
\e_m\int_{\de M}\phi_m \Z_{l,q} &= \e_m \int_{\de'B^+_\sigma(\eta(\e_m)\tau_{m\,l})}\phi_m\frac{1}{\delta_{m\,l}^\frac{n-2}{2}}\,\mathfrak z_q\left(\frac{x-\eta(\e_m)\tau_{m\,l}}{\delta_{m\,l}}\right) = \e_m\delta_{m\,l}^2\int_{\de'B^+_\frac{\sigma}{\delta_{m\,l}}}\phi_{m\,l}\,\mathfrak{j}_q \\ &=\e_m\times\left\lbrace\begin{array}{ll}\delta_{m\,l}^2\abs{\ln \delta_{m\,l}} & \text{if}\hsp 1\leq q<n,\\
	\delta_{m\,l} &\text{if}\hsp q=n.
\end{array}\right.
\end{align*}
Then, by \eqref{scalarproductZ}, \eqref{exp:int} and \eqref{exp:bound}:
\begin{equation}\label{inv:eq2}
	C_m^{ji}\mathcal O_m(1) = \frac{n+2}{n-2}\int_{B^+_{\sigma \delta_{m\,l}^{-1}}}\abs{K}U^\frac{4}{n-2}\phi_{m\,l} \,\mathfrak z_q -c_n\frac{n}{2}\int_{\de'B^+_{\sigma \delta_{m\,l}^{-1}}} HU^\frac{2}{n-2}\phi_{m\,l}\,\mathfrak z_q+o_m(1),
\end{equation}
On the other hand, since $\mathfrak z_i$ satisfies \eqref{L} we get 
\begin{equation*}
\begin{split}
0&=\scal{\phi_m,\Z_{l,q}}=c_n\int_M\nabla_g\phi_m\cdot\nabla_g\Z_{l,q} + \int_M S_g\phi_m\Z_{l,q} \\ &= -c_n\int_M\phi_m\Delta_g\Z_{l,q}+c_n\int_{\de M}\phi_m\frac{\de\Z_{l,q}}{\de\nu}+\int_MS_g\phi_m\Z_{l,q}\\ &= -c_n\int_{B^+_\sigma(\eta(\e_m)\tau_{m\,l})}\phi_m\frac{1}{\delta_{m\,l}^\frac{n+2}{2}}\Delta\mathfrak z_q\left(\frac{x-\eta(\e_m\tau_{m\,l})}{\delta_{m\,l}}\right) \\&+c_n\int_{\de'B^+_\sigma(\eta(\e_m)\tau_{m\,l})}\phi_m\frac{1}{\delta_{m\,l}^\frac{n}{2}}\frac{\de\mathfrak z_q}{\de\nu}\left(\frac{x-\eta(\e_m\tau_{m\,l})}{\delta_{m\,l}}\right) \\ &+c_n\int_{B^+_\sigma(\eta(\e_m)\tau_{m\,l})}\phi_m\frac{1}{\delta_{m\,l}^\frac{n-2}{2}}\mathfrak z_q\left(\frac{x-\eta(\e_m\tau_{m\,l})}{\delta_{m\,l}}\right)+o_m(1)\\ &=-\frac{n+2}{n-2}\int_{B^+_{\sigma\delta_{m\,l}^{-1}}}\abs{K}U^\frac{4}{n-2}\phi_{m\,l}\,\mathfrak z_q+c_n\frac{n}{2}\int_{\de'B^+_{\sigma \delta_{m\,l}^{-1}}} HU^\frac{2}{n-2}\phi_{m\,l}\,\mathfrak z_q+o_m(1).
\end{split}
\end{equation*}
Hence,
\begin{equation}\label{inv:eq3}
\frac{n+2}{n-2}\int_{B^+_{\sigma\delta_{m\,l}^{-1}}}\abs{K}U^\frac{4}{n-2}\phi_{m\,l}\,\mathfrak z_q-c_n\frac{n}{2}\int_{\de'B^+_{\sigma \delta_{m\,l}^{-1}}} HU^\frac{2}{n-2}\phi_{m\,l}\,\mathfrak z_q=o_m(1).	
\end{equation}
The claim follows from \eqref{inv:eq2} and \eqref{inv:eq3}. Now, take any $\varphi\in C^2(\R^n_+)$ with compact support, and define
\begin{equation*}
\varphi_{m\,j}(\xi)=\frac{1}{\delta_{m\,j}^\frac{n-2}{2}}\varphi\left(\frac{(\psi^\de_{p_m})^{-1}(\xi)-\eta(\e_m)\tau_{m\,j}}{\delta_{m\,j}}\right)\chi\left((\psi^\de_{p_m})^{-1}(\xi)\right).
\end{equation*}
By \eqref{inv:eq1}, arguing as in the proof of Lemma \ref{A:exp}:
\begin{align*}
&\scal{\phi_m,\varphi_{m\,j}} = C_m^{ji}\scal{\Z_{j,i},\varphi_{m\,j}}+\scal{\iM(K\sgg'(\W)\phi_m),\varphi_{m\,j}}\\&+\scal{\idM\left(\frac{n-2}{2}H\sff'(\W)\phi_m\right),\varphi_{m\,j}}-\e_m\scal{\idM(\phi_m),\varphi_{m\,j}}+\scal{\psi_m,\varphi_{m\,j}}\\&=-\frac{n+2}{n-2}\abs{K}\int_{B^+_{\sigma\delta_{m\,j}^{-1}}}U^\frac{4}{n-2}\phi_{m\,j}\varphi + c_n\frac n2 \int_{\de'B^+_{\sigma\delta_{m\,j}^{-1}}}HU^\frac{2}{n-2}\phi_{m\,j}\varphi + o_m(1).
\end{align*}
Moreover,
\begin{align*}
\scal{\phi_m,\varphi_{m\,j}} &= c_n\int_M\nabla_g\phi_m\cdot\nabla_g\varphi_{m\,j}+\int_M S_g\phi_m\varphi_{m\,j} \\&= c_n\int_{B^+_{\sigma\delta_{m\,j}^{-1}}}\nabla\phi_{m\,j}\cdot\nabla\varphi +o_m(1).
\end{align*}
Therefore
\begin{equation}\label{inv:weakform}
\begin{split}
c_n\int_{B^+_{\sigma\delta_{m\,j}^{-1}}}\nabla\phi_{m\,j}\cdot\nabla\varphi+\frac{n+2}{n-2}\abs{K}\int_{B^+_{\sigma\delta_{m\,j}^{-1}}}U^\frac{4}{n-2}\phi_{m\,j}\varphi \\ - c_n\frac n2 \int_{\de'B^+_{\sigma\delta_{m\,j}^{-1}}}HU^\frac{2}{n-2}\phi_{m\,j}\varphi = o_m(1).
\end{split}
\end{equation}
If we take limits in \eqref{inv:weakform} (possible thanks to \eqref{inv:conver}), we see that $\phi_j$ is a weak solution to \eqref{L}, and then by Theorem  \ref{nondegeneracy} $\phi_j\in\text{Span}\left(\mathfrak z_\gamma:\:\gamma=1,\ldots,n\right)$. Now, by the orthogonality condition
\begin{equation*}
0=\scal{\phi_m,\Z_{j,i}}=\scal{\phi_j,\,\mathfrak z_i}_{\R^n_+}+o_m(1),\hsp\text{for every}\hsp i=1,\ldots,n,
\end{equation*}
which implies $\phi_j=0$ for every $j=1,\ldots,k$. However, again by \eqref{inv:eq1}:
\begin{align*}
\norm{\phi_m}^2_{H^1} &= -\frac{n+2}{n-2}\abs{K}\int_{B^+_{\sigma\delta_{m\,j}^{-1}}}U^\frac{4}{n-2}\phi_{m\,j}^2 + c_n\frac n2 \int_{\de'B^+_{\sigma\delta_{m\,j}^{-1}}}HU^\frac{2}{n-2}\phi_{m\,j}^2 + o_m(1) \\ &= o_m(1).	
\end{align*}
This contradicts $\norm{\phi_m}_{H^1}=1$ and finishes the proof.
\end{proof}

\begin{lemma} It holds true that
\begin{equation}\label{scalarproductZ}
	\scal{\Z_{j,i},\Z_{l,q}}=\left\lbrace \begin{array}{ll} \delta^{ik}\norm{\Z_{l,q}}_{H^1(M)}^2=\delta^{ik}\bigo{1} &\text{if}\hsp j=l\\[0.15cm] \bigo{\delta_j^\frac{n-2}{2}\delta_l^\frac{n-2}{2}} & \text{if}\hsp j\neq l.
	\end{array} \right.
\end{equation}
\end{lemma}
\begin{proof}
The first statement is a consequence of the orthogonality of $\mathfrak z_i$ and $\mathfrak{j}_q$ in $H^1(\R^n_+)$, so let us assume $j\neq l.$ The following decay property will be used throughout this proof:
\begin{equation}\label{Zdecay}
\abs{\nabla^\alpha Z_{j,i}}\leq \left\lbrace\begin{array}{ll}\delta_j^\frac{n}{2}\abs{x-\eta(\e)\tau_j}^{1-\alpha-n} & \text{if}\hsp i=1,\ldots,n-1, \\ \delta_j^\frac{n-2}{2}\abs{x-\eta(\e)\tau_j}^{2-\alpha-n} & \text{if}\hsp i=n, \end{array}\right.
\end{equation}
for $\alpha=0,1.$ Take a small $\sigma>0$ such that $B^+_\sigma(\eta(\e)\tau_j)\cap B^+_\sigma(\eta(\e)\tau_l)=\emptyset$, and let $\Omega_\sigma = B^+_R\setminus\left(B^+_\sigma(\eta(\e)\tau_j)\cup B^+_\sigma(\eta(\e)\tau_l)\right)$. First, observe that
\begin{align*}
\int_M \mathcal S_g\Z_{j,i}\Z_{l,q} &= \int_{\Omega_\sigma}\mathcal S_gZ_{j,i}Z_{l,q}+\int_{B^+_\sigma(\eta(\e)\tau_j)}\mathcal S_g\frac{1}{\delta_j^\frac{n-2}{2}}\mathfrak z_i\left(\frac{x-\eta(\e)\tau_j}{\delta_j}\right)Z_{l,q} \\ &+\int_{B^+_\sigma(\eta(\e)\tau_l)}\mathcal S_g\frac{1}{\delta_l^\frac{n-2}{2}}\mathfrak z_q\left(\frac{x-\eta(\e)\tau_l}{\delta_l}\right)Z_{j,i} \\ &= \int_{\Omega_\sigma}\mathcal S_gZ_{j,i}Z_{l,q} +\delta_j^\frac{n-2}{2}\int_{B^+_{\frac{\sigma}{\delta_j}}}\mathcal S_g(\delta_jy+\eta(\e)\tau_j)\mathfrak z_i(y)Z_{l,q}(\delta_jy+\eta(\e)\tau_j)dy \\ &+ \delta_l^\frac{n-2}{2}\int_{B^+_{\frac{\sigma}{\delta_l}} }\mathcal S_g(\delta_ly+\eta(\e)\tau_l)\mathfrak z_q(y)Z_{j,i}(\delta_ly+\eta(\e)\tau_l)dy \\ &= \left\lbrace \begin{array}{ll} \bigo{\delta_j^\frac{n}{2}\delta_l^\frac n2}&\text{if}\hsp 1\leq i,q<n, \\[0.15cm] \bigo{\delta_j^\frac{n}{2}\delta_l^\frac{n-2}{2}} &\text{if}\hsp 1<i<q=n,\\[0.15cm] \bigo{\delta_j^\frac{n-2}{2}\delta_l^\frac{n-2}{2}} &\text{if}\hsp i=q=n. \end{array}\right.
\end{align*}
Similarly,
\begin{align*}
	c_n\int_M\nabla \Z_{j,i}\cdot\nabla\Z_{l,q} &= c_n\int_{\Omega_\sigma}\nabla Z_{j,i}\nabla Z_{l,q} +c_n\int_{B^+_\sigma(\eta(\e)\tau_j)}\frac{1}{\delta_j^\frac{n}{2}}\nabla \mathfrak z_i\left(\frac{x-\eta(\e)\tau_j}{\delta_j}\right)\cdot\nabla Z_{l,q} \\ &+c_n\int_{B^+_\sigma(\eta(\e)\tau_l)}\frac{1}{\delta_l^\frac{n}{2}}\mathfrak \nabla j_q\left(\frac{x-\eta(\e)\tau_l}{\delta_l}\right)\cdot\nabla Z_{j,i} \\ &= c_n\int_{\Omega_\sigma}\nabla Z_{j,i}\nabla Z_{l,q} +c_n\delta_j^\frac{n}{2}\int_{B^+_{\frac{\sigma}{\delta_j}}}\nabla\mathfrak z_i(y)\cdot\nabla Z_{l,q}(\delta_jy+\eta(\e)\tau_j)dy \\ &+ c_n\delta_l^\frac{n}{2}\int_{B^+_{\frac{\sigma}{\delta_l}}}\nabla\mathfrak z_q(y)\cdot\nabla Z_{j,i}(\delta_ly+\eta(\e)\tau_ldy \\ &= \left\lbrace \begin{array}{ll} \bigo{\delta_j^\frac{n}{2}\delta_l^\frac n2(1+\abs{\ln  \delta_j}+\abs{\ln  \delta_l})}&\text{if}\hsp 1\leq i,q<n, \\[0.15cm] \bigo{\delta_j^\frac{n}{2}\delta_l^\frac{n-2}{2}(1+\abs{\ln \delta_j})} &\text{if}\hsp 1<i<q=n,\\[0.15cm] \bigo{\delta_j^\frac{n-2}{2}\delta_l^\frac{n-2}{2}} &\text{if}\hsp i=q=n. \end{array}\right. 
\end{align*}
\end{proof}
\begin{lemma}\label{A:exp} For any $\phi\in H^1(M)$ it holds true
\begin{align}\label{exp:int}
\int_M\abs{K}\abs{\W}^\frac{4}{n-2}\phi\Z_{l,q} &= \abs{K}\int_{B^+_{\frac{\sigma}{\delta_l}}}U^\frac{4}{n-2}\phi_l\:\mathfrak z_q + o(1),\\ \label{exp:bound} \int_{\de M}H\abs{\W}^\frac{2}{n-2}\phi \Z_{l,q} &= H\int_{\partial'B^+_{\frac{\sigma}{\delta_l}}}U^\frac{2}{n-2} \phi_l \:\mathfrak z_q + o(1).
\end{align}
where $\phi_l(y)=\delta_l^\frac{n-2}{2}\phi(\delta_ly+\eta(\e)\tau_l)$ and $\sigma>0$ is small enough.
\end{lemma}
\begin{proof} For the sake of brevity we will only prove \eqref{exp:int}, as the proof of \eqref{exp:bound} follows the same argument. Take $\sigma>0$ small enough such that 
	\begin{equation*}
		\left\lbrace B^+_\sigma(\eta(\e)\tau_\gamma):\gamma =1,\ldots,k. \right\rbrace
	\end{equation*}
is a disjoint family, and denote by $\Omega_\sigma= B^+_R\setminus\bigcup_{\gamma=1}^kB^+_\sigma(\eta(\e)\tau_\gamma)$. Then,
\begin{align*}
&\int_M\abs{K}\abs{\W}^\frac{4}{n-2}\phi\Z_{l,q} = \int_{B^+_R}\abs{K}\abs{\W}^\frac{4}{n-2}\phi Z_{l,q} \\ &=\int_{\Omega_\sigma}\abs{K}\abs{\W}^\frac{4}{n-2}\phi Z_{l,q}+\int_{B^+_\sigma(\eta(\e)\tau_l)}\abs{K}\abs{\W}^\frac{4}{n-2}\phi Z_{l,q}+\sum_{\substack{\gamma=1\\\gamma \neq l}}^k\int_{B^+_\sigma(\eta(\e)\tau_\gamma)}\abs{K}\abs{\W}^\frac{4}{n-2}\phi Z_{l,q}
\\ &\simeq \int_{\Omega_\sigma}\abs{K}\abs{\W}^\frac{4}{n-2}\phi Z_{l,q} +\int_{B^+_\sigma(\eta(\e)\tau_l)}\abs{K}\abs{\W_l}^\frac{4}{n-2}\phi Z_{l,q} + \sum_{\substack{\alpha=1\\\alpha \neq l}}\int_{B^+_\sigma(\eta(\e)\tau_l)}\abs{K}\abs{\W_\alpha}^\frac{4}{n-2}\phi Z_{l,q} \\ &+ \sum_{\substack{\gamma=1\\\gamma \neq l}}^k\int_{B^+_\sigma(\eta(\e)\tau_\gamma)}\abs{K}\abs{\W_\gamma}^\frac{4}{n-2}\phi Z_{l,q} + \sum_{\substack{\gamma=1\\\gamma \neq l}}^k\sum_{\substack{\alpha=1\\\alpha \neq \gamma}}^k\int_{B^+_\sigma(\eta(\e)\tau_\gamma)}\abs{K}\abs{\W_\alpha}^\frac{4}{n-2}\phi Z_{l,q}.
\end{align*}
We proceed to estimate every term in the right-hand side. To that aim, we will use the bound \eqref{Zdecay} together with the fact that 
\begin{equation}\label{Wdecay}
\abs{\W_\alpha(x)}\leq C\delta_\alpha^{n-2\over 2} \left(\frac{1}{\abs{x-\eta(\e)\tau_\alpha}^{n-2}}+\frac{1}{\abs{x-\eta(\e)\tau_\alpha}^{n-3}}\right).
\end{equation}
First of all, by \eqref{Zdecay} and \eqref{Wdecay}, it is easy to see that
\begin{equation*}
\int_{\Omega_\sigma}\abs{K}\abs{\W}^\frac{4}{n-2}\phi Z_{l,q} = \left\lbrace \begin{array}{ll}\bigo{(\sum_{\alpha=1}^k\delta_\alpha^2)\delta_l^\frac{n}{2}} & \text{if}\hsp 1\leq q<n,\\[0.15cm] \bigo{(\sum_{\alpha=1}^k\delta_\alpha^2)\delta_l^\frac{n-2}{2}} & \text{if}\hsp q=n. \end{array}\right.
\end{equation*}
The second addend gives us the main term of \eqref{exp:int}:
\begin{align*}
&\int_{B^+_\sigma(\eta(\e)\tau_l)}\abs{K}\abs{\W_l}^\frac{4}{n-2}\phi Z_{l,q} \\ &= \int_{B^+_\sigma(\eta(\e)\tau_l)}\abs{K}\phi \frac{1}{\delta_l^\frac{n-2}{2}}\mathfrak z_q\left(\frac{x-\eta(\e)\tau_l}{\delta_l}\right)\frac{1}{\delta_l^2} (U+\delta_lV_p)\left(\frac{x-\eta(\e)\tau_l}{\delta_l}\right)^\frac{4}{n-2}dx \\ &= \int_{B^+_\sigma(\eta(\e)\tau_l)}\abs{K}\phi_l(y)\mathfrak z_q(y)(U+\delta_lV_p)(y)^\frac{4}{n-2}dy\\ &= \int_{B^+_\sigma(\eta(\e)\tau_l)}\abs{K}\phi_l(y)\mathfrak z_q(y)U(y)^\frac{4}{n-2}dy+o(1).
\end{align*}
The rest of the terms go to zero as $\delta_l\to 0$:
\begin{align*}
&\sum_{\substack{\alpha=1\\\alpha \neq l}}\int_{B^+_\sigma(\eta(\e)\tau_l)}\abs{K}\abs{\W_\alpha}^\frac{4}{n-2}\phi Z_{l,q} \\ &\leq \sum_{\substack{\alpha=1\\\alpha \neq l}}\delta_\alpha^2\delta_l^\frac{n+2}{2}\int_{B^+_\frac{\sigma}{\delta_\alpha}}\abs{K}\phi(\delta_ly+\eta(\e)\tau_l)\mathfrak z_q(y)dy \\ &= \left\lbrace \begin{array}{ll} \bigo{\left(\sum_{\substack{\alpha=1\\\alpha \neq l}}\delta_\alpha^2\right)\delta_l^\frac{n}{2}} & \text{if}\hsp 1\leq q <n, \\ \bigo{\left(\sum_{\substack{\alpha=1\\\alpha \neq l}}\delta_\alpha^2\right)\delta_l^\frac{n-2}{2}} & \text{if}\hsp q =n.
\end{array} \right.
\end{align*}
Analogously,
\begin{align*}
	&\sum_{\substack{\gamma=1\\\gamma \neq l}}^k\int_{B^+_\sigma(\eta(\e)\tau_\gamma)}\abs{K}\abs{\W_\gamma}^\frac{4}{n-2}\phi Z_{l,q} \\ 
	&\leq\sum_{\substack{\gamma=1\\\gamma \neq l}}^k\delta_\gamma^{n-2}\int_{B^+_{\frac{\sigma}{\delta_\gamma}}}\abs{K}\abs{U+\delta_\gamma V_p}(y)^\frac{4}{n-2}\phi(\delta_\gamma y+\eta(\e)\tau_\gamma) Z_{l,q} \\ &= \left\lbrace\begin{array}{ll} \bigo{p_n(\delta_\gamma)\delta_l^\frac{n}{2}} & \text{if}\hsp 1\leq q<n, \\[0.15cm]
		\bigo{p_n(\delta_\gamma)\delta_l^\frac{n-2}{2}} & \text{if}\hsp q=n,
	\end{array} \right.
\end{align*}
being $$p_n(\delta_\gamma)=\left\lbrace\begin{array}{ll}  \delta_\gamma & \text{if}\hsp n=3, \\[0.15cm]
\delta_\gamma^2\abs{\ln \delta_\gamma} & \text{if}\hsp n=4, \\[0.15cm] \delta_\gamma^2 & \text{if}\hsp n\geq 5.
\end{array} \right.$$
Finally,
\begin{align*}
\sum_{\substack{\gamma=1\\\gamma \neq l}}^k\sum_{\substack{\alpha=1\\\alpha \neq \gamma}}^k\int_{B^+_\sigma(\eta(\e)\tau_\gamma)}\abs{K}\abs{\W_\alpha}^\frac{4}{n-2}\phi Z_{l,q}=\left\lbrace\begin{array}{ll} \bigo{\left(\sum_{\substack{\alpha=1\\\alpha \neq \gamma}}\delta_\alpha^2\right)\delta_l^{\frac n2}} & \text{if}\hsp 1\leq q<n, \\[0.15cm]
\bigo{\left(\sum_{\substack{\alpha=1\\\alpha \neq \gamma}}\delta_\alpha^2\right)\delta_l^{\frac{n-2}{2}}} & \text{if}\hsp q=n.
\end{array} \right.
\end{align*}
\end{proof}

\section{Proof of Proposition \ref{ridotto}}\label{app-e}

\begin{proof}[\bf Proof of (1)]
Take $\lambda = 1,\ldots,k$ and $s=1,\ldots,n-1.$ Since $(d_1,\ldots,d_k,\tau_1,\ldots,\tau_k)$ is a critical point for $\mathfrak J_\e$, then
\begin{equation}\label{ptocrit1}
\begin{split}
0 &= \frac{\de}{\de\tau_\lambda^s}\, \mathfrak J_\e(d_1,\ldots,d_k,\tau_1,\ldots,\tau_k) = \bigg\langle \W+\Phi_\e-\iM\left(K\mathfrak g(\W+\Phi_\e)\right)\\ &-\left.\idM\left(\frac{n-2}{2}(H\sff(\W+\Phi_\e)-\e(\W+\Phi_\e))\right),\frac{\de(\W+\Phi_\e)}{\de\tau_\lambda^s}\right\rangle.
\end{split}
\end{equation}
Using equation \eqref{eaux}, we can write
\begin{equation*}
\W+\Phi_\e-\iM\left(K\mathfrak g(\W+\Phi_\e)\right) -\idM\left(\frac{n-2}{2}(H\sff(\W+\Phi_\e)-\e(\W+\Phi_\e))\right) = \hspace{-0.2cm}\sum_{\substack{j=1,\ldots,k \\ i=1,\ldots,n}}\hspace{-0.2cm}c_{ji}\,\Z_{j,i}.
\end{equation*}
The proof concludes if we show that $c_{j,i}\to0$ as $\e\to 0$ for every $i=1,\ldots,n$ and $j=1,\ldots,k$. Taking this into account, \eqref{ptocrit1} becomes
\begin{equation*}
0 = \sum_{i,j}c_{ji}\scal{\Z_{j,i},\frac{\de\W}{\de\tau_\lambda^s}+\frac{\Phi_\e}{\de\tau_\lambda^s}} = \sum_{i,j}c_{ji}\scal{\Z_{j,i},\frac{\de\W}{\de\tau_\lambda^s}}-\scal{\frac{\Z_{j,i}}{\de\tau_\lambda^s},\Phi_\e},
\end{equation*}
where for the last identity we have used that $\Phi_\e\in \mathcal K^\perp$, so 
\begin{equation*}
	0 = \frac{\de}{\de \tau_\lambda^s}\scal{\Z_{j,i},\Phi_\e} = \scal{\Z_{j,i},\frac{\de\Phi_\e}{\de\tau_\lambda^s}}+\scal{\frac{\de \Z_{j,i}}{\de\tau_\lambda^s},\Phi_\e}.
\end{equation*}
By Lemma \ref{ptocrit:est}, 
\begin{equation*}
	0 = c_{\lambda,s} \frac{\eta(\e)^2}{\delta_\lambda^2}\norm{\nabla \,\mathfrak z_s}_{L^2(\R^n_+)}^2+\bigo{\eta(\e)^2}.
\end{equation*}
This proves $c_{\lambda,s}\to 0$ as $\e\to 0$ for every $\lambda = 1,\ldots,k$ and $s=1,\ldots,n-1.$ By taking derivatives with respect to $d_\lambda$ and arguing as above, we can prove that $c_{\lambda,n}\to 0$ as $\e\to 0$ and  claim follows.
\end{proof}
\begin{lemma}\label{ptocrit:est} For  any $\lambda=1,\ldots,k$ and $s=1,\ldots,n-1$, it holds true that
\begin{align*}
\scal{\Z_{j,i},\frac{\de \W}{\de \tau_\lambda^s}} &= \delta^{j \lambda}\delta^{is} c_n \frac{\eta(\e)}{\delta_j}\norm{\nabla \mathfrak z_i}^2_{L^2(\R^n_+)}+\bigo{\eta(\e)}, \\ \norm{\frac{\de \Z_{j,i}}{\de \tau_\lambda^s}}^2 &= \delta^{j\lambda}c_n\frac{\eta(\e)^2}{\delta_j^2}\norm{\nabla\frac{\de\,\mathfrak z_i}{\de x_s}}_{L^2(\R^n_+)}^2+\bigo{\eta(\e)^2}.
\end{align*}
\end{lemma}
\begin{proof}
First of all, observe that the following estimates hold:
\begin{equation}\label{AR:1}
	\begin{split}
\abs{\nabla^\alpha\frac{\de\Z_{j,i}}{\de\tau_\lambda^s}}&\leq \delta^{j\lambda}\times\left\lbrace\begin{array}{ll}
\delta_j^{\frac n2}\eta(\e)\abs{x-\eta{\e}\tau_j}^{-n-\alpha} & \text{if}\hsp i=1,\ldots,n-1, \\
\delta_j^{\frac{n-2}{2}}\eta(\e)\abs{x-\eta{\e}\tau_j}^{1-n-\alpha} & \text{if}\hsp i=n.
\end{array} \right. \\[0.25cm]
\abs{\frac{\de\W_j}{\de\tau_\lambda^s}}&\leq \delta^{j\lambda}\eta(\e)\delta_j^{\frac{n-2}{2}}\left(\abs{x-\eta(\e)\tau_j}^{1-n}+\abs{x-\eta(\e)\tau_j}^{2-n}\right).
\end{split}
\end{equation}
Assume $j\neq \lambda$, take $\sigma>0$ small enough such that $B_\sigma^+(\eta(\e)\tau_j)\cap B_\sigma^+(\eta(\e)\tau_\lambda)=\emptyset$, and call $\Omega_\sigma = M\setminus\left(B_\sigma^+(\eta(\e)\tau_j)\cup B_\sigma^+(\eta(\e)\tau_\lambda)\right)$. Then, by \eqref{Zdecay} and \eqref{AR:1},
\begin{align*}
&\scal{\Z_{j,i},\frac{\de \W}{\de \tau_\lambda^s}} = \scal{\Z_{j,i},\frac{\de \W_\lambda}{\de \tau_\lambda^s}} =\int_Mc_n\nabla_g\Z_{j,i}\cdot\nabla_g\frac{\de\W_\lambda}{\de\tau_\lambda^s}+ S_g\Z_{j,i}\frac{\de \W_\lambda}{\de \tau_\lambda^s} \\ &= \int_{\Omega_\sigma}c_n\nabla_g\Z_{j,i}\cdot\nabla_g\frac{\de\W_\lambda}{\de\tau_\lambda^s}+ S_g\Z_{j,i}\frac{\de \W_\lambda}{\de \tau_\lambda^s} + \int_{B^+_\sigma(\eta(\e)\tau_j)}c_n\nabla_g\Z_{j,i}\cdot\nabla_g\frac{\de\W_\lambda}{\de\tau_\lambda^s}\\&+\int_{B^+_\sigma(\eta(\e)\tau_j)} S_g\Z_{j,i}\frac{\de \W_\lambda}{\de \tau_\lambda^s}+\int_{B^+_\sigma(\eta(\e)\tau_\lambda)}c_n\nabla_g\Z_{j,i}\cdot\nabla_g\frac{\de\W_\lambda}{\de\tau_\lambda^s}+\int_{B^+_\sigma(\eta(\e)\tau_\lambda)} S_g\Z_{j,i}\frac{\de \W_\lambda}{\de \tau_\lambda^s} \\ &\leq \bigo{\eta(\e)\delta_j^{\frac{n-2}{2}}\delta_\lambda^{\frac{n-2}{2}}}+c_n\delta_\lambda^{\frac{n-2}{2}}\eta(\e)\int_{B^+_{\frac{\sigma}{\delta_j}}}\delta_j^{\frac n2}\abs{\nabla\,\mathfrak z_i}+\delta_\lambda^{\frac{n-2}{2}}\eta(\e)\int_{B^+_{\frac{\sigma}{\delta_j}}}S_g(0) \delta_j^\frac{n+2}{2}\, \abs{\mathfrak z_i}\\&+c_n\delta_j^{\frac{n-2}{2}}\eta(\e)\int_{B^+_{\frac{\sigma}{\delta_\lambda}}}\delta_\lambda^\frac{n-2}{2}\abs{\nabla\frac{\de}{\de x_s}(U+\delta_\lambda V_p)}+\delta_j^{\frac{n-2}{2}}\eta(\e)\int_{B^+_{\frac{\sigma}{\delta_\lambda}}}\delta_\lambda^\frac{n}{2}\abs{\frac{\de(U+\delta_\lambda V_p)}{\de x_s}} \\ &= \bigo{\eta(\e)\delta_j^{\frac{n-2}{2}}\delta_\lambda^{\frac{n-2}{2}}\abs{\ln \delta_\lambda}}.
\end{align*}
On the other hand, if $\lambda=j$,
\begin{align*}
&\scal{\Z_{j,i},\frac{\de \W}{\de \tau_j^s}} = \scal{\Z_{j,i},\frac{\de \W_j}{\de \tau_j^s}}=\int_Mc_n\nabla_g\Z_{j,i}\cdot\nabla_g\frac{\de\W_j}{\de\tau_j^s}+ S_g\Z_{j,i}\frac{\de \W_j}{\de \tau_j^s} \\ &=\bigo{\eta(\e)\delta_j^{n-2}}+c_n\frac{\eta(\e)}{\delta_j} \int_{B^+_{\frac{\sigma}{\delta_j}}}\nabla\,\mathfrak z_i\cdot\nabla\,\mathfrak z_s + \delta_j\eta(\e)\int_{B^+_{\frac{\sigma}{\delta_j}}}S_g(0)\,\mathfrak z_i \,\mathfrak z_s \\ &+ \delta_j^2\eta(\e)\int_{B^+_{\frac{\sigma}{\delta_j}}}S_g(0)\,\mathfrak z_i \frac{\de V_p}{\de x_s} = \delta^{is}c_n\frac{\eta(\e)}{\delta_j}\norm{\nabla\,\mathfrak z_i}^2_{L^2(\R^n_+)}+\bigo{\eta(\e)}
\end{align*}
This proves the first part of the proposition. Now, reasoning as before,
\begin{align*}
\norm{\frac{\de \Z_{j,i}}{\de \tau_\lambda^s}}^2&=\delta^{j\lambda}\int_Mc_n\abs{\nabla\frac{\de \Z_{j,i}}{\de \tau_\lambda^s}}^2+S_g\abs{\frac{\de \Z_{j,i}}{\de \tau_\lambda^s}}^2 \\ &= \delta^{j\lambda}\left(\bigo{\delta_j^{n-2}}+\frac{\eta(\e)^2}{\delta_j^2}c_n\int_{B^+_{\frac{\sigma}{\delta_j}}}\abs{\nabla \frac{\de\,\mathfrak z_i}{\de x_s}}^2+\eta(\e)^2\int_{B^+_{\frac{\sigma}{\delta_j}}}S_g(0)\abs{\frac{\de\,\mathfrak{j}_i}{\de x_s}}^2\right) \\ &=\delta^{j\lambda}\frac{\eta(\e)^2}{\delta_j^2}c_n\norm{\nabla \frac{\de\,\mathfrak{j}_i}{\de x_s}}_{L^2(\R^n_+)}^2+\bigo{\eta(\e)^2}.
\end{align*}
\end{proof}

\begin{proof}[\bf Proof of (2)]${}$\\
\begin{itemize}
\item[(i)]
First we prove that
\begin{equation}\label{step1}
J_\e(\mathcal W+\Phi_\e)=J_\e(\mathcal W)+\mathcal O(\|\Phi_\e\|^2)\end{equation} $C^0-$ uniformly with respect to $(d_1, \ldots, d_k, \tau_1,\ldots, \tau_k)$ in a compact subsets of $[0, +\infty)^k\times \mathcal C$.\\
 We will follow the arguments of the proof of Proposition \ref{errorsize}, so many details will be skipped for the sake of brevity. By Taylor expansion, there exists $\sigma \in (0,1)$ such that:
	\begin{align*}
		&J_\e(\mathcal W+\Phi_\e)-J_\e(\mathcal W) = J_\e'(\mathcal W)[\Phi_\e]+\frac{1}{2}J_\e''(\mathcal W+\sigma \Phi_\e)[\Phi_\e,\Phi_\e]  \\ &=c_n\int_M \nabla_g\mathcal W  \nabla_g\Phi_\e + \int_M \mathcal S_g \mathcal W\Phi_\e+\int_M \abs{K}\mathfrak g(\mathcal W)\Phi_\e \\ &-c_n\frac{n-2}{2}\int_{\de M }H\mathfrak f(\mathcal W)\Phi_\e+2\e (n-1)\int_{\de M} \mathcal W\Phi_\e +\frac{1}{2}\|\Phi_\e\|^2+\e(n-1)\int_{\de M} \Phi_\e^2\\& +\frac 12\int_M \abs{K}\mathfrak g'(\mathcal W+\sigma\Phi_\e)\Phi_\e^2 - c_n\frac{n-2}{4}\int_{\de M}H\mathfrak f'(\mathcal W+\sigma\Phi_\e)\Phi_\e^2.
	\end{align*}
	Immediately, by Sobolev embedding:
\begin{equation*}
		\e(n-1)\int_{\de M} \Phi_\e^2 \lesssim \e \norm{\Phi_\e}^2, 
\end{equation*}	
Instead, $$\int_M \mathcal S_g \mathcal W\Phi_\e =\sum_{j=1}^k\int_M \mathcal S_g \mathcal U_j \Phi_\e +\d_j \sum_{j=1}^k\int_M \mathcal S_g\mathcal V_j\Phi_\e$$ and  
\begin{align*}
	\int_M \mathcal S_g\,\mathcal U_j \Phi_\e &\lesssim\norm{\Phi_\e}\norm{\mathcal U_j}_{L^\frac{2n}{n+2}(M)} &\lesssim\norm{\Phi_\e}\times\left\lbrace\begin{array}{ll}
			\mathcal O\left(\delta_j^2\right) & \text{if} \hsp n\geq 7, \\
			\mathcal  O\left(\delta_j^2 \abs{\ln (\delta_j)}^\frac23\right) & \text{if} \hsp n=6 \\
			\mathcal O\left(\delta_j^\frac{n-2}{2}\right) & \text{if} \hsp n=4,5.
		\end{array} \right., 
\end{align*}	
	\begin{align*}
		\delta_j \int_M\mathcal S_g\mathcal V_j\Phi_\e &\lesssim \delta \norm{\Phi_\e}\norm{\mathcal V_j}_{L^2(M)}&\lesssim\norm{\Phi_\e}\times \left\lbrace\begin{array}{ll}
		\mathcal  O\left(\delta_j^2\right) & \text{if} \hsp n\geq 7, \\
		\mathcal O\left(\delta_j^2 \abs{\ln (\delta)}^\frac12\right) & \text{if} \hsp n=6 \\
		\mathcal O\left(\delta_j^\frac{n-2}{2}\right) & \text{if} \hsp n=4,5.
	\end{array} \right.
\end{align*}
Moreover
\begin{align*}
		\e\int_{\de M}\: \mathcal U_j \Phi_\e &\lesssim \e \norm{\Phi_\e}\norm{\mathcal U_j}_{L^\frac{2(n-1)}{n}(\de M)} &\lesssim\e\norm{\Phi_\e}_{H^1(M)}\times \left\lbrace\begin{array}{ll} \mathcal O(\delta_j) & \text{if} \hsp n\geq 5,\\ \mathcal O(\delta_j \abs{\ln (\delta_j)}^\frac23)& \text{if} \hsp n=4.\end{array} \right.,
\end{align*}
\begin{align*}
		\e\int_{\de M}\mathcal V_j \Phi_\e &\lesssim \e \norm{\Phi_\e}\norm{\mathcal V_j}_{L^2(\de M)}&\lesssim\e\norm{\Phi_\e}\times \left\lbrace\begin{array}{ll} \mathcal O(\delta_j^\frac{3}{2}) & \text{if} \hsp n\geq 6,\\ \mathcal O(\delta_j^\frac{3}{2} \abs{\ln (\delta_j)}^\frac{1}{2}) & \text{if} \hsp n=5 \\ \mathcal O(\delta_j^\frac{n-2}{2}) & \text{if}\hsp n=4.\end{array} \right..
	\end{align*}
	Integrating by parts,
	\begin{align*}
		&c_n\int_M \nabla_g\mathcal W\nabla_g\Phi_\e + \int_M |K|\mathfrak g(\mathcal W) \Phi_\e -2(n-1)\int_{\de M}H\mathfrak f(\mathcal W)\Phi_\e \\ &= \sum_{j=1}^k\left(\int_M \left(-c_n\Delta_g\mathcal W_j \Phi_\e+|K|\mathfrak g(\mathcal W_j)\Phi_\e\right)-2(n-1)\int_{\de M}H \mathfrak f(W_j)\Phi_\e+\int_{\de M}c_n \frac{\de \mathcal W_j}{\de\nu}\Phi_\e\right)\\ &+\int_M |K| \left(\mathfrak g(\mathcal W)-\mathfrak g\left(\sum_j \mathcal W_j\right)\right)\Phi_\e -2(n-1)\int_{\de M}H\left(\mathfrak f(\mathcal W)-\mathfrak f\left(\sum_j \mathcal W_j\right)\right)\Phi_\e \\[0.2cm] 
		&\lesssim \norm{\Phi_\e}\times \left[\left\lbrace\begin{array}{ll} \mathcal O(\delta_j^2) & \text{if} \hsp n\geq 5,\\ \mathcal O(\delta_j^2 \abs{\ln (\delta_j)}^\frac23) & \text{if} \hsp n=4 \end{array} \right.+\left\lbrace\begin{array}{ll} \mathcal O(\e^{(1-\alpha)\frac n 2}) & \text{if} \hsp n\geq 5,\\ \mathcal O\left(((\rho(\e))^2|\ln \rho(\e)|^{\frac 23}\right) & \text{if}\hsp n=4. \end{array} \right.\right]
	\end{align*}
Finally,
$$\int_M |K| \mathfrak g'(\mathcal W+\sigma \Phi_\e)\Phi_\e^2 \lesssim \|\Phi_\e\|^2\|\mathcal W+\sigma\Phi_\e\|^{\frac{4}{n-2}}_{L^{\frac{2n}{n-2}}(M)}\lesssim \|\Phi_\e\|^2$$ and similarly

$$\int_{\de M} |K| \mathfrak f'(\mathcal W+\sigma \Phi_\e)\Phi_\e^2 \lesssim \|\Phi_\e\|^2\|\mathcal W+\sigma\Phi_\e\|^{\frac{2}{n-2}}_{L^{\frac{2(n-1)}{n-2}}(\de M)}\lesssim \|\Phi_\e\|^2.$$	
Collecting all the above estimates \eqref{step1} follows.
\\

\item[(ii)]Next we estimate the leading term $J_\e(\mathcal W)$.\\

We claim that
\begin{equation}\label{je}\begin{aligned}
 J_\e(\mathcal W)&:=\underbrace{\sum_{i=1}^k J_\e (\mathcal W_i)}_{(I)}-\underbrace{\sum_{j<i}\int_M K \mathfrak g(\mathcal W_i)\mathcal W_j\,d\nu_g}_{(II)}-\underbrace{\sum_{j<i}c_n\frac{n-2}{2}\int_{\partial M} H \mathfrak f(\mathcal W_i)\mathcal W_j\, d\sigma_g}_{(III)}\\
 &+\sum_{i<j}\int_M\left(c_n \nabla_g \mathcal W_i\nabla_g\mathcal W_j+\mathcal S_g \mathcal W_i \mathcal W_j -K\mathfrak g(\mathcal W_i)\mathcal W_j\right)\, d\nu_g\\
 &-c_n \frac{n-2}{2}\int_{\de M} H\left(\mathfrak F\left(\sum_{i=1}^k \mathcal W_i\right)-\sum_{i=1}^k \mathfrak F(\mathcal W_i)-\mathfrak f (\mathcal W_i)\mathcal W_j\right)\, d\sigma_g\\
 &-\int_M K\left(\mathfrak G\left(\sum_{i=1}^k \mathcal W_i\right)-\sum_{i=1}^k \mathfrak G(\mathcal W_i)-\mathfrak g (\mathcal W_i)\mathcal W_j\right)\, d\nu_g+(n-1)\e\sum_{i\neq j}\int_{\de M} \mathcal W_i\mathcal W_j\, d\sigma_g\\ 
&=k\mathfrak E -\sum\limits_{i=1}^k\zeta_n(\delta_i)\left[\mathfrak b_n\|\pi(p_i)\|^2+o'_n(1)\right] -\e\sum\limits_{i=1}^k\delta_i(\mathfrak c_n+o''_n(1))
\\
&-\sum_{j<i}{\mathfrak d}_n\frac{1}{|K|^{\frac{n-2}{2}}}\frac{\d_i^{\frac{n-2}{2}}\d_j^{\frac{n-2}{2}}}{\eta(\e)^{n-2}}\frac{1}{|\tau_i-\tau_j|^{n-2}} 
+o\left(\frac{\d_i^{\frac{n-2}{2}}\d_j^{\frac{n-2}{2}}}{\eta(\e)^{n-2}}\right).
\end{aligned}\end{equation}

The contribution of each single bubble is encoded in the first  term (I) whose expansion is given in  Proposition \ref{energy-bubble}.
All the other terms   come from the interaction among different bubble. First we estimate the leading term    $(II)+(III)$.\\ 
For any $h=1, \ldots, k$ let $B_h^+:=B^+\left(\eta(\e)\tau_h, \eta(\e)\frac \sigma 2\right)\subset B^+_R$ provide $\sigma$ is small enough and moreover $B_h^+$ are disjoint each other and $\de' B^+_h =B_h^+\cap \partial\mathbb R^n_+$.\\ 
$$\begin{aligned} (II)&=\int_{B_i^+}K \mathfrak g\left(\mathcal W_i(x)\right)\mathcal W_j(x)|g(x)|^{\frac 12}\, dx+\int_{B^+_R\setminus B^+_i}K \mathfrak g\left(\mathcal W_i(x)\right)\mathcal W_j(x)|g(x)|^{\frac 12}\, dx\\
&+\int_{B^+_R}(1-\chi^\dst(|x|))K \mathfrak g\left(\mathcal W_i(x)\right)\mathcal W_j(x)|g(x)|^{\frac 12}\, dx\\
\end{aligned}$$
Now, the main term in (II) is given by
$$ \begin{aligned} &\int_{B^+_i}K \mathfrak g\left(\mathcal W_i(x)\right)\mathcal W_j(x)|g(x)|^{\frac 12}\, dx= \\
&\int_{B_i^+}K \mathfrak g\left(\delta_i^{-\frac{n-2}{2}}U\left(\frac{x-\eta(\e) \tau_i}{\delta_i}\right)+\delta_i \delta_i^{-\frac{n-2}{2}}V_p\left(\frac{x-\eta(\e)\tau_i}{\delta_i}\right)\right)\times\\
&\quad \times \left(\delta_j^{-\frac{n-2}{2}}U\left(\frac{x-\eta(\e) \tau_j}{\delta_j}\right)+\delta_j \delta_j^{-\frac{n-2}{2}}V_p\left(\frac{x-\eta(\e)\tau_j}{\delta_j}\right)\right)|g(x)|^{\frac 12}\, dx\\
&=\delta_i^{\frac{n-2}{2}}\delta_j^{\frac{n-2}{2}}\frac{\alpha_n}{|K|^{\frac{n-2}{4}}}\int_{B^+_{\frac{\eta(\e)\sigma}{2\delta_i}}}K \mathfrak g(U+\delta_i V_p)\times\\
&\times\left(\frac{1}{\left(|\delta_i \tilde y +\eta(\e) (\tilde\tau_i-\tilde\tau_j)|^2+(\delta_i y_n+\delta_j\mathfrak D_n)^2-\delta_j^2 \right)^{\frac{n-2}{2}}}\right)|g(\delta_i y+\eta(\e) \tau_i)|^{\frac 12}\, dy\\ 
&+\mathcal O\left(\frac{\delta_i^{\frac{n-2}{2}}\delta_j^{\frac{n-2}{2}}}{\eta(\e)^{n-3}}\int_{B^+_{\frac{\eta(\e)\sigma}{2\delta_i}}}\mathfrak g(U+\delta_i V_p)\, dy\right)\\
&=\frac{\alpha_n}{|K|^{\frac{n-2}{4}}}\frac{\delta_i^{\frac{n-2}{2}}\delta_j^{\frac{n-2}{2}}}{\eta(\e)^{n-2}|\tau_i-\tau_j|^{n-2}}K\int_{\mathbb R^n_+}U^{\dst-1}\, dy+o\left(\frac{\delta_i^{\frac{n-2}{2}}\delta_j^{\frac{n-2}{2}}}{\eta(\e)^{n-2}}\right)\\
&=-\frac{\alpha_n^\dst}{|K|^{\frac{n-2}{2}}} \frac{\delta_i^{\frac{n-2}{2}}\delta_j^{\frac{n-2}{2}}}{\eta(\e)^{n-2}|\tau_i-\tau_j|^{n-2}}\int_{\mathbb R^n_+}\frac{1}{(|\tilde y|^2 + (y_n+\mathfrak D_n)^2-1)^{\frac{n+2}{2}}}\, dy+o\left(\frac{\delta_i^{\frac{n-2}{2}}\delta_j^{\frac{n-2}{2}}}{\eta(\e)^{n-2}}\right)\\
&=- {\underbrace{\alpha_n^\dst\omega_{n-1}\int_0^{+\infty}\frac{r^{n-2}}{(1+r^2)^{\frac{n+2}{2}}}\,dr}_{:={\mathfrak d}_n}\varphi_{\frac 32}}\frac{\delta_i^{\frac{n-2}{2}}\delta_j^{\frac{n-2}{2}}}{\eta(\e)^{n-2}|\tau_i-\tau_j|^{n-2}}\frac{1}{|K|^{\frac{n-2}{2}}}+o\left(\frac{\delta_i^{\frac{n-2}{2}}\delta_j^{\frac{n-2}{2}}}{\eta(\e)^{n-2}}\right)\\
&= -{\mathfrak d}_n \left(\frac{\mathfrak D_n}{(\mathfrak D_n^2-1)^{\frac 12}}-1\right)\frac{\delta_i^{\frac{n-2}{2}}\delta_j^{\frac{n-2}{2}}}{\eta(\e)^{n-2}|\tau_i-\tau_j|^{n-2}}\frac{1}{|K|^{\frac{n-2}{2}}}+o\left(\frac{\delta_i^{\frac{n-2}{2}}\delta_j^{\frac{n-2}{2}}}{\eta(\e)^{n-2}}\right).
\end{aligned}$$
Now $$\begin{aligned} &\left|\int_{B_R^+\setminus B^+_i}K \mathfrak g\left(\mathcal W_i(x)\right)\mathcal W_j(x)|g(x)|^{\frac 12}\, dx\right|\\ &\lesssim \int_{\mathbb R^n_+\setminus B^+(\eta(\e)\tau_i, \eta(\e) \frac \sigma 2)}\frac{\delta_i^{\frac{n+2}{2}}}{|x-\eta(\e)\tau_i|^{n+2}}\frac{\delta_j^{\frac{n-2}{2}}}{|x-\eta(\e)\tau_j|^{n-2}}\, dx= (\hbox{setting}\,\, x=\eta(\e) y)\\ 
&=c \frac{\delta_i^{\frac{n+2}{2}}\delta_j^{\frac{n-2}{2}}}{\eta(\e)^{n}}\int_{\mathbb R^n_+\setminus B^+(\tau_i, \frac \sigma 2)}\frac{1}{|y-\tau_i|^{n+2}}\frac{1}{|y-\tau_j|^{n-2}}\, dx\\
&=o\left(\frac{\delta_i^{\frac{n-2}{2}}\delta_j^{\frac{n-2}{2}}}{\eta(\e)^{n-2}}\right)\end{aligned}$$
and similarly $$\left|\int_{B_R^+}(1-\chi^\dst(|x|))K {\mathfrak g}\left(\mathcal W_i(x)\right)\mathcal W_j(x)|g(x)|^{\frac 12}\, dx\right|=o\left(\frac{\delta_i^{\frac{n-2}{2}}\delta_j^{\frac{n-2}{2}}}{\eta(\e)^{n-2}}\right).$$
For the interaction terms $(III)$, we argue as before, obtaining that
$$\begin{aligned}&2(n-1)\int_{\partial M}H \mathfrak f(\mathcal W_i)\mathcal W_j\\ &=2(n-1)\alpha_n\frac{\delta_i^{\frac{n-2}{2}}\delta_j^{\frac{n-2}{2}}}{\eta(\e)^{n-2}|\tau_i-\tau_j|^{n-2}}\frac{H}{|K|^{\frac{n-2}{4}}}\int_{\mathbb R^{n-1}}U^{\dsh-1}\, d\tilde x(1+o(1))\\
&= {\underbrace{\frac{2(n-1)\alpha_n^{\dsh}\omega_{n-1}}{\sqrt{n(n-1)}}\int_0^{+\infty}\frac{r^{n-2}}{(1+r^2)^{\frac n 2}}\, dr}_{:=\mathfrak h_n}}\frac{\delta_i^{\frac{n-2}{2}}\delta_j^{\frac{n-2}{2}}}{\eta(\e)^{n-2}|\tau_i-\tau_j|^{n-2}}\frac{\mathfrak D_n}{|K|^{\frac{n-2}{2}}(\mathfrak D_n^2-1)^{\frac 12}}(1+o(1)).\end{aligned}$$
Then, 
$$\begin{aligned}
(II)+(III)& =\sum_{j<i}{\mathfrak d}_n\frac{1}{|K|^{\frac{n-2}{2}}}\frac{\d_i^{\frac{n-2}{2}}\d_j^{\frac{n-2}{2}}}{\eta(\e)^{n-2}}\frac{1}{|\tau_i-\tau_j|^{n-2}} 
+o\left(\frac{\d_i^{\frac{n-2}{2}}\d_j^{\frac{n-2}{2}}}{\eta(\e)^{n-2}}\right)\\
\end{aligned}$$
since a simple computation shows that ${\mathfrak d}_n-{\mathfrak h_n}=0$.\\
Now we evaluate the remaining terms.\\
For $i\neq j$ $$\begin{aligned}&\left|\e \int_{\partial M}\mathcal W_i \mathcal W_j\, d\nu_g\right|
\lesssim\e \frac{\delta_i^{\frac{n-2}{2}}\delta_j^{\frac{n-2}{2}}}{\eta(\e)^{n-3}}\int_{\mathbb R^{n-1}}\frac{1}{|\tilde y-\tilde \tau_i|^{n-2}}\frac{1}{|\tilde y-\tilde \tau_j|^{n-2}}\,d\tilde y=o \left(\frac{\delta_i^{\frac{n-2}{2}}\delta_j^{\frac{n-2}{2}}}{\eta(\e)^{n-2}}\right)\end{aligned}$$
Now $$\begin{aligned}&
\sum_{i<j}\left[\int_M \nabla_g\mathcal W_i\nabla_g\mathcal W_j+\mathcal S_g \mathcal W_i\mathcal W_j -K \mathfrak g(\mathcal W_i)\mathcal W_j\right]\, d\nu_g -2(n-1)\sum_{i<j}\int_{\partial M}H \mathfrak f(\mathcal W_i)\mathcal W_j\,d\sigma_g\\
&=\sum_{i<j}\left[\int_M\left(-c_n\Delta_g \mathcal W_i +\mathcal S_g \mathcal W_i-K\mathfrak g(\mathcal W_i)\right)\mathcal W_j\, d\nu_g\right]+\sum_{i<j}c_n\int_{\partial M}\left(\frac{\partial \mathcal W_i}{\partial\nu}-\frac{n-2}{2}H \mathfrak f(\mathcal W_i)\right)\mathcal W_j\, d\sigma_g\end{aligned}$$ and by using \eqref{lp} we get that
$$\left|-c_n\Delta_g \mathcal W_i +\mathcal S_g \mathcal W_i-Kg(\mathcal W_i)\right|\lesssim \frac{\delta_i^{\frac{n-2}{2}}}{(|\tilde x -\eta(\e)\tilde\tau_i|^2+(x_n-\eta(\e)\tau_{i, n}+\mathfrak D_n\delta_i)^2-\delta_i^2)^{\frac{n-2}{2}}}$$
Hence $$\begin{aligned}&\sum_{i<j}\left[\int_M\left(-c_n\Delta_g \mathcal W_i +\mathcal S_g \mathcal W_i-Kg(\mathcal W_i)\right)\mathcal W_j\, d\nu_g\right]\lesssim \sum_{i<j}\int_{B_R^+}\frac{\delta_i^{\frac{n-2}{2}}}{|x -\eta(\e)\tau_i|^{n-2}}\frac{\delta_j^{\frac{n-2}{2}}}{|x -\eta(\e)\tau_j|^{n-2}}\\
&\lesssim \frac{\delta_i^{\frac{n-2}{2}}\delta_j^{\frac{n-2}{2}}}{\eta(\e)^{2(n-2)}}\eta(\e)^n\int_{\mathbb R^n_+}\frac{1}{|y-\tau_i|^{n-2}}\frac{1}{|y-\tau_j|^{n-2}}=o\left(\frac{\d_i^{\frac{n-2}{2}}\d_j^{\frac{n-2}{2}}}{\eta(\e)^{n-2}}\right)
 \end{aligned}$$
 and similary 
 $$\begin{aligned}&\sum_{i<j}c_n\int_{\partial M}\left(\frac{\partial \mathcal W_i}{\partial\nu}-\frac{n-2}{2}H f(\mathcal W_i)\right)\mathcal W_j\, d\sigma_g=o\left(\frac{\d_i^{\frac{n-2}{2}}\d_j^{\frac{n-2}{2}}}{\eta(\e)^{n-2}}\right)\\
 \end{aligned}$$ Now
$$\begin{aligned}&\int_M K\left[\mathfrak G\left(\sum_{j=1}^k \mathcal W_j\right)-\sum_{j=1}^k \mathfrak G(\mathcal W_j)-\sum_{i\neq j}\mathfrak g(\mathcal W_i)\mathcal W_j\right]\\
&=\sum_{h=1}^k\int_{B_h^+} K\left[\mathfrak G\left(\sum_{j=1}^k  W_j\right)-\sum_{j=1}^k \mathfrak G(W_j)-\sum_{i\neq j}\mathfrak g(W_i) W_j\right]|g(x)|^{\frac 12}\, dx\\
&+\int_{B_R^+\setminus \bigcup_h B_h^+}K\left[\mathfrak G\left(\sum_{j=1}^k W_j\right)-\sum_{j=1}^k \mathfrak G( W_j)-\sum_{i\neq j}\mathfrak g( W_i) W_j\right]|g(x)|^{\frac 12}\, dx\\
&+\int_{B_R^+}(1-\chi^\dst(|x|)) K\left[\mathfrak G\left(\sum_{j=1}^k W_j\right)-\sum_{j=1}^k \mathfrak G(W_j)-\sum_{i\neq j}\mathfrak g(W_i)W_j\right]|g(x)|^{\frac 12}\, dx\end{aligned}$$
It is immediate that
$$\begin{aligned}&\left|\int_{B_R^+}(1-\chi^\dst(|x|)) K\left[\mathfrak G\left(\sum_{j=1}^k W_j\right)-\sum_{j=1}^k \mathfrak G(W_j)-\frac{2n}{n-2}\sum_{i\neq j}\mathfrak g(W_i)W_j\right]|g(x)|^{\frac 12}\, dx\right|\\
&=\mathcal O\left(\d_j^n +\d_i^2\frac{\d_j^{\frac{n-2}{2}}\d_i^{\frac{n-2}{2}}}{\eta(\e)^{n-2}}\right).\end{aligned}$$
Now, outside the $k-$ balls
$$\begin{aligned}&\left|\int_{B_R^+\setminus \bigcup_h B_h^+}K\left[\mathfrak G\left(\sum_{j=1}^k W_j\right)-\sum_{j=1}^k \mathfrak G( W_j)-\sum_{i\neq j}\mathfrak g( W_i) W_j\right]|g(x)|^{\frac 12}\, dx\right|\\
&\lesssim \sum_{i\neq j}\int_{B_R^+\setminus \bigcup_h B_h^+}\left(|W_i|^{\dst-2}W_j^2+|W_j|^{\dst-2}W_i^2\right)\, dx=o\left(\frac{\d_i^{\frac{n-2}{2}}\d_j^{\frac{n-2}{2}}}{\eta(\e)^{n-2}}\right)\end{aligned}$$
because if $i\neq j$ $$\begin{aligned}\int_{B_R^+\setminus \bigcup_h B_h^+}|W_i|^{\dst-2}W_j^2&\lesssim \int_{B_R^+\setminus \bigcup_h B_h^+}\frac{\delta_i^2}{|x-\eta(\e)\tau_i|^4}\frac{\delta_j^{n-2}}{|x-\eta(\e)\tau_j|^{2(n-2)}}\, dx\\
&\lesssim \frac{\delta_i^2\delta_j^{n-2}}{\eta(\e)^{n}}=o\left(\frac{\d_i^{\frac{n-2}{2}}\d_j^{\frac{n-2}{2}}}{\eta(\e)^{n-2}}\right).\end{aligned}$$
On each ball $B_h^+$ we also have
$$\begin{aligned}&\int_{B_h^+}\left|K\left[\mathfrak G\left(\sum_{j=1}^k W_j\right)-\sum_{j=1}^k \mathfrak G(W_j)-\sum_{i\neq j}\mathfrak g(W_i)W_j\right]|g(x)|^{\frac 12}\right|\, dx\\
&\leq \int_{B_h^+}\left|\mathfrak G\left(W_h+\sum_{i\neq h}W_i\right)-\mathfrak G(W_h)-\sum_{j\neq h}\mathfrak g(W_h)W_j\right|\, dx+\sum_{i\neq h}\int_{B_h^+}|\mathfrak G(W_i)|\, dx\\
&+\sum_{i\neq h\atop j\neq i}\int_{B_h^+}|\mathfrak g(W_h)W_j|\, dx\lesssim \sum_{i\neq h}\int_{B_h^+}W_h^{\dst-2}W_i^2+c\sum_{i\neq h}\int_{B_h^+}W_i^\dst\, dx\\
&+c \sum_{i\neq h\atop j\neq i}\int_{B_h^+}W_i^{\dst-1}W_j\, dx\end{aligned}$$
Now if $i\neq h$ and $n\geq 5$ then
$$\begin{aligned}\int_{B_h^+}W_h^{\dst-2}W_i^2 &\lesssim \int_{B_h^+}\frac{\delta_h^2}{|x-\eta(\e) \tau_h|^4}\frac{\delta_i^{n-2}}{|x-\eta(\e)\tau_i|^{2(n-2)}}\, dx\\
&\lesssim\frac{\delta_h^2\delta_i^{n-2}}{\eta(\e)^{n}}\int_{B^+(\tau_h, \sigma/2)}\frac{1}{|y-\tau_h|^4}\frac{1}{|y-\tau_i|^{2(n-2)}}=o\left(\frac{\d_i^{\frac{n-2}{2}}\d_j^{\frac{n-2}{2}}}{\eta(\e)^{n-2}}\right).\end{aligned}$$
If, instead, $n=4$ then
$$\begin{aligned}\int_{B_h^+}W_h^{\dst-2}W_i^2 &\lesssim\frac{\d_h^{n-2}\d_i^{\frac{n-2}{2}}}{\eta(\e)^{n-2}} \int_{B^+_{\frac{\eta(\e)}{\d_h}}}\frac{1}{(|x|^2+1)^2}\, dx\\
&\lesssim \frac{\d_h^{n-2}\d_i^{\frac{n-2}{2}}}{\eta(\e)^{n-2}}\left|\ln \frac{\eta(\e)}{\d_h}\right|=o\left(\frac{\d_i^{\frac{n-2}{2}}\d_j^{\frac{n-2}{2}}}{\eta(\e)^{n-2}}\right).\end{aligned}$$
If $i, j\neq h$ then
$$\begin{aligned} \int_{B^+_h}W_i^{\dst-1}W_j&\leq \int_{B_h^+}\frac{\d_i^{\frac{n+2}{2}}}{|x-\eta(\e)\tau_i|^{n+2}}\frac{\d_j^{\frac{n-2}{2}}}{|x-\eta(\e)\tau_j|^{n-2}}\\
&\leq \frac{\d_i^{\frac{n+2}{2}}\d_j^{\frac{n-2}{2}}}{\eta(\e)^n}=o\left(\frac{\d_i^{\frac{n-2}{2}}\d_j^{\frac{n-2}{2}}}{\eta(\e)^{n-2}}\right).\end{aligned}$$
If $i\neq h$ then
$$\begin{aligned} \int_{B^+_h}W_i^{\dst-1}W_h&\leq \int_{B_h^+}\frac{\d_i^{\frac{n+2}{2}}}{|x-\eta(\e)\tau_i|^{n+2}}\frac{\d_h^{\frac{n-2}{2}}}{|x-\eta(\e)\tau_h|^{n-2}}\\
&\leq \frac{\d_i^{\frac{n+2}{2}}\d_h^{\frac{n-2}{2}}}{\eta(\e)^n}=o\left(\frac{\d_i^{\frac{n-2}{2}}\d_j^{\frac{n-2}{2}}}{\eta(\e)^{n-2}}\right)\end{aligned}$$
Finally, if $i\neq h$
$$\begin{aligned} \int_{B^+_h}W_i^{\dst}&\leq \int_{B_h^+}\frac{\d_i^n}{|x-\eta(\e)\tau_i|^{n}}\leq \frac{\d_i^n}{\eta(\e)^n}=o\left(\frac{\d_i^{\frac{n-2}{2}}\d_j^{\frac{n-2}{2}}}{\eta(\e)^{n-2}}\right)\end{aligned}$$
In a similar way
$$\begin{aligned} &\int_{\partial M}\left[\mathfrak F\left(\sum_i  \mathcal W_i\right)-\sum_i \mathfrak F(\mathcal W_i)-\sum_{i\neq j}\mathfrak f( \mathcal W_i)\mathcal W_j\right]\, d\nu_g\\
&=\sum_{h=1}^k\int_{\partial' B_h^+}\left[\ldots\ldots\right]|g(x)|^{\frac 12}\, dx +\int_{\partial'B_R^+\setminus \bigcup_h \partial'B_h^+}\left[\ldots\ldots\right]|g(x)|^{\frac 12}\, dx\\
&+\int_{\partial' B_R^+}(1-\chi^\dsh)\left[\ldots\ldots\right]|g(x)|^{\frac 12}\\
&=o\left(\frac{\d_i^{\frac{n-2}{2}}\d_j^{\frac{n-2}{2}}}{\eta(\e)^{n-2}}\right).\end{aligned}$$

\end{itemize}

Let us look at the main term of \eqref{je}. Let  $\mathfrak Q(p)$ be the quadratic form associated with the second derivative of $p\to\|\pi(p)\|^2$ (being zero the first derivative)
If $n=4$  by \eqref{deltaj4}  
$$\begin{aligned}&\sum\limits_{i=1}^k( \d_i^2|\ln \d_i|)\left[\mathfrak b_4\|\pi(p_i)\|^2+o'_n(1)\right] -\e\sum\limits_{i=1}^k\delta_i(\mathfrak c_n+o''_n(1))
\\
&-\sum_{j<i}{\mathfrak d}_n\frac{1}{|K|^{\frac{n-2}{2}}}\frac{\d_i^{\frac{n-2}{2}}\d_j^{\frac{n-2}{2}}}{\eta ^{n-2}}\frac{1}{|\tau_i-\tau_j|^{n-2}} 
+o\left(\frac{\d_i^{\frac{n-2}{2}}\d_j^{\frac{n-2}{2}}}{\eta ^{n-2}}\right)
\\ &=\sum\limits_{i=1}^k\rho^2\left(d_0+\eta d_j\right)^2\left(|\ln \rho|+\mathcal O(1)\right)\mathfrak b_4\Big[\|\pi(p)\|^2+\frac12\eta^2 \mathfrak Q(\tau_i, \tau_i) +\mathcal O(\eta^3)+o'_n(1)\Big]\\
&-\e\sum\limits_{i=1}^k\delta_i(\mathfrak c_n+o''_n(1))\\
& - \frac{\rho^2}{\eta ^{ 2}}\sum_{j<i}{\mathfrak d}_n\frac{1}{|K| }\frac{d_0^{ 2}}{|\tau_i-\tau_j|^{ 2}} +o\left(\frac{\rho^2}{\eta ^{ 2}}\right)
 \end{aligned}$$
and the claim follows because of the choice of $d_0$ in  \eqref{d0} and the fact that since $\eta=|\ln\rho|^{-\frac14}$ (see  \eqref{on})
$$o'_n(1)=\mathcal O(|\ln\rho|)=o(\eta^2).$$
 
If $5\leq n\leq 7$  by \eqref{deltaj}  
$$\begin{aligned}&\sum\limits_{i=1}^k \delta_i ^2\left[\mathfrak b_n\|\pi(p_i)\|^2+o'_n(1)\right] -\e\sum\limits_{i=1}^k\delta_i(\mathfrak c_n+o''_n(1))
\\
&-\sum_{j<i}{\mathfrak d}_n\frac{1}{|K|^{\frac{n-2}{2}}}\frac{\d_i^{\frac{n-2}{2}}\d_j^{\frac{n-2}{2}}}{\eta ^{n-2}}\frac{1}{|\tau_i-\tau_j|^{n-2}} 
+o\left(\frac{\d_i^{\frac{n-2}{2}}\d_j^{\frac{n-2}{2}}}{\eta ^{n-2}}\right)
\\ &=\sum\limits_{i=1}^k\e^{2}(d_0+\eta d_j)^2\mathfrak b_n\Big[\|\pi(p)\|^2+\frac12\eta^2 {\mathfrak Q}(p)(\tau_i, \tau_i) +\mathcal O(\eta^3)+o'_n(1)\Big]\\
&-\e\sum\limits_{i=1}^k\delta_i(\mathfrak c_n+o''_n(1))\\
& - \frac{\e^{n-2}}{\eta ^{n-2}}\sum_{j<i}{\mathfrak d}_n\frac{1}{|K|^{\frac{n-2}{2}}}\frac{d_0^{n-2}}{|\tau_i-\tau_j|^{n-2}} o\left(\frac{\e^{n-2}}{\eta ^{n-2}}\right)
 \end{aligned}$$
and the claim follows because of the choice of $d_0$ in  \eqref{d0} and the fact that since $\eta=\e^{n-4\over n}$ (see  \eqref{on})
$$o'_n(1)=\mathcal O(\e)=o(\eta^2)\ \hbox{if}\ n=6,7\ \hbox{and}\ o'_n(1)=\mathcal O(\e|\ln\e|)=o(\eta^2)\ \hbox{if}\ n=5.$$

We point out that in higher dimensions $n\geq8$ this is not true anymore.

\end{proof}


\begin{thebibliography}{}
\bibitem{ACA}
W. Abdelhedi, H. Chtioui, M.O. Ahmedou, {\it A Morse theoretical approach for the boundary mean curvature problem on $B^4$}, J. Funct. Anal. {\bf 254} (2008), no. 5, 1307--1341.

\bibitem{almaraz2011} 
S.M. Almaraz, \textit{A compactness theorem for scalar-flat metrics on manifolds with boundary}. Calc. Var. Part. Differ. Equ. {\bf 41} (2011), no. 3--4, 341--386.

\bibitem{A} S.M. Almaraz, \textit{An existence theorem of conformal scalar flat metrics on manifolds
with boundary}, Pacific J. Math. \textbf{248} (2010), no. 1, 1–22.

\bibitem{aubinbook} 
T. Aubin, \textit{Some Nonlinear Problems in Riemannian Geometry}, Springer Monographs in Mathematics. Springer, Berlin (1998).

\bibitem{AYM}
A. Ambrosetti, Y.Y. Li, A. Malchiodi, {\it On the Yamabe problem and the scalar curvature problems under boundary conditions}, Math. Ann. {\bf 322} (2002), no. 4, 667--699.

\bibitem{ABA} M. Ben Ayed, M.O. Ahmedou, \textit{Non simple blow ups for the Nirenberg problem on half spheres}, https://arxiv.org/abs/2108.08608

\bibitem{BEO1}
M. Ben Ayed, K. El Mehdi, M.O. Ahmedou, {\it Prescribing the scalar curvature under minimal boundary conditions on the half sphere},  Adv. Nonlinear Stud. {\bf 2} (2002), no. 2, 93--116.

\bibitem{BEO2}
M. Ben Ayed, K. El Mehdi, M.O. Ahmedou, {\it The scalar curvature problem on the four dimensional half sphere}, Calc. Var. Part. Differ. Equ. {\bf 22} (2005), no. 4, 465--482.



\bibitem{BC}
S. Brendle,  S.Y.S. Chen, \textit{An existence theorem for the Yamabe problem on manifolds
with boundary}, J. Eur. Math. Soc. \textbf{16} (2014), no. 5, 991–1016.

\bibitem{CXY}
A. Chang, X. Xu, P. Yang, {\it A perturbation result for prescribing mean curvature}, Math.
Ann. {\bf 310} (1998), no. 3, 473--496.

\bibitem{CHS}
X. Chen, P.T. Ho, L. Sun, {\it Liming Prescribed scalar curvature plus mean curvature flows
in compact manifolds with boundary of negative conformal invariant}, Ann. Global Anal.
Geom. {\bf 53} (2018), no. 1, 121--150.

\bibitem{crs} X. Chen, Y. Ruan, L. Sun, \textit{The Han-Li conjecture in constant scalar curvature and
constant boundary mean curvature problem on compact manifolds}, Adv. Math. \textbf{358} (2019), 56 pp.

\bibitem{C}
P. Cherrier, {\it Problemes de Neumann non lineaires sur les varietes riemanniennes}, J. of Funct. Anal. {\bf 57} (1984), 154--206.

\bibitem{cfs} 
M. Chipot, M. Fila, I. Shafrir, \textit{On the Solutions to some Elliptic Equations with Nonlinear Neumann Boundary Conditions}. Adv. in Diff. Eqs, 1, {\bf 1} (1996), 91--110.

\bibitem{CMR}
S. Cruz-Bl\'azquez, A. Malchiodi, D. Ruiz, {\it Conformal metrics with prescribed scalar and mean curvature}, to appear on J. f\" ur die Reine und Ang. Math.

\bibitem{DMA1}
Z. Djadli, A. Malchiodi, M.O. Ahmedou, {\it Prescribing Scalar and Boundary Mean Curvature on the Three Dimensional Half Sphere}, The Journal of Geom. Anal. {\bf 13} (2003),  no. 2. 255–289.

\bibitem{DMA2}
Z. Djadli, A. Malchiodi, M.O. Ahmedou, {\it The prescribed boundary mean curvature
problem on $B^4$}, J. Diff. Eqs {\bf 206} (2004), no. 2, 373--398.

\bibitem{dh} O. Druet, E. Hebey, \textit{Blow-up examples for second order elliptic PDEs of critical
Sobolev growth}, Trans. Amer. Math. Soc. \textbf{357} (5) (2005) 1915–1929.

\bibitem{E1} J. Escobar, {\it Conformal deformation of a Riemannian metric to a scalar at metric with constant mean curvature on the boundary}, Ann. Math. {\bf 136} (1992), 1--50.

\bibitem{E3} J. Escobar, {\it Conformal metrics with prescribed mean curvature on the boundary}, Calc. Var. {\bf 4} (1996), 559--592.

\bibitem{E2} J. Escobar, {\it The Yamabe problem on manifolds with boundary}, J. Differ. Geom., {\bf 35} (1992), 21--84.

\bibitem{GMP2019} M. Ghimenti, A.M. Micheletti, A. Pistoia, \textit{Blow-up phenomena for linearly perturbed Yamabe problem on manifolds with umbilic boundary}, J. Differential Equations \textbf{267} (2019), no. 1, 587–618.

\bibitem{gm} M. Ghimenti, A.M. Micheletti, A. Pistoia, \textit{Blow-up solutions concentrated along minimal submanifolds for some supercritical elliptic problems on Riemannian manifolds}, J. Fixed Point Theory Appl. 14 (2013), no. 2, 503–525.

\bibitem{GMP} M. Ghimenti, A.M. Micheletti, A. Pistoia, {\it Linear Perturbation of the Yamabe Problem on Manifolds with Boundary}, J. Geom. Anal. {\bf 28} (2018) 1315--1340.

\bibitem{HY2} Z.C. Han, Y.Y. Li, {\it The existence of conformal metrics with constant scalar curvature and constant boundary mean curvature}, Comm. Anal. Geom. {\bf 8} (2000), 809--869.

\bibitem{HY1} Z.C. Han, Y.Y. Li, {\it The Yamabe problem on manifolds with boundaries: existence and compactness results}, Duke Math. J. {\bf 99} (1999) 489--542.

\bibitem{L} Y.Y. Li, {\it The Nirenberg problem in a domain with boundary}, Topol. Meth. Nonlinear Anal. {\bf 6} (1995), no. 2, 309--329.

\bibitem{Li-CPAM96}Y.Y. Li, {\it Prescribing scalar curvature on $S^n$ and related problems. II. Existence and compactness,} Comm. Pure Appl. Math. \textbf{49} (1996), no. 6, 541–597.

\bibitem{YZ} Y.Y. Li, M. Zhu, {\it Uniqueness theorems through the method of moving spheres}, Duke Math. J. {\bf 80} (1995), no. 2, 383--417.

\bibitem{lz} Y.Y. Li, M. Zhu, \textit{Yamabe type equations on three-dimensional Riemannian manifolds},
Commun. Contemp. Math. \textbf{1} (1) (1999) 1–50.

\bibitem{M2} F.C. Marques, \textit{Conformal deformations to scalar-flat metrics with constant mean curvature on the boundary}, Comm. Anal. Geom. \textbf{15} (2007), no. 2, 381–405.

\bibitem{M1}  F.C. Marques, \textit{Existence results for the Yamabe problem on manifolds with boundary}, Indiana Univ. Math. J. (2005), 1599–1620.

\bibitem{MN} M. Mayer, C.B. Ndiaye, \textit{Barycenter technique and the Riemann mapping problem of Cherrier-Escobar}, J. Differential Geom. \textbf{107} (2017), no. 3, 519–560. 

\bibitem{MP2} A.M. Micheletti, A. Pistoia, {\it A generic result on Weyl tensor}, Top. Meth. Nonlin. Anal. {\bf 53} (2019), no. 1, 257--269.

\bibitem{MP1} A.M. Micheletti, A. Pistoia, {\it Generic properties of critical points of the Weyl tensor}, Adv. Nonlinear Stud. {\bf 17} (2017), no. 1, 99--109.


\bibitem{mpv} F. Morabito, A. Pistoia, G. Vaira, \textit{Towering phenomena for the Yamabe equation
on symmetric manifolds}, Potential Anal. \textbf{47} (1) (2017), 53–102.
 
\bibitem{pv} A. Pistoia, G. Vaira, \textit{Clustering phenomena for linear perturbation of the Yamabe equation}, Partial Differential Equations Arising from Physics and Geometry, London Mathematical Society Lecture Note Series \textbf{450} (2019), 311–331. Cambridge University Press, Cambridge.

\bibitem{pre} B. Premoselli, \textit{Towers of bubbles for Yamabe-type equations and for the Br\'ezis-Nirenberg problem in dimensions $n\geq7$}, J. Geom. Anal. \textbf{32} (2022), no. 3, paper no. 73, 65 pp.

\bibitem{punzo} F. Punzo, \textit{On well-posedness of semilinear parabolic and elliptic problems in the hyperbolic space}, J. Differential Equations \textbf{251} (2011), no. 7, 1972-1989.

\bibitem{rv} F. Robert, J. Vétois, \textit{Examples of non-isolated blow-up for perturbations of the
scalar curvature equation on non-locally conformally flat manifolds}, J. Differential Geom. 98 (2) (2014) 349–356.

\bibitem{tv} P.D. Thizy, J. Vétois,\textit{Positive clusters for smooth perturbations of a critical elliptic equation in dimensions four and five}, J. Funct. Anal. \textbf{275} (2018), no. 1, 170–195.

\bibitem{XZ} X. Xu, H. Zhang, {\it Conformal metrics on the unit ball with prescribed mean curvature},
Math. Ann. {\bf 365} (2016), no. 1--2, 497--557.

\end{thebibliography}
\end{document}